\documentclass[11pt,reqno]{amsart}
\usepackage{paralist,calc,amsfonts,amsthm,amscd,amsmath,amssymb,enumerate,mathrsfs,epsfig,graphicx,tikz,xcolor}
\usepackage[numeric,initials,nobysame]{amsrefs}
\usepackage{booktabs}
\numberwithin{equation}{section}

\newenvironment{wquote}{%
  \list{}{%
    \leftmargin0.65cm   %
    \rightmargin\leftmargin
  }
  \item\relax
}
{\endlist}

\newcommand{\bin}{\operatorname{Bin}}
\renewcommand{\restriction}{\mathord{\upharpoonright}}
\renewcommand{\epsilon}{\varepsilon}

\usepackage{color}
 \definecolor{refkey}{gray}{.5}
 \definecolor{labelkey}{gray}{.5}
\definecolor{light}{gray}{.9}

\newtheorem{theorem}{Theorem}[section]
\newtheorem{lemma}[theorem]{Lemma}
\newtheorem{proposition}[theorem]{Proposition}
\newtheorem{corollary}[theorem]{Corollary}
\newtheorem{remark}[theorem]{Remark}

\newtheorem{claim}[theorem]{Claim}
\newtheorem{definition}[theorem]{Definition}
\newtheorem{maintheorem}{Theorem}

\newtheorem*{question*}{Question}
\newtheorem{question}[theorem]{Question}
\newtheorem*{remark*}{Remark}
\newtheorem*{idefinition*}{Definition}

\newcommand{\R}{\mathbb R}
\newcommand{\Z}{\mathbb Z}

\newcommand{\cB}{\ensuremath{\mathcal B}}
\newcommand{\cC}{\ensuremath{\mathcal C}}
\newcommand{\cD}{\ensuremath{\mathcal D}}
\newcommand{\cE}{\ensuremath{\mathcal E}}

\newcommand{\cH}{\ensuremath{\mathcal H}}
\newcommand{\cI}{\ensuremath{\mathcal I}}

\newcommand{\cM}{\ensuremath{\mathcal M}}
\newcommand{\cN}{\ensuremath{\mathcal N}}

\newcommand{\cP}{\ensuremath{\mathcal P}}
\newcommand{\cQ}{\ensuremath{\mathcal Q}}
\newcommand{\cR}{\ensuremath{\mathcal R}}

\newcommand{\cZ}{\ensuremath{\mathcal Z}}

\newcommand{\bbN}{{\ensuremath{\mathbb N}} }

\newcommand{\bbP}{{\ensuremath{\mathbb P}} }

\newcommand{\bbR}{{\ensuremath{\mathbb R}} }

\newcommand{\bbZ}{{\ensuremath{\mathbb Z}} }

\newcommand{\sC}{{\ensuremath{\mathscr C}}}

\newcommand{\ind}{{\bf 1}}

\newcommand{\Dim}{\textsc{d} }
\newcommand{\pif}{\varpi}
\newcommand{\cZf}{\overline{\cZ}}
\newcommand{\mac}{\cM}

\renewcommand{\P}{\mathbb{P}}
\newcommand{\E}{\mathbb{E}}

\begin{document}
\title{Harmonic pinnacles in the Discrete Gaussian model}
\author[E. Lubetzky]{Eyal Lubetzky}
\address{E.\ Lubetzky\hfill\break
Microsoft Research\\ One Microsoft Way\\ Redmond, WA 98052-6399, USA.}
\email{eyal@microsoft.com}

 \author[F. Martinelli]{Fabio Martinelli}
 \address{F. Martinelli\hfill\break Dip.\ Matematica \& Fisica,
   Universit\`a Roma Tre, Largo S.\ Murialdo 1, 00146 Roma, Italy.}
\email{martin@mat.uniroma3.it}

\author[A. Sly]{Allan Sly}
\address{A. Sly\hfill\break
Department of Statistics\\
UC Berkeley\\
Berkeley, CA 94720, USA.}
\email{sly@stat.berkeley.edu}

\begin{abstract}
The 2\Dim Discrete Gaussian model gives each height
function $\eta : \Z^2\to\Z$ a probability proportional to $\exp(-\beta
  \mathcal{H}(\eta))$, where $\beta$ is the inverse-temperature and
  $\mathcal{H}(\eta) = \sum_{x\sim y}(\eta_x-\eta_y)^2$ sums over nearest-neighbor bonds.
We consider the model at large fixed $\beta$, where it is flat unlike its continuous analog (the Gaussian Free Field).

We first establish that the maximum height in an $L\times L$ box with 0 boundary conditions concentrates on
two integers $M,M+1$ with $M\sim \sqrt{(1/2\pi\beta)\log L\log\log L}$.
The key is a large deviation estimate for the height at the origin in $\Z^2$, 
dominated by ``harmonic pinnacles'', integer approximations of a harmonic variational problem.
Second, in this model conditioned on $\eta\geq 0$ (a floor), the average height rises, and in fact the height of almost all sites concentrates on
  levels $H,H+1$ where $H\sim M/\sqrt{2}$.

This in particular pins down the asymptotics, and corrects the order, in results of Bricmont, El-Mellouki and Fr{\"o}hlich (1986), where it was argued that the maximum and the height of the surface above a floor are both of order $\sqrt{\log L}$.

Finally, our methods extend to other classical surface models (e.g., restricted SOS), featuring connections to $p$-harmonic analysis and alternating sign matrices.
\end{abstract}


{
\baselineskip16pt\
\maketitle
}
\vspace{-0.32cm}

\section{Introduction}

The \emph{Discrete Gaussian} (DG) model on $\Lambda\subset \Z^2$ is a distribution over height functions $\eta$ on $\Z^2$ with $\Lambda\ni x \mapsto \eta_x \in \Z$ whereas $\eta_x=0$ for all $x\notin\Lambda$ (zero boundary conditions).
The probability of $\eta$ is penalized exponentially in the squared gradients of $\eta$, namely,
\begin{equation} \label{eq-dg-def}
\pi_\Lambda(\eta) = \frac{1}{\cZ_{\beta,\Lambda}} \exp \Big[ -\beta\, \cH(\eta) \Big]\quad\mbox{ for }\quad \cH(\eta)=\sum_{x\sim y} |\eta_x-\eta_y|^2\,,\end{equation}
where $\beta> 0$ is a parameter (the inverse-temperature), the notation $x\sim y$ denotes nearest-neighbor bonds in the lattice
and $\cZ_{\beta,\Lambda}$ is a normalizer (the partition function).
When it exists, the limit as $L\to\infty$ of $\pi_{\Lambda_L}$ for $\Lambda_L=\{1,\ldots,L\}^2$ will be denoted by $\pi$.

The DG model, dubbed so by Chui and Weeks in 1976 (cf.~\cites{CW76,WG79}), belongs to a family of random surface models introduced as far back as the 1950's to model the shape of crystals and the interfaces in 3-dimensional Ising ferromagnets.
It is the dual of the Villain XY model~\cite{Villain} and is also related by a duality transformation to the Coulomb gas model, hence its vital role in
understanding the \emph{Kosterlitz-Thouless phase-transition} that is anticipated in this family of models (see, e.g.,~\cites{Abraham,Swendsen} and the references therein).

The following basic features of the DG on $\Lambda_L = \{1,\ldots,L\}^2$ (and related models) were rigorously studied in breakthrough papers from the 1980's (\cites{BFL,BMF,BW,FS1,FS2,FS3};  see~\cite{Abraham}).
\begin{question}
  \label{q-height}
  What are the height fluctuations at the origin
  (or some given site), e.g., what is $\E [ \eta_0^2]$ and does it diverge with $L$?
  What is the maximum height $X_L=\max_x \eta_x$?
\end{question}
\begin{question}
  \label{q-floor}
How are these affected by conditioning that $\eta \geq 0$ (a floor constraint\footnote{This appears in
situations where the surface lies above a physical barrier, e.g., modeling the discrete interface between $+$/$-$ in 3-dimensional Ising with boundary conditions $+$ on one face and $-$ elsewhere.})?
\end{question}

Comparing the answers to these questions as the inverse-temperature $\beta$ varies reveals the \emph{roughening transition} that
the DG surface undergoes\footnote{This transition occurs only in dimension $d=2$:
the surface is rough for $d=1$ and rigid for $d\geq3$~\cite{BFL}.}
 at a critical $\beta_{\textsc{r}}$, suggested by numerical experiments to be about $0.665$: The surface transitions from being rigid (localized) at low temperatures (the height at any given site $x$ is bounded in probability)
to rough (delocalized) at high temperatures (that height typically diverges); see~\cites{Abraham,Weeks80}. In the latter regime, the DG model is believed to
be qualitatively similar to its analogue where the height functions are real-valued --- in which case the parameter $\beta$ scales out from~\eqref{eq-dg-def} and the model reduces to the Discrete Gaussian Free Field (DGFF).

Indeed, surface rigidity at large enough $\beta$ is known, as a Peierls argument~(\cites{BW,GMM}) then shows that $\E [\eta_0^2] =O(1)$.
That the surface is rough for small enough $\beta$ was established in the celebrated work of Fr\"ohlich and Spencer~\cites{FS1,FS2}, whence $\E [\eta_0^2] \asymp \log L$ (as is the case for the DGFF). The lower bound on the fluctuations (the main difficulty) was proved via an ingenious analysis of the Coulomb gas model, from which the results for the DG (and related models) followed using the aforementioned duality.

In their beautiful paper~\cite{BMF} from 1986, Bricmont, El-Mellouki and Fr\"ohlich provided a detailed examination of the behavior at low temperatures (the regime we focus on).
They showed that for  large $\beta$, conditioning on $\eta \geq 0$ induces an \emph{entropic repulsion} phenomenon: though in the rigid regime $\beta>\beta_{\textsc{r}}$, the surface rises and the expected average height $\E\big[\frac1{|\Lambda|}\sum_x \eta_x \,\big|\, \eta\geq 0\big]$ diverges as $L\to\infty$. As Abraham wrote in~\cite{Abraham}*{p59},
\begin{wquote}
  ``The origin of this apparently paradoxical
result is that `spikes' grow downwards from the surface; if any spike touches
the surface, such a configuration does not contribute to the entropy. This
drives the surface away `to infinity'.''
\end{wquote}
\vspace{-0.03cm}
 More precisely, it was stated in~\cite{BMF} (Thm.~4.1, Thm.~3.2 and their proofs; cf.~\cite{Abraham}) that
 \begin{equation}
    \label{eq-bef}
     \mbox{$\E\left[\frac{1}{|\Lambda|}\sum_x  \eta_x \;\big|\; \eta\geq 0\right]$}\asymp \sqrt{\beta^{-1}\log L}
    \quad\mbox{ and }\quad
    \E[X_L] \asymp \sqrt{\beta^{-1}\log L}\,,
  \end{equation}
 where $X_L$ is the maximum of the DG surface.
 That is, the average height rises until it become comparable with the maximum of the standard (unconstrained) DG surface.
(Analogous bounds were obtained for the related \emph{Absolute-Value Solid-On-Solid} model, in which $|\eta_x-\eta_y|$ replaces $|\eta_x-\eta_y|^2$ in~\eqref{eq-dg-def}, whence these bounds turn into $\beta^{-1}\log L$.)

To gain some intuition for this result, first consider the maximum: raising a given site to height $h$ via a single spike incurs a cost of $\exp(-c \beta h^2)$
(since its neighbors are typically at height $O(1)$ in the rigid regime), explaining one side of the bound on $\E[X_L]$.
The typical value of the maximum is also an upper bound on the average height when conditioning on $\eta\geq 0$ (at that surface height the floor at 0 is no longer noticeable);
the matching lower bound was quite more involved, using
Pirogov-Sina\"{\i} theory (see~\cite{Sinai}).

It is worthwhile noting that for the DGFF (associated to the high temperature DG), it was shown by Bolthausen, Deuschel and Giacomin~\cite{BDG} that
the maximum concentrates on
$2\sqrt{2/\pi}\log L$,
whereas conditioning on $\eta\geq 0$ raises the height of most sites\footnote{This result of~\cite{BDG} applies to sites at distance at least $\delta L$ from the boundary for some positive $\delta>0$.}
to concentrate on the same $2\sqrt{2/\pi}\log L$
(cf.~\cite{BDZ} for analogous entropic repulsion results for the DGFF in dimensions $d\geq 3$).
That is, the surface rises to the asymptotic level of the unconstrained maximum/minimum (at which point the floor becomes irrelevant).
In view of~\eqref{eq-bef}, it is natural to ask if this is also the case for the low temperature DG.

Specifically, one can ask for asymptotic bounds refining those of~\cite{BMF} (Eq.~\eqref{eq-bef} above), as well as for tight concentration estimates.
Significant progress in this direction was recently obtained~\cites{CLMST1,CLMST2,CRASS} for the related Absolute-Value Solid-On-Solid (SOS) model. There it was shown, amid detailed results on the ensemble of level lines and its scaling limit, that the maximum concentrates on $\tfrac1{2\beta}\log L$ while the typical height above a floor is asymptotically a \emph{half} of that. Supporting many of those arguments was the fact that, in the SOS model, the contribution of the $h$-level lines to the probability of a configuration $\eta$ is only a function of the $(h-1)$-level and $(h+1)$-level lines (enabling an iterative analysis of the surface, one level line at a time). This is unfortunately absent in the DG model due to the quadratic terms $|\eta_x-\eta_y|^2$, calling for additional ideas.

\subsection{Maximum in a box and large deviations in infinite volume}

\begin{figure}
\begin{center}
\vspace{-0.1in}
\includegraphics[width=0.675\textwidth]{dgpinn1}
\hspace{-0.13in}
\raisebox{0\height}{
\includegraphics[width=0.3\textwidth]{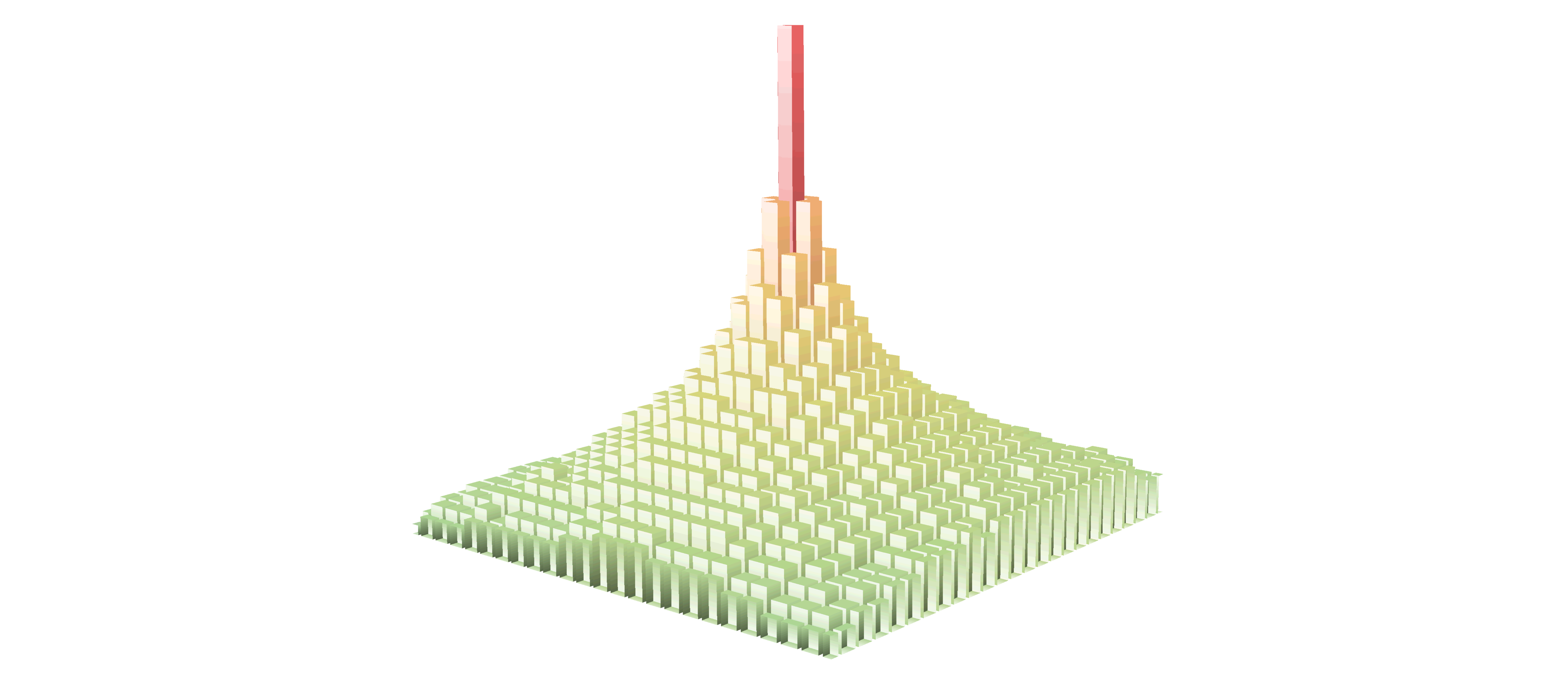}
}
\end{center}
\vspace{-0.1in}
\caption{The low temperature Discrete Gaussian surface conditioned on positive and negative large deviations (magnified on the right).}
\label{fig:dg}
\vspace{-0.05in}
\end{figure}

Our first main result is a 2-point concentration estimate for the maximum of the DG model on a box. (In what follows, we write $f\sim g$ to denote that $\lim_{L\to\infty}f/g = 1$.)
\begin{maintheorem}\label{mainthm:max}
Fix $\beta > 0$ large enough and let $X_L$ be the maximum of the DG model on an $L\times L$ box in $\Z^2$ at inverse-temperature $\beta$. Then there exists some $M=M(L)$ with
 \begin{equation}
   \label{eq:E[XL]}
   M \sim \sqrt{(1/2\pi\beta) \log L \log\log L}
 \end{equation}
such that $X_L \in \{ M,M+1\}$ with probability going to 1 as $L\to\infty$.
\end{maintheorem}
\noindent(The error probability in the above statement can be taken to be $\exp[-(\log L)^{1/2-o(1)}]$.)
\begin{remark}\label{rem:1-pt-max}
For every $L$ except for a subset of logarithmic density 0 of the integers,
the maximum $X_L$ concentrates on a single integer $M=M(L)$ with high probability.
\end{remark}

Interestingly, upon comparing the estimate~\eqref{eq:E[XL]} with the previous bounds (Eq.~\eqref{eq-bef}) we see that they disagree on the order of the maximum by a factor of $\sqrt{\log\log L}$
 (similarly missing also from the result of~\cite{BMF} on the average height above a floor; see our Theorem~\ref{mainthm:floor-shape}). This is due to the typical type of large deviations (LD) in the surface: instead of forming spikes of height $h$, it is preferable (by a $\log h$ factor) for the DG model to create ``harmonic pinnacles,'' integer approximations of a harmonic variational problem (see Fig.~\ref{fig:dg}),
as seen in the next LD result on $\pi$, the infinite-volume DG measure:
  \begin{equation}
    \label{eq:inf-vol-ldp}
    \pi(\eta_0 \geq h) = \exp\left[-\left(2\pi\beta+o(1)\right)\frac{h^2}{\log h}\right]\,.
  \end{equation}
This estimate (see Theorem~\ref{th:p(h)} in \S\ref{sec:ldev}) will be the main ingredient in proving Theorem~\ref{mainthm:max}.

The integer $M(L)$ such that the maximum $X_L$ belongs to $\{M,M+1\}$ w.h.p.\ (and moreover $X_L=M$ w.h.p.\ for most $L$'s)
is explicitly given as the maximum integer such that $\pi(\eta_0 \geq M) \geq L^{-2}\log^5 L$  (see~\eqref{eq-M-def} in \S\ref{sec:ldev}).
Comparing~\eqref{eq:E[XL]} to~\eqref{eq:inf-vol-ldp} we see that $X_L$ behaves as if the surface consisted of i.i.d.\ variables with law $\pi(\eta_0\in\cdot)$.

For an explanation of how the extra $\log h$ factor arises in Eq.~\eqref{eq:inf-vol-ldp}, see \S\ref{sec:intro-pfs} below.
It is worthwhile noting a separate consequence of this extra factor vs.\ the results in~\cite{BMF}:
\begin{remark}
The convergence of free energy $\psi_\ell = \log \cZ_{\beta,\Lambda}$ on a slab $\Lambda=[-\ell,\ell]\times \Z^2$ to $\psi_\infty$, the free energy of the infinite-volume DG, satisfies
  \[ \exp(-c_1\ell^2/\log\ell) \leq |\psi_\ell-\psi_\infty| \leq \exp(-c_2\ell^2/\log\ell)\]
  for constants $c_1(\beta)>c_2(\beta)>0$
(in contrast with the convergence rate of $\exp(-c \ell^2)$ that was stated in~\cite{BMF}*{Theorem 3.2}; see also~\cite{Abraham}*{p67} for a discussion on that result).
\end{remark}
\subsection{Entropic repulsion in the presence of a floor}
We now address Question~\ref{q-floor} regarding the conditioning on $\eta \geq 0$ (a floor at 0). Here the analysis is considerably delicate, and not only do we show a 2-point concentration for the typical height about $H\sim \sqrt{(4\pi\beta)^{-1}\log L \log\log L}$ (recall that the lower bound of order $\sqrt{\log L}$ due to~\cite{BMF}, which was correct albeit not sharp, relied on the highly nontrivial Pirogov-Sina\"{\i} theory), but furthermore we describe the \emph{shape} of the surface in terms of its level lines.

Deferring formal definitions to \S\ref{sec:floor}, the $h$-level lines are the closed loops that separate $\{x:\eta_x\geq h\}$ and $\{x:\eta_x<h\}$, and a loop is \emph{macroscopic} if it is of length at least $\log^2 L$.
The DG trivially exhibits local fluctuations (e.g., see Eq.~\eqref{eq:inf-vol-ldp}),
which we can filter out in our study of the surface shape by restricting our attention to the macroscopic loops\footnote{one may set the cutoff for macroscopic loops at $C\log L $ for a large $C(\beta)$ without affecting the proofs.}.

Beyond those local fluctuations (occurring at an $\epsilon_\beta$-fraction of the sites for $\epsilon_\beta$ fixed),
we show that the DG surface is typically a \emph{plateau} at an asymptotic height $ (1/\sqrt{2})M$:

\begin{maintheorem}
  \label{mainthm:floor-shape}
Fix $\beta > 0$ sufficiently large, and consider the Discrete Gaussian model on an $L\times L$ box in $\Z^2$ at inverse-temperature $\beta$ with a floor at 0.
Then there exists some $H=H(L)$ with $H \sim \sqrt{(1/4\pi \beta)\log L\log\log L}$ such that w.h.p.\
\begin{equation}\label{eq-height-conc}
 \#\big\{v : \eta_v \in \{H,H+1\}\big\} \geq (1-\epsilon_\beta) L^2\,,
\end{equation}
where $\epsilon_\beta$ can be made arbitrarily small as $\beta$ increases.
Furthermore, w.h.p., 
\begin{compactenum}[(i)]
  \item at each height $1 \leq h \leq H-1$ there is one macroscopic loop with area $(1-o(1))L^2$;
  \item at height $H$ there is one macroscopic loop with area at least $(1-\epsilon_\beta)L^2$;
  \item there is no macroscopic loop at height $H+2$ nor any macroscopic negative loop.
\end{compactenum}
\end{maintheorem}
In a sense, this plateau behaves as a raised version of the unconstrained surface, e.g., the probability that $\eta_x \geq H+h$ will be approximately $\pi(\eta_0\geq h)$ and similarly for $\eta_x \leq H-h$ (until capped at the floor).
The integer $H$ is explicitly given by
\begin{equation}
  \label{eq-H-def}
  H = H(L) = \max\left\{ h : \pi(\eta_0 \geq h) \geq 5\beta/L\right\}\,.
\end{equation}

\begin{remark}\label{rem:1-pt-level}
For every $L$ except for a subset of logarithmic density 0 of the integers almost all sites are at level $H$, namely $\#\{v:\eta_v = H\} \geq (1-\epsilon_\beta)L^2$ w.h.p.
Furthermore, for all the non-exceptional values of $L$ we have that the macroscopic loop at height $H$ has area $(1-o(1))L^2$, and there is no macroscopic loop at height $H+1$.
\end{remark}
By combining Theorem~\ref{mainthm:floor-shape} (and the comment following it) with Theorem~\ref{mainthm:max} we get that conditioning on $\eta\geq 0$
tends to increase the maximum $X_L$ by a factor of $1+1/\sqrt{2}+o(1)$.

\begin{maintheorem}
  \label{mainthm:floor-max}
  Fix $\beta > 0$ large enough and let $X_L^* $ be the maximum of the DG on an $L\times L$ box at inverse-temperature $\beta$
  with a floor at 0. There exists $M^*=M^*(L)$ with
 \begin{equation}
   \label{eq:E[XL*]}
   M^* \sim \frac{1+\sqrt{2}}{2\sqrt{\pi\beta}}\sqrt{ \log L \log\log L}
 \end{equation}
 such that $X_L^*\in\{M^*,M^*+1,M^*+2\}$ with probability going to 1 as $L\to\infty$.
\end{maintheorem}

\begin{figure}
\vspace{-0.3cm}
\begin{center}
\includegraphics[width=0.35\textwidth]{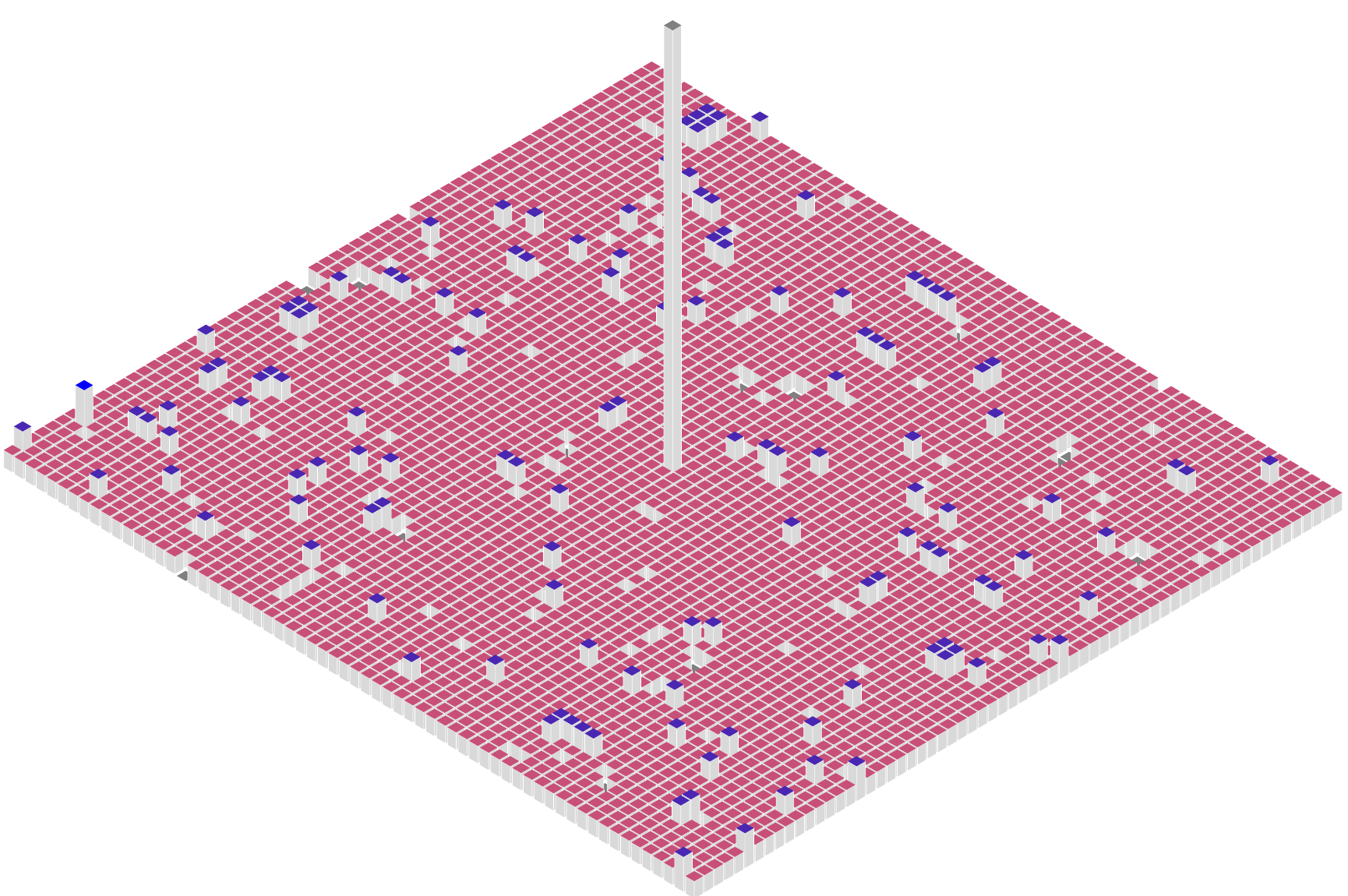}
\hspace{-0.6in}
\raisebox{0.2\height}{
\includegraphics[width=0.35\textwidth]{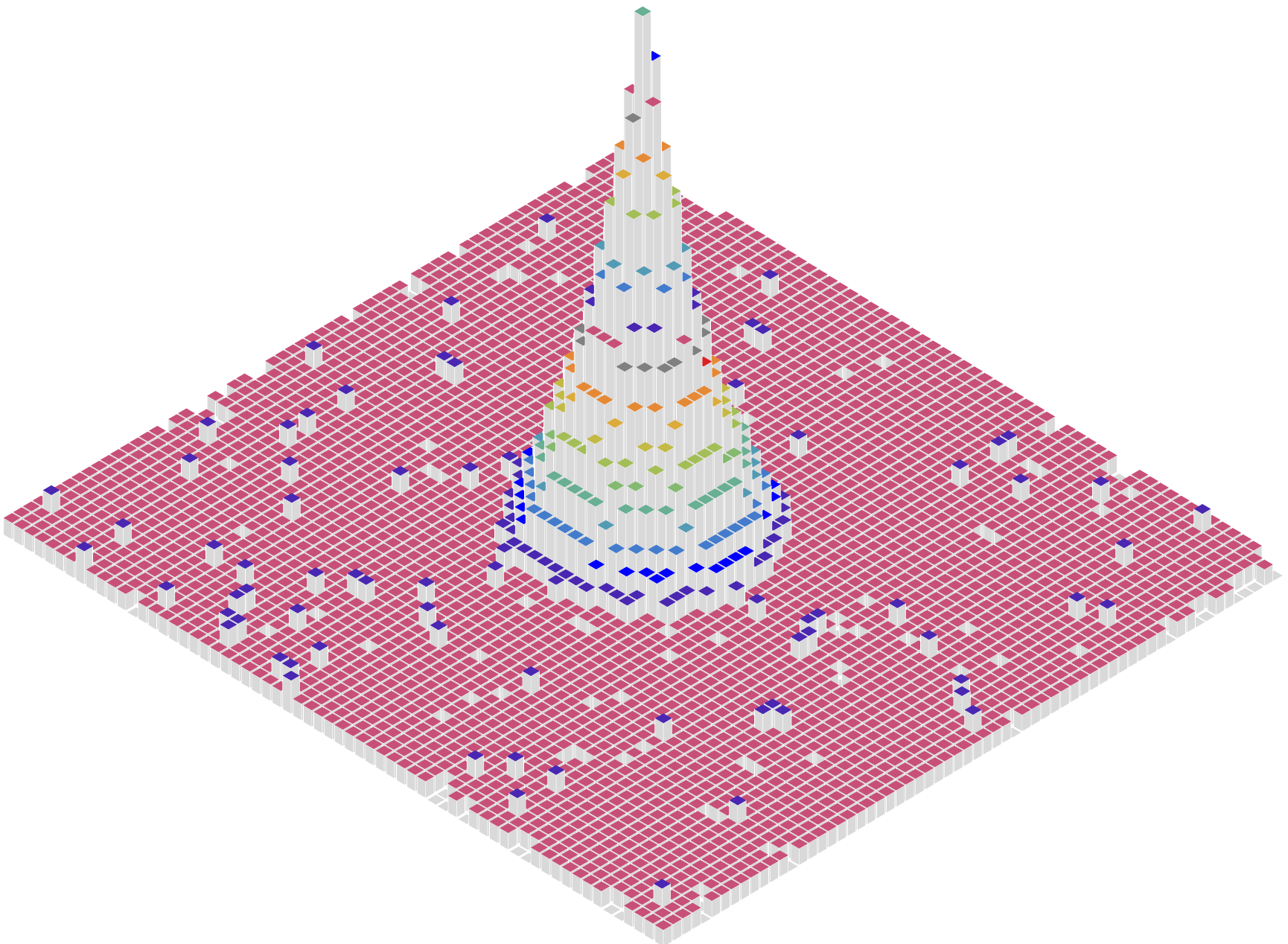}
}
\hspace{-0.45in}
\includegraphics[width=0.35\textwidth]{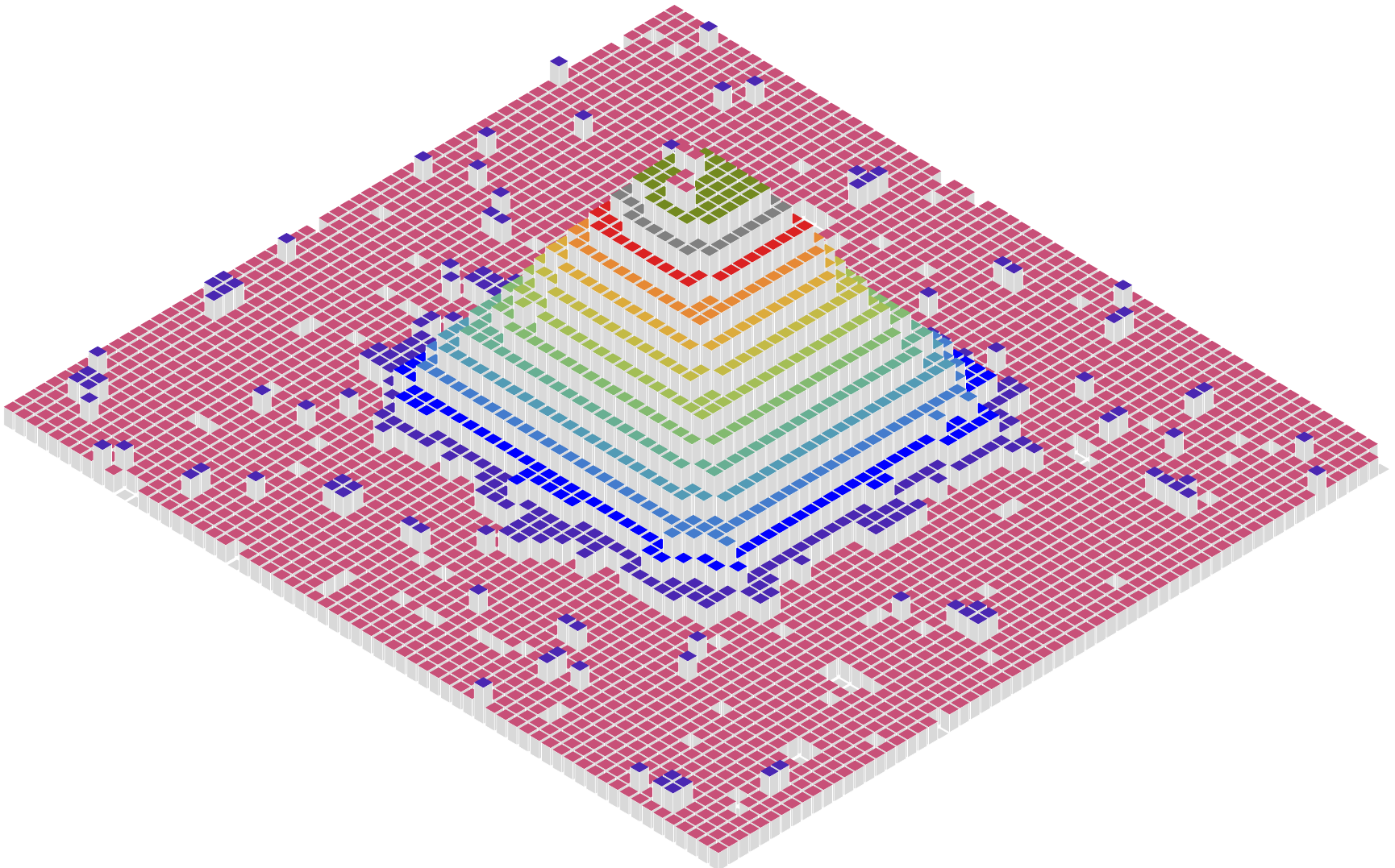}
\end{center}
\vspace{-0.15in}
\caption{Large deviations of the height at the origin: from left to right, SOS (spike), DG (harmonic pinnacle) and RSOS (pyramid).}
\label{fig:ld}
\vspace{-0.2cm}
\end{figure}

\subsection{Generalizations to random surfaces with $p$-Hamiltonians}
 Our arguments extend to the family of random surface models in which
 the Hamiltonian $\cH(\eta)$ in~\eqref{eq-dg-def} is replaced by $\sum_{x\sim y} |\eta_x-\eta_y|^p$ for any $p\in(1,\infty]$.
 (The case $p=\infty$, i.e., $|\eta_x-\eta_y|\in\{0,\pm1\}$ for all $x\sim y$, is the \emph{restricted SOS} (RSOS) model.)
 The next table summarizes our analogous results for general $p$ (see Fig.~\ref{fig:ld} for the LD comparison of $p=1,2,\infty$).
{
\tabcolsep=0.10cm
\begin{table}[h]
\hspace{-0.15cm}
\begin{tabular}{ccccccc}
\toprule%
Model & Large deviation  & \multicolumn{2}{c}{\qquad \quad Maximum} &  \multicolumn{2}{c}{\; ~~ Height above floor} & Ref. \\
   & {\footnotesize $-\log\pi(\eta_0 \geq h)$} &  {\footnotesize center ($M$)} &{\footnotesize window} &  {\footnotesize center ($H$)}&{\footnotesize window} &     \\
\midrule
\noalign{\medskip}
\begin{tabular}
  {c} $p=1$ \\ {\small (SOS)}
\end{tabular}
& $4\beta h + \epsilon_\beta$
& $ \frac1{2\beta}\log L $
& $ O(1)$
& $ \lceil\frac1{4\beta}\log L\rceil $
& $\pm 1$
& \cites{CLMST1,CLMST2}\\
\noalign{\medskip}
\midrule[0.25pt]
\noalign{\medskip}
$1<p<2$
& $ \left(c_p \beta +o(1)\right) h^p $
& $ \big(\frac{2+o(1)}{c_p\beta}\log L\big)^{\frac1p} $
& $ \pm1 $
& $  \big(\frac{1+o(1)}2\big)^{\frac1p} M $
& $\pm 1$
& \S\ref{sec:1<p<2}\\
\noalign{\medskip}
\midrule[0.25pt]
\noalign{\medskip}
\begin{tabular}
  {c} $p=2$ \\ {\small (DG)}
\end{tabular} & $\big(2\pi\beta+o(1)\big)\frac{h^2}{\log h}$
& {\small $ \sqrt{\frac{1+o(1)}{2\pi\beta}\log L \log\log L} $}
& $ \pm 1$
& $ \frac{1+o(1)}{\sqrt2}M$
& $\pm 1$
& \S\ref{sec:ldev}--\ref{sec:floor}\\
\noalign{\medskip}
\midrule[0.25pt]
\noalign{\medskip}
$2<p<\infty$
& $\asymp \beta h^2 $
& $ \asymp\sqrt{\frac{1}{\beta}\log L} $
& $\pm 1$
& $ \frac{1+o(1)}{\sqrt{2}} M $
& $\pm 1$
& \S\ref{sec:2<p<inf}\\
\noalign{\medskip}
\midrule[0.25pt]
\noalign{\medskip}
\begin{tabular}
  {c} $p=\infty$ \\ {\small (RSOS)}
\end{tabular}
& 
$\left(4\beta + 2\log\frac{27}{16}+\epsilon_\beta\right) h^2$
& $(1\pm\epsilon_\beta)  \sqrt{\frac2{c_\infty}\log L}$
& $\pm 1$
& $ \frac{1+o(1)}{\sqrt2}M$
& $\pm 1$
& \S\ref{sec:RSOS}\\
\noalign{\medskip}
\midrule[0.25pt]
\bottomrule
\end{tabular}
\end{table}
}

As the above table shows, while the values of $M$ and $H$ --- the centers of the maximum and the height of the plateau conditioned on $\eta\geq 0$, respectively --- vary with $p$, the qualitative behavior of a 2-point concentration for the two corresponding variables, as well as having the ratio $H/M$ converge to some fixed $c_p\in(0,1)$ as $L\to\infty$, is universal.

Thanks to the generality of the framework for proving Theorem~\ref{mainthm:floor-shape}, all that is needed to obtain analogous results for any $p>1$ is to estimate the large deviation problem at the origin under the infinite-volume measure $\pi$ (analogous to~\eqref{eq:inf-vol-ldp}), as well as the 2-point large deviation problem (e.g., estimate  $\pi(\eta_z=h\mid\eta_0=h)$ for $z$ near the origin).
These, in turn, reduce to variational problems with connections to $p$-harmonic analysis (for $1<p<2$) and Alternating Sign Matrices (ASMs) (for $p=\infty$, see Fig.~\ref{fig:rsos1}).

\begin{figure}
\begin{center}
\raisebox{0.22\height}{\includegraphics[width=0.25\textwidth]{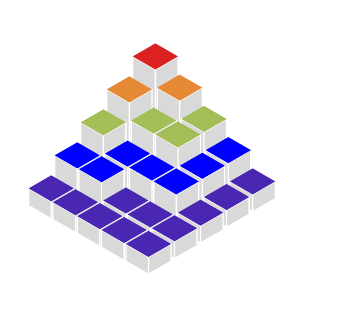}}
\hspace{-0.3in}
\includegraphics[width=0.4\textwidth]{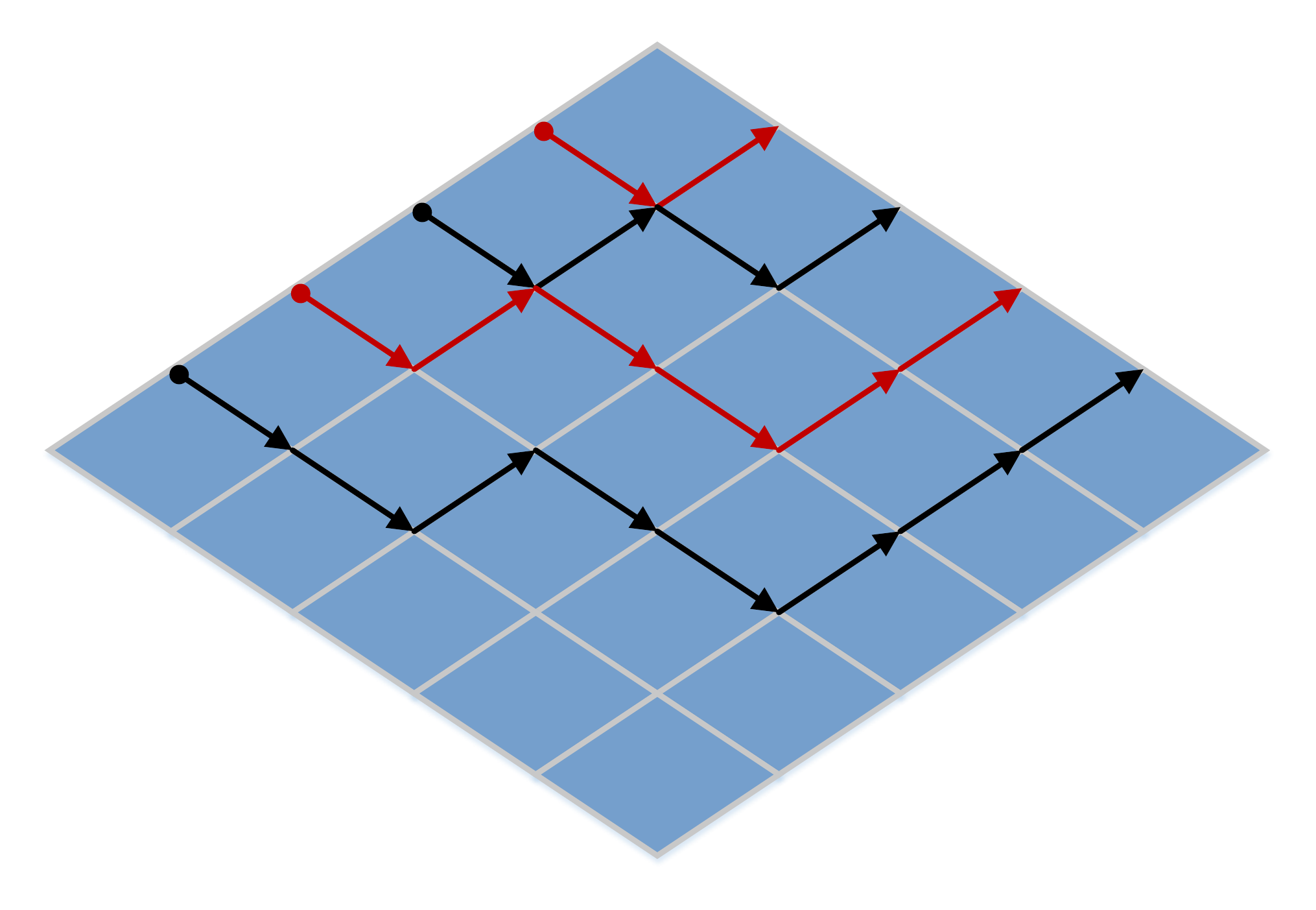}
\hspace{-0.1in}
\raisebox{1.8\height}{\color{labelkey}
$\begin{pmatrix}
  0 & 1 &0 &0 \\
  0 & 0 & 1 &0 \\
   1 & -1 &0 &1 \\
    0 & 1 & 0 &0
\end{pmatrix}$
}
\end{center}
\vspace{-0.1in}
\caption{Correspondence between the RSOS optimal-energy surfaces, edge-disjoint walks, and (via the six-vertex model) ASMs.}
\label{fig:rsos1}
\end{figure}

\subsection{Ideas from the proofs for the DG}\label{sec:intro-pfs}
 The following heuristics demonstrates the extra $\log h$ factor in the LD result on $\pi$.
  Suppose first that the height functions were real-valued
on the region $B_r$ --- the discrete ball of radius $r$ in $\Z^2$ centered at the origin --- for some large integer $r$.
Denoting these by $\varphi: B_r\mapsto
\bbR$, the LD problem is to find
\begin{equation}
\label{eq:variation}
I_r(h):=\inf \{\cD_{B_r}(\varphi):\ \varphi\restriction_{B_r^c}=0,\ \varphi_0=h\}\quad\mbox{ where }\quad\cD_{B_r}(\varphi)=\sum_{x\sim y }(\varphi_x-\varphi_y)^2\,;
\end{equation}
its minimizer $\phi$ is well-known to be the solution of the Dirichlet problem on $B_r\setminus\{0\}$,
\begin{align*}
(\Delta\phi)\restriction_{B_r\setminus \{0\}}=0\,, \quad \phi\restriction_{\partial B_r}=0\,,
\quad\phi_0 =h\,,
  \end{align*}
in which $\Delta$ denotes the discrete Laplacian $\Delta \phi_x= \frac 14
\sum_{y\sim x} (\phi_y -\phi_x)$.
Therefore, $\phi$ has the explicit representation $
\phi_x= h\bbP_x(\tau_0<\tau_{\partial B_r})$,
where $\tau_0$ and $\tau_{\partial B_r}$ are the hitting times of the
origin and of $\partial B_r$, respectively, for the simple random walk started at $x$. In particular, by well-known estimates on the Green's function (see~\cite{Lawler}*{Prop.~1.6.7}),
\[
\phi_x = \left(1- \frac{\log |x|+O(1)}{\log
    r}\right)h\quad\mbox{ for all $x$ with $1<|x|<r$}\,.
\]

Now let us return to the setting of integer values $\eta:B_r\to\Z$, and for the moment suppose that the real-valued solution $\phi$ can be \emph{rounded} without any loss in the cost function. Still, $\phi$ sends (smaller and smaller) mass to $\infty$, while the integer-valued solution analogous to~\eqref{eq:variation} must be truncated to 0 once it drops below $1$.
Taking $|x|=r-1$ (near $\partial B_r$) and solving $\phi_x \asymp 1$ using the last display gives $r \sim h/\log h$.

Two observations at this point complete the heuristical explanation of~\eqref{eq:inf-vol-ldp}:
\begin{compactenum}[(i)]
  \item the real-valued solution for $r\asymp \frac{h}{\log h}$ is $I_r(h)\sim 2 \pi\beta \frac{h^2}{\log h}$ (our final LD estimate);
\item the volume of $B_r$ is $O(h^2/\log^2 h)$, and so the rounding cost (even when charging $2\beta$ per bond in $B_r$) is negligible in comparison  with the main term $I_r(h)$.
\end{compactenum}
The essence of proving Theorem~\ref{mainthm:max} is to rigorously establish that the solution to the integer-valued variational problem is indeed of this form, e.g., that is supported on a ball of radius $O(h/\log h)$, etc.
To that end, we write this solution as $\phi+\sigma$ and bound the effect of the residue $\sigma$
using the harmonic properties of the real-valued solution $\phi$.

\smallskip

One of the main keys for proving Theorem~\ref{mainthm:floor-shape} is a building block (Proposition~\ref{prop:blackbox}) that allows us to say that, if $h$ and $\ell$ are two integers \emph{satisfying a specific condition in terms of the LD rate function for the DG}, then a square of side-length $\ell$ with boundary conditions $h-1$ will contain, with very high probability, an $h$-level line loop filling almost its entire area.
Namely, the condition that $h,\ell$ must satisfy is that
\[ 4\beta+2  \leq  \pi\left(\eta_0\geq h\right) \ell \leq  4\beta+4 \,,\]
where this relation embodies the entropic repulsion tradeoff between increasing the height (the large deviation term) and increasing the area (the side-length, governing the area via an isoperimetric inequality, whence the factor 4 that appears here).

Our strategy is then to iteratively ``grow the surface'', assuming inductively that the $(h-1)$-level line fills almost the entire square and establishing the next level for each $h=1,\ldots,H$.
In order to raise the surface height from $h-1$ to $h$, we consider a small enough $\ell\times\ell$ tile for which the above condition would hold, and apply the above result to overlapping tilings of the $L \times L$ square $\Lambda_L$ using such tiles; these lead to a single loop that fills all but a margin of at most $\ell$ from the boundary of $\Lambda_L$.

That the loops at levels $1,\ldots,H-1$ have area $(1-o(1))L^2$ is explained by the fact that the prescribed $\ell\times \ell$ tile used to establish levels $h=1,\ldots,H-1$ satisfy $\ell=o(L)$, and so it asymptotically fills $\Lambda_L$. At the final level $H$ this may no longer be the case, and indeed there should be values of $L$ where the $H$-level line will indeed erode linearly away from the corners, forming a Wulff shape as in the case of the SOS model~\cite{CLMST2}.

\subsection{Open problems}
The universality of the family of random surface models for $p\in[1,\infty]$, as discussed above, suggests that the DG should possess many of the features of the SOS surface. Following the recent understanding in~\cite{CLMST2}, it is plausible that, for the values of $L$ where the $H$-level line asymptotically fills the square, it would
feature $L^{1/3+o(1)}$ fluctuations from the boundary of the box; for the exceptional values of $L$, the scaling-limit of the $H$-level line should be the result of a tiling of a properly rescaled Wulff-shape, whence it would overlap with the boundary near the center-sides while featuring rounded corners; one would expect $L^{1/3+o(1)}$ fluctuations of the $H$-level lines along the straight parts of this limit, and $L^{1/2+o(1)}$
fluctuations along the corners.

\subsection{Organization}
In~\S\ref{sec:ldev} we study the maximum of the DG on a box through the related LD question in infinite-volume, proving Theorem~\ref{mainthm:max}. The shape of the DG above a floor, as well as the entropic repulsion effect on the maximum, is analyzed in~\S\ref{sec:floor}, where we prove Theorems~\ref{mainthm:floor-shape}--\ref{mainthm:floor-max}. Finally, the extensions of these results to the family of random surface models where the Hamiltonian features $p$-powers of the gradients appear in~\S\ref{sec:other-p}.

\section{Large deviations and Proof of Theorem~\ref{mainthm:max}}\label{sec:ldev}
Our main result in this section is the following LD estimate. Throughout this section, we let $\partial B_r$ denote the external boundary of $B_r$ (i.e., $x\notin B_r$ with $x\sim y$ for some $y\in B_r$).
\begin{theorem}
\label{th:p(h)}
Fix $\beta$ large enough and let $\Phi(h):=I_{h/\log h}(h)$ with $I_r(h)$ as in~\eqref{eq:variation}. There exist constants $c_0,c_1,c_2$ such
that the following hold for any $h\in \bbN$ and $z\in \Z^2$:
\begin{align}
e^{-c_0 \beta h/\log h}\le \frac{\pi(\eta_0=h)}{\pi(\eta_0=h-1)}\le
e^{-c_1\beta h/\log h}\,,\label{eq-ratio-p(h)/p(h-1)}\\
\pi(\eta_0=h)= \exp\left[-\beta \Phi(h) + O\left(h^2/\log^2 h\right)\right]\,,
\label{eq-p(h)}\\
\pi(\eta_z=h\mid \eta_0=h)\le e^{-c_2 h^2/(\log
  h)^2 }\,.
  \label{eq-p(h,h)}
\end{align}
\end{theorem}
As we will next see, Eq.~\eqref{eq-p(h)} above translates into
\begin{equation}
   \label{eq-p(h)-2}
   \pi(\eta_0 = h) = \exp\left[- 2 \pi\beta \frac{h^2}{\log h} + O\left(\frac{h^2}{\log^2 h}\right)\right]
 \end{equation}
by substituting the value of $\Phi(h)$ as given by the following simple lemma.
\begin{lemma}
\label{lem:dirichlet}
Set $\kappa = \gamma + \frac32 \log 2$ where $\gamma$ is Euler's constant. For any $r>0$
\[
I_r(h)= \big(2\pi  +O(1/r)\big)\frac{h^2}{\log r
  + \kappa}\,.
\]
In particular, $I_r(h) \sim 2\pi h^2/\log h$ for any choice of $r\asymp h/\log h$.
\end{lemma}
\begin{proof}
Let $S_t$ denote simple random walk in $\Z^2$ and
write $\tau_\partial = \min\{t : |S_t| \geq r\}$. By the Hitting-Time Identity for electric networks (see, e.g.,~\cite{LP}*{Proposition~2.20} as well as~\cite{LP}*{\S2.1 and \S2.4} for further background),
\begin{equation}
  \label{eq-Ir-formula}
  I_r(h) = 4 h^2 \frac{\sum_x \P_x(\tau_0 < \tau_\partial)}{\E_0\tau_\partial}\,.
\end{equation}
(By Dirichlet's Principle, the effective conductance $\sC(0\leftrightarrow \partial B_r)$ in the network with unit conductances is precisely $I_r(h) / h^2$.
The Hitting-Time Identity, combined with Ohm's law, implies that $\E_0 \tau_\partial=4\sC(0\leftrightarrow\partial B_r)^{-1} \sum \P_x(\tau_0 < \tau_\partial)$,
with the factor $4$ due to the transition probability of simple random walk along an edge, and~\eqref{eq-Ir-formula} follows.)
For the denominator in~\eqref{eq-Ir-formula}, since $|S_n|^2 - n$ is a martingale in $\Z^2$, Optional Stopping (and the fact that $\tau_\partial$ is a.s.\ finite) implies that
\[ \E_0\tau_\partial = r^2 + O(r)\,,\]
where the $O(r)$ term is due to the fact that $r \leq |S_{\tau_\partial}| < r+1$.

As for the numerator in~\eqref{eq-Ir-formula}, we first approximate the sum by $\int_{1\leq |x|\leq r} \P_x(\tau_0< \tau_\partial)$ at the cost of a factor of $1+O(1/r)$.
Next , let $a(x) = \lim_{n\to\infty}(G_n(0)-G_n(x))$ denote the potential kernel, where $G_n(x)$ is the Green's function.
It is known (see, e.g.,~\cite{Lawler}*{\S1.6}) that
\[ a(x) = \frac2{\pi}\left(\log|x| + \gamma + \frac32\log2\right) + O\left(1/|x|^2\right)\,,\]
where $\gamma$ is Euler's constant, and that $a(S_t)$ is a martingale.
Thus, by Optional Stopping,
\begin{equation}
  \label{eq-phi-x-formula}
  \P_x(\tau_\partial < \tau_0) = \frac{\log|x|+\kappa+O\left(1/|x|^2\right)}{\log r + \kappa + O(1/r)}\,,
\end{equation}
where the $O(1/r)$ in the denominator (vs.\ the $O(1/r^2)$ error in estimating the potential kernel) is again since at time $\tau_\partial$ we can only assert that $r \leq |S_t| < r+1$ in $\Z^2$ (translating into an $O(1/r)$ additive error through the series expansion of $\log r$).
Therefore,
\begin{multline*}
\int\P_x(\tau_0< \tau_\partial)dx =  2\pi \int_1^r \left(1 - \frac{\log x + \kappa + O(x^{-2})}{\log r + \kappa + O(1/r)}\right)xdx \\
= \pi r^2 - 2\pi\frac{\frac12 r^2 \log r + \left(\frac12\kappa - \frac14\right)r^2 + O(\log r)}{\log r + \kappa + O(1/r)} = \left(\frac\pi2 + O(1/r)\right)\frac{r^2}{\log r + \kappa}\,,
\end{multline*}
and combining this with~\eqref{eq-Ir-formula} and the expression for $\E_0\tau_\partial$ completes the proof.
\end{proof}

Throughout the proof of Theorem~\ref{th:p(h)}, set  $ R = h/\log h $;
as outlined in~\S\ref{sec:intro-pfs},
we will show that the large deviation problem for the DG measure $\pi$ is well-approximated by the real-valued variational problem~\eqref{eq:variation} on a ball whose radius is of this order.
\subsection{Proof of Theorem~\ref{th:p(h)}, Eq.~\eqref{eq-ratio-p(h)/p(h-1)}}
We begin by proving the lower bound on the ratio
$\pi(\eta_0=h)/\pi(\eta_0=h-1)$.

Fix $c>0$ and consider the event $A$ in which $\eta_x\ge \lambda_h$  for all
four neighbors of the origin, where $\lambda_h:=h-1 - cR/8$. For any $\eta\in A$ such that $\eta_0=h-1$ we define $\eta'_x=\eta_x
+\delta_{0,x}$ so that $\eta'_0=h$ and
\[
\cH(\eta')-\cH(\eta) =4 + 2
\sum_{x: \, x\sim 0}(h-1-\eta_x)\le 4+ c R\,.
\]
Hence, by the FKG inequality,
\begin{align*}
\frac{\pi(\eta_0=h)}{\pi(\eta_0=h-1)}&\ge e^{-c\beta R-4\beta} \pi(A\mid \eta_0=h-1) \ge e^{-c\beta
  R-4\beta} \pi(\eta_{a}\ge \lambda_h\mid \eta_0=h-1)^4\,,
\end{align*}
where $a=(1,0)$ (say).
The sought lower bound would thus follow from showing that
\begin{equation}
  \label{eq-eta-lambdah-bound}
\pi(\eta_a\le \lambda_h\mid \eta_0=h-1)\le 1/2
\end{equation}
if the
constant $c$ entering in the definition of $\lambda_h$
is chosen to be large enough.

Given $\eta$ such that $\eta_0=h-1$, we define the new variables $\sigma=\{\sigma_x\}_{x\in\bbZ^2}$ by
\[
\eta_x = \phi_x +\sigma_x\,,
\]
where $\phi$ is the optimizer of the variational problem~\eqref{eq:variation} for the ball $B_R$ with height $h-1$ at the origin.
Notice that $\sigma_0=0$ and that $\sigma_x=\eta_x$ outside the ball
$B_R$. Moreover, using the fact that $\phi$ is harmonic inside
$B_R\setminus \{0\}$
and non-negative inside $B_R$,
\begin{align}
  \label{eq-eta-sigma-energy}
\cH(\eta)= \cH(\phi) + \cH(\sigma) -8\sum_{x\in \partial B_R}\sigma_x\, \Delta\phi_x\,.
\end{align}
Thus, the distribution $\mu$ of the variables
$\{\sigma_x\}_{x\in \bbZ^2}$ can be written as
\[
\mu(\sigma) \propto \exp\bigg[
-\beta \bigg(
\cH(\sigma)-8\sum_{x\in \partial B_R}\sigma_x \,\Delta\phi_x
\bigg)\bigg]\,,
\]
while insisting that within $B_R$ the variables $\sigma$
must take values which, after adding $\phi$, become integers.
Recalling that $\phi_x= h\bbP_x(\tau_0<\tau_{\partial B_R})$
as well as~\eqref{eq-phi-x-formula}, we can now take $c$ sufficiently large so that $h-1-cR/8 - \phi_a\le -(cR/16 +1)$. With this
choice, we get
\begin{gather*}
\pi(\eta_a\le \lambda_h\mid
\eta_0=h-1)
\le
\mu\Bigl(\sigma_a\le -(cR/16 +1)\Bigr)\,.
\end{gather*}
Notice that the event $\{\sigma_a\le -(cR/16 +1)\}$ is decreasing while
the function
\[
F(\sigma):=\exp\bigg(8\beta \sum_{x\in \partial B_R}\sigma_x \,\Delta\phi_x\bigg)
\]
is increasing since $\Delta \phi_x\ge 0$ for any $x\in \partial B_R$. Thus, we can apply FKG to get
that
\[
\mu\Bigl(\sigma_a\le -(cR/16+1)\Bigr)\le \tilde \mu\Bigl(\sigma_a\le -(cR/16+1)\Bigr)\,,
\]
where $\tilde \mu \propto \exp(-\beta \cH(\sigma))$.
To bound the latter probability from above, we make a
final change of variables:
for any $z\in \bbR$, put $z=\bar z +\{z\}$, where $z\in \bbZ$ and $\{z\}\in [-1/2,1/2)$. As $\phi_x+\sigma_x\in \bbZ$, clearly
$\{\sigma_x\}=-\{\phi_x\}$; thus, we can write the Hamiltonian of $\bar \sigma$
as
\begin{align}
  \label{eq-energy-sigma-barsigma}
\bar \cH(\bar \sigma):=\cH(\bar \sigma) + \cH(\{\phi\}) -2\sum_{x\sim
  y}\nabla_{x,y}\{\phi\}\nabla_{x,y}\bar\sigma\,,
\end{align}
where $\nabla_{x,y} f=f_x-f_y$. 
As usual, the constant term $\cH(\{\phi\}) $ does not play any role, and so
the law $\bar \mu$ of the variables $\bar \sigma$ satisfies
\[
\bar\mu(\bar\sigma) \propto \exp\bigg[-\beta \cH(\bar \sigma)+2\beta \sum_{x\sim
  y}\nabla_{x,y}\{\phi\}\nabla_{x,y}\bar\sigma\bigg]\,.
\]
Altogether, as $\{\sigma:\ \sigma_a\le -(cR/16+1)\}\subset
\{\sigma:\ \bar \sigma_a\le -cR/16\}$, the inequality~\eqref{eq-eta-lambdah-bound} will follow from showing that
\begin{equation}
  \label{eq-mu-bar-bound}
 \bar{\mu}\left(\bar\sigma_a  \leq -cR/16\right) \leq 1/2\,.
\end{equation}
To this end, we compare
$\bar \mu$ to a slight modification of the measure of the original DG.
Let
$\nu$ be the Gibbs measure of the \emph{non-homogeneous} DG
model on $\bbZ^2\setminus \{0\}$
with zero boundary condition at the origin, in which
the coupling constant for bonds inside $B_R$ (or on its interface) is equal to $1/2$ while it is
1 for the bonds outside $B_R$. More formally, $\nu=\lim_{\Lambda\uparrow
  \bbZ^2}\nu_\Lambda^0$ where $\nu_\Lambda^0$ is the Gibbs measure in $\Lambda$ with
zero boundary conditions at $\partial \Lambda\cup\{0\}$ and inverse
temperature $\beta$, associated to
the Hamiltonian
\[
\sum_{\substack{x\sim y \\ \{x,y\}\cap
      B_R=\emptyset}}(\nabla_{x,y}\bar\sigma)^2 + \frac 12 \sum_{\substack{x\sim y \\ \{x,y\}\cap
      B_R\neq \emptyset}}(\nabla_{x,y}\bar\sigma)^2\,.
\]
\begin{claim}
There exists some absolute $D>0$ such that  $\bar\mu(\bar \sigma)/\nu(\bar \sigma)\le  e^{D R^2}$ for $\beta$ large.
\end{claim}
\begin{proof}
Without loss of generality, and only to give a sense to the  partition
functions that will appear below, assume that both $\bar \mu$ and $\nu$ are restricted to a ball
of radius $L\gg R$ with zero boundary conditions. (Our bounds will of course be
uniform in $L$.)

Letting $\cZ_{\bar \mu},\cZ_\nu$ denote the partition functions of $\bar \mu$ and $\nu$
respectively, we have
\begin{align*}
  \frac{\bar\mu(\bar \sigma)}{\nu(\bar \sigma)}
= \frac{\cZ_\nu}{\cZ_{\bar
      \mu}} \exp\biggl[-\beta\bigg(\sum_{\small\substack{x\sim y \\ \{x,y\}\cap
      B_R\neq \emptyset}}\tfrac 12(\nabla_{x,y}\bar \sigma)^2
  -2\nabla_{x,y}\{\phi\}\nabla_{x,y}\bar\sigma\bigg)\biggr]
\le \frac{\cZ_\nu}{\cZ_{\bar
      \mu}} e^{c\beta R^2}
\end{align*}
for an absolute $c>0$, where the inequality followed from the fact that for any $a,b$ we have
$\frac12 a^2 - 2a b \geq -2b^2$, and so (using $b\in[-1/2,1/2)$), the above exponent is at most $\exp[\frac12 \beta \cE(B_R)]$, in which $\cE(B_R)\asymp R^2$ is the number
of bonds incident to the ball $B_R$.

The ratio $\cZ_{\bar \mu}/\cZ_\nu$ can be bounded from below
using
Jensen's inequality by
\begin{align*}
\frac{\cZ_{\bar \mu}}{\cZ_\nu}&=\E_\nu \Bigg[\exp\biggl[-\beta\,\biggl(\sum_{\small\substack{x\sim y \\ \{x,y\}\cap
      B_R\neq \emptyset}}\tfrac{1}{2}(\nabla_{x,y}\bar \sigma)^2
  -2\nabla_{x,y}\{\phi\}\nabla_{x,y}\bar\sigma\,\biggr)\biggr]\Bigg]\\
&\ge
\exp\bigg[-\beta\bigg(\sum_{\small\substack{x\sim y\\ \{x,y\}\subset
    B_R}}\tfrac 12\,\E_\nu\left[(\nabla_{x,y}\bar\sigma)^2\right]
  -2\nabla_{x,y}\{\phi\}\E_\nu\left[\nabla_{x,y}\bar\sigma\right]\bigg)\bigg]
  \end{align*}
(here denoting by $\E_\nu$ expectation over $\sigma$ w.r.t.\ $\nu$), which in turn
is at least $ e^{-\epsilon(\beta)R^2}$ for some $\epsilon(\beta)$ that vanishes as $\beta\to \infty$. This completes the proof.
\end{proof}
Following the above claim, in order to prove~\eqref{eq-mu-bar-bound} (and thereby complete the proof of the lower bound on
$\pi(\eta_0=h)/\pi(\eta_0=h-1)$) it will suffice to show that
\begin{equation}
  \label{eq-nu-bd}
\nu(\bar\sigma_a\le -cR/16)\le
e^{-\alpha(c)R^2}\quad\mbox{ for some $\alpha(c)$ with $\lim_{c\to \infty}\alpha(c)=\infty$}\,.
\end{equation}
We claim that this is
obvious because the vertex $a$ is a nearest neighbor of the origin, at which
$\bar\sigma_0=0$. Call a closed circuit in the dual lattice of $\Z^2$
a \emph{$0$-contour of $\eta$} if it separates negative and non-negative heights in $\eta$ (i.e., it consists of bonds dual to edges $x\sim y$ with $\eta_x<0$ and $\eta_y\geq 0$; see~\S\ref{sec:level-line-tools} for a formal (more general) definition).
If $\bar\sigma_a\le -cR/16$, then $\eta$ contains some 0-contour $\Gamma_0$ that goes through the bond dual to the edge $0\sim a$.
The energy cost of $\Gamma_0$ is at least $\frac12 \beta \big( |\Gamma_0| -1 + (cR/16)^2\big) $ (with the factor 1/2 due to the modified coupling constants in $\nu$)
and~\eqref{eq-nu-bd} now follows from a Peierls argument (cf., e.g., Claim~\ref{cl:eta-eq-geq} below).
This establishes the sought lower bound.

It remains to prove the upper bound in~\eqref{eq-ratio-p(h)/p(h-1)}. We start with a na\"ive Peierls argument that gives a weaker bound of $\epsilon(\beta)$ (vs.\ the targeted $\exp(-c_1\beta h/\log h)$ from~\eqref{eq-ratio-p(h)/p(h-1)}).
\begin{claim}
  \label{cl:eta-eq-geq}
  For any finite connected subset $V\subset\Z^2$ and any $z\in V$ and $h \geq 0$, we have
  \[ \pi_V(\eta_z > h) \leq \epsilon(\beta) \pi_V(\eta_z=h)\qquad\mbox{where $\epsilon(\beta)\to 0$ as $\beta\to\infty$}\,.\]
\end{claim}
\begin{proof}
If $\eta_z\geq h$ for $h\geq 1$ then (by the zero boundary) $\eta$ contains an $h$-contour (the analogue of the 0-contour from above, i.e., separating $x\sim y$ with $\eta_x< h$ and $\eta_y\geq h$) surrounding $z$ in $V$.
If a fixed circuit $\gamma$ is an $h$-contour of $\eta$, then the bijection taking $\eta\mapsto \eta-1$ in the interior of $\gamma$ decreases the Hamiltonian
by at least $|\gamma|$ (as $(b-a)^2\ge 1 +(b-a-1)^2$ for any $b\ge 1$ and $a\le 0$).
This $\gamma$ must intersect the $x$-axis at distance at most $|\gamma|/2$ from $z$,
from which there are at most $4^{|\gamma|}$ choices for its path, so
\begin{align*}
\pi_V(\eta_0=h) &\leq \sum_{\ell \geq 4} \ell \big(4 e^{-\beta}\big)^{\ell} \pi_V(\eta_0=h-1) \leq e^{-\beta}\pi_V(\eta_0=h-1)\,,
\end{align*}
where the last inequality holds for large enough $\beta$, and the desired result follows.
\end{proof}
To boost this upper bound to its required form, we need the following result.
 \begin{lemma}
 \label{lem:Br-bound}
Let $V\subset \Z^2$ with $0\in V$. For any $h \geq 1$ and $r\geq 1$,
\begin{align}
  \label{eq-Br-upper-bnd}
  \pi_V(\eta_0=h) \leq e^{-\frac34 \beta r}\pi_V(\eta_0=h-1)  + e^{\epsilon(\beta) r}\pi_{B_{r}}(\eta_0 = h)\,,
\end{align}
where $\epsilon(\beta)\to 0$ as $\beta\to\infty$.
 \end{lemma}
 \begin{proof}
For any $\eta$ with $\eta_0\geq 1$ let $\Gamma_1=\Gamma_1(\eta) $ be
the outermost $1$-contour around the origin in $\eta$.
By the same Peierls argument that was used in the proof of Claim~\ref{cl:eta-eq-geq},
\begin{align*}
\pi\left(\eta_0=h\,,\, |\Gamma_1|\geq r\right) &\le \sum_{\ell\geq r} \ell \big(4 e^{-\beta}\big)^{\ell} \pi(\eta_0=h-1) \leq e^{-\frac34 \beta r} \pi(\eta_0 = h-1)
\end{align*}
if $\beta$ is suitably large.
On the other hand, the event $|\Gamma_1|\le r$
implies that in $B_{r}$ there exists a chain of sites enclosing the origin, with
length at most $r$,  where the
heights are at most zero. If $C_1=C_1(\Gamma_1)$ denotes this chain of sites, then
\begin{equation}\label{eq-pi-eta-small-Gamma1} \pi(\eta_0\geq h\,,\,|\Gamma_1|\leq r) \leq
\pi\left(\eta_0 \geq h \mid \eta\restriction_{\partial C_1}\leq 0\right)
\leq \pi_{C_1}(\eta_0\geq h) \leq \max_{\substack{\Lambda\subset B_r \\ |\partial \Lambda|\leq r}} \pi_{\Lambda}(\eta_0 \geq h)\,,
\end{equation}
where we used monotonicity to replace the condition $\{\eta\restriction_{C_1}\leq 0\}$ by $\{\eta\restriction_{C_1}=0\}$.

Finally, observe that for any $r\geq 1$ and any sets $V_2\supset V_1 \ni 0$ (including $V_2=\Z^2$),
\begin{align}\label{eq-pi-region-compare}
\pi_{V_2}(\eta_0\geq h) \geq e^{-\epsilon(\beta) |\partial V_1|}\pi_{V_1}(\eta_0\geq h)\,,
\end{align}
since, again by monotonicity (now allowing us to replace $\{\eta\restriction_{\partial V_1}\geq 0\}$ by $\{\eta\restriction_{\partial V_1}=0\}$),
\begin{align*}
\pi_{V_2}(\eta_0\geq h)&\geq \pi_{V_2}(\eta_0\ge h\,,\, \eta\restriction_{\partial V_1} \geq 0) \geq \pi_{V_1}(\eta_0\ge h)\,\pi_{V_2}(\eta\restriction_{\partial V_1} \ge 0)\\
&\geq \pi_{V_1}(\eta_0\ge h)\prod_{x\in \partial V_1}\pi_{V_2}(\eta_x\ge 0) \geq e^{-\epsilon(\beta)|\partial V_1|}\,\pi_{V_1}(\eta_0\geq h)\,,
\end{align*}
where the inequality between the lines is by FKG, and the last transition used that $\pi_{V_2}(\eta_x\neq 0)<\epsilon(\beta)$ thanks to Claim~\ref{cl:eta-eq-geq}.
In particular, the right-hand side of~\eqref{eq-pi-eta-small-Gamma1} is at most
$e^{\epsilon(\beta)r}\pi_{B_r}(\eta_0\ge h)$, and a final application of Claim~\ref{cl:eta-eq-geq} concludes the proof.
 \end{proof}

\begin{corollary}
\label{cor:2}
There exists some $\epsilon(\beta)$ with $\lim_{\beta\to\infty}\epsilon(\beta)=0$ such that, for any $r\geq 1$,
\[
 \pi(\eta_0=h)\geq (1-\epsilon(\beta))e^{-\epsilon(\beta) r}\pi_{B_{r}}(\eta_0=h)\,,
\]
whereas for any $r\geq 2c_0 R$ with $c_0$ from~\eqref{eq-ratio-p(h)/p(h-1)},
\[
\pi(\eta_0=h)\le (1+\epsilon(\beta))e^{\epsilon(\beta) r} \pi_{B_{r}}(\eta_0= h)\,.\]
 \end{corollary}
\begin{proof}
Letting $V_1=B_r$ and $V_2=\Z^2$ in~\eqref{eq-pi-region-compare} gives $\pi(\eta_0\geq h) \geq e^{-\epsilon(\beta)r}\pi_{B_r}(\eta_0\geq h)$, and
Claim~\ref{cl:eta-eq-geq} extends this lower bound to $\pi(\eta_0=0)$ via an extra $(1-\epsilon(\beta))$-factor.

For the upper bound we appeal to Lemma~\ref{lem:Br-bound}, and examine the two terms featured in the right-hand side of~\eqref{eq-Br-upper-bnd}.
We will retain the second term, $e^{\epsilon(\beta)}\pi_{B_r}(\eta_0=h)$, as our main term in the upper bound,
while the first term, using our lower bound on $\pi(\eta_0=h)/\pi(\eta_0=h-1)$ from~\eqref{eq-ratio-p(h)/p(h-1)}, is
\[ e^{-\frac34\beta r} \pi(\eta_0=h-1) \leq e^{-\frac34\beta r+c_0\beta R}\pi(\eta_0=h) \leq e^{-\beta r/4}\pi(\eta_0=h)\]
for any $r\geq 2c_0 R$. The latter is at most $\epsilon(\beta)\pi(\eta_0=h)$, which concludes the proof.
\end{proof}
We are now ready to establish the upper bound on
$\pi(\eta_0=h)/\pi(\eta_0=h-1)$.
\begin{lemma}
  \label{lem:1}
With $I_r(h)$ as in~\eqref{eq:variation},
there is a constant $c'>0$ so that, for any $r\geq 1$,
\begin{align*}
  \exp\left(-\beta I_r(h) - c' r^2\right) \leq \pi_{B_{r}}(\eta_0=h)\leq \exp\left(-\beta I_r(h) +c' r^2\right)\,.
\end{align*}
\end{lemma}
\begin{proof}
As before, we let $\phi$ be the optimizer of the variational problem~\eqref{eq:variation} in
 $B_r$ and let $\sigma_x=\eta_x-\phi_x$.  The representation of the Hamiltonian in~\eqref{eq-eta-sigma-energy} shows that
 \[
 \cH(\eta)= \cH(\phi) + \cH(\sigma) -8\sum_{x\in \partial B_R}\sigma_x\, \Delta\phi_x = I_r(h)+\cH(\sigma)\,,
\]
where the sum vanished since $\eta\restriction_{B_r^c}=\phi\restriction_{B_r^c}=0$ (and in particular $\sigma\restriction_{\partial B_r}=0$). Hence,
\begin{align*}
\pi_{B_{r}}(\eta_0=h)&=e^{-\beta I_r(h)}\frac{1}{\cZ_{B_r}}\sum_{\sigma:\, \sigma_0=0}e^{-\beta
  \cH(\sigma)}\,.
 \end{align*}
Since $1 \leq \cZ_{B_r} \leq e^{d r^2}$ for some constant $d>0$ (see, e.g.,~\cite{BW}), it will suffice to show that
the sum above is bounded between $e^{-d' r^2}$ and $e^{d' r^2}$ for some other $d'>0$.

Writing $\sigma_x=\bar\sigma_x-\{\phi_x\}$ with
$\bar\sigma_x\in \bbZ$ and $\{\phi_x\}\in [-1/2,1/2)$,
for the lower bound we simply take $\sigma$ with $\bar{\sigma}_x=0$
(i.e., $\sigma_x=-\{\phi_x\}$) for all $x$, whence of course $\sigma_0=0$ and
\[
  e^{-\beta
  \cH(\sigma)} = e^{-\beta \sum_{x\sim
    y}(\nabla_{x,y}\{\phi\})^2}\geq e^{-\beta |\cE(B_r)|}\,,
\]
where $\cE(B_r)$ denotes the number of bonds incident to $B_r$.

For the upper bound, we infer from~\eqref{eq-energy-sigma-barsigma} that
\[ \cH(\sigma) = \cH(\bar\sigma) + \cH(\{\phi\})  -2\sum_{x\sim
  y}\nabla_{x,y}\{\phi\}\nabla_{x,y}\bar\sigma \geq \tfrac12\cH(\bar\sigma)-\cH(\{\phi\})\]
using $
 \frac12 a^2 - 2 a b \geq -2 b^2
$ for any $a,b\in\R$.
Thus,
\begin{align*}
\sum_{\sigma:\sigma_0=0}e^{-\beta\cH(\sigma)} \leq
e^{\beta |\cE(B_r)|} \sum_{\sigma:\
  \sigma_0= 0}e^{-\frac 12\beta \cH(\bar\sigma)} \le e^{\beta |\cE(B_r)|+d' r^2}
\end{align*}
again using the results in~\cite{BW}, completing the proof.
\end{proof}

Let $r=\delta R$ for a fixed (small) $\delta>0$. Recalling that $I_r(h)\sim 2\pi\beta h^2/\log r$ from Lemma~\ref{lem:dirichlet}, we get $I_r(h) \geq I_R(h) + C(\delta) R^2$ with $\lim_{\delta\to 0}C(\delta)=\infty$. Thus, by Lemma~\ref{lem:1},
\[
\pi_{B_{r}}(\eta_0 = h)\le e^{- R^2 }\pi_{B_R}(\eta_0=h)
\]
provided $\delta$ is chosen to be small enough.
Now, for $\beta$ large enough, by Claim~\ref{cl:eta-eq-geq} we get
\[ \pi_{B_R}(\eta_0=h) \leq \pi_{B_R}(\eta_0=h-1) \leq c e^{cR}\pi(\eta_0=h-1)\,,\]
with the last inequality using the first part of Corollary~\ref{cor:2}.
Combining these with~\eqref{eq-Br-upper-bnd},
\[ \frac{\pi(\eta_0=h)}{\pi(\eta_0=h-1)} \leq e^{-\frac34\beta r} + e^{-(\delta^{-2} -o(1))r^2}
= (1+o(1))e^{-\frac34\beta r}\,,\]
which concludes the proof of the required upper bound in~\eqref{eq-ratio-p(h)/p(h-1)}.
\qed
\subsection{Proof of Theorem~\ref{th:p(h)},  Eq.~\eqref{eq-p(h)}}
Let $r = \lceil 2c_0 R \rceil$. Corollary~\ref{cor:2} shows that
$\pi(\eta_0=h) = \pi_{B_r}(\eta_0=h) \exp(O(R))$
while Lemma~\ref{lem:1} and the fact $I_r(h) \asymp R^2 \log R$ (by Lemma~\ref{lem:dirichlet}) yield that $\pi_{B_r}(\eta_0=h) = \exp(-I_R(h) + O(R^2))$,
as required.
 \qed

\subsection{Proof of Theorem~\ref{th:p(h)},  Eq.~\eqref{eq-p(h,h)}}
Fix $z\in \bbZ^2$ and let
\[
X:=\max_{x\sim
  z}\eta_x\,,\, \quad Y(\eta):=\min_{x\sim
  z}\eta_x\,.
\]
Given $0<\delta\le 1$, define the events $F= \{X\le
h\}$ and $E=\{Y\ge h-\delta \sqrt{h/\log h}\}$. Since $\pi(F^c) \leq 4\pi(\eta_0\geq h+1)$ by a union bound, we can infer from~\eqref{eq-ratio-p(h)/p(h-1)} that
\[
\pi(F^c\mid \eta_0=h)\le \frac{4\pi(\eta_0\ge h+1)
}{\pi(\eta_0=h)} \leq O\left(e^{-c_1 \beta h/\log h}\right)\,.
\]
Therefore, it will suffice to establish a similar upper bound on $\pi(\eta_z=h\mid \eta_0=h\,,\, F)$.  Conditioning over the values of the
neighbors of $z$ and then using monotonicity yields
\[
\pi(\eta_z=h\mid \eta_0=h\,,\, E^c\,,\, F)\le
e^{-c'\beta h/\log h}\,.
\]
Finally, we will bound $\pi(E\mid \eta_0=h\,,\, F)$ from above as follows.
On one hand we have
\[
\pi\left(\eta_z\ge h+1\mid \eta_0=h\,,\, E\,,\, F\right)\ge e^{-4\beta \delta^2 h/\log h}\,,
\]
while on the other hand
\begin{align*}
\pi\left(\eta_z\ge h+1\mid \eta_0=h\,,\, E\,,\,F\right)
&\le
\frac{\pi(\eta_z\ge h+1\mid \eta_0=h)}{\pi\left(E\mid \eta_0=h\,,\, F\right)}
\le \frac{(1+o(1))e^{-c_1\beta h/\log h}}{\pi\left(E\mid \eta_0=h\,,\, F\right)}\,,
\end{align*}
where the last inequality used $\pi(\eta_z \geq h+1 \mid \eta_0 = h) \leq \pi(\eta_z\geq h+1) / \pi(\eta_0=h)$ together with the upper bound in~\eqref{eq-ratio-p(h)/p(h-1)}.  Combining the last two displays gives
\[
\pi\left(E\mid \eta_0=h\,,\, F\right)\le (1+o(1))e^{- \beta (c_1- 4\delta^2)h/\log h}\,,
\]
and the proof is completed by choosing $\delta^2< c_1/4$. \qed

\subsection{Proof of Theorem~\ref{mainthm:max}}
Recalling~\eqref{eq-p(h)-2}, the following definition of $M$ satisfies~\eqref{eq:E[XL]}.
\begin{equation}
  \label{eq-M-def}
  M = M(L) = \max\left\{ h : \pi(\eta_0 \geq h) \geq L^{-2} \log^5 L \right\}\,.
\end{equation}

For the lower bound, let us partition $\Lambda_L$ into disjoint boxes of side-length $\log^2 L$, and denote by $S$ the set of sites that are at their centers (whence $|S| \sim L^2 / \log^4 L$).
Then
 \begin{align*}
   \pi_\Lambda\bigg(\bigcap_{x\in S}\{\eta_x < M\} \bigg) &\geq
   \prod_{x\in S}\pi_\Lambda(\eta_x < M) = 1 - \left[1-\pi(\eta_0 \geq M) + O\left(L^{-10}\right)\right]^{|S|} \\
    &\geq   1 - \bigg[1 - \frac{\log^5 L}{L^2} + O\left(L^{-10}\right)\bigg]^{|S|}  
   \geq 1 - L^{-1+o(1)}  = 1-o(1)
   \end{align*}
(in the first line, the inequality is by FKG and the equality used that for any $x\in S$ at distance $r=\log^2 L$ from $\partial\Lambda$, one can couple $\pi_{\Lambda}$ and $\pi$ so that with probability, say, $1-O(L^{-10})$, they would agree on $B_{r}(x)$; see, e.g.,~\cite{BW}). This completes the lower bound.

The upper bound on $X_L$ will follow from a first moment argument. Thanks to~\eqref{eq-p(h,h)},
\[ \pi(\eta_0 \geq M+2) \leq L^{-2} e^{-(\log L)^{1/2-o(1)}}\,.\]
In particular, by the decay-of-correlation results of~\cite{BW}, for any $x\in\Lambda_L$ at distance at least $\log^2 L$ (say) from the boundary we readily have $\pi_\Lambda(\eta_x \geq M+2) = o(L^{-2})$. For the $O(L \log^2 L)$ sites near $\partial\Lambda_L$, letting $r=\log L$ and $h=M+2$ in~\eqref{eq-Br-upper-bnd} gives
\[ \pi_{\Lambda_L}(\eta_x = M+2) \leq L^{-\frac34\beta} + L^{\epsilon(\beta)} \pi_{B_r}(\eta_x = M+2)\,.\]
Moreover, by the first part of Corollary~\ref{cor:2},
\[ \pi_{B_r}(\eta_x = M+2) \leq (1+\epsilon(\beta))L^{\epsilon(\beta)} \pi(\eta_0=M+2)\,.\]
Therefore, with Claim~\ref{cl:eta-eq-geq} in mind, $\pi_{\Lambda_L}(\eta_x \geq M+2) \leq L^{-3/2}$ for $\beta$ large, vs.\ the $L^{1+o(1)}$ sites under consideration near $\partial \Lambda_L$. Overall, $\pi_\Lambda(X_L \geq M+2) \leq o(1)$, as needed.
\qed

\section{Entropic repulsion: Proofs of Theorems~\ref{mainthm:floor-shape} and~\ref{mainthm:floor-max}}\label{sec:floor}

Throughout this section, let $\varpi_\Lambda = \pi_\Lambda(\cdot\mid\eta\geq 0)$ denote the DG measure with a floor.
Further let $\varpi_\Lambda^j$ (similarly $\pi_\Lambda^j$) denote a boundary condition of $j$ (i.e., $\eta_x = j$ for $x\notin\Lambda$).
Occasionally we will use $\epsilon_\beta$ to denote a positive real function of $\beta$ with $\lim_{\beta\to\infty}\epsilon_\beta=0$.
\subsection{Tools for level line analysis in the DG model with and without a floor}\label{sec:level-line-tools}
\begin{definition}[Geometric contour]
\label{contourdef}
Let ${\Z^2}^*$ denote the dual lattice of $\Z^2$.
A pair of orthogonal bonds which meet in a site $x^*\in {\Z^2}^*$ is said to be a
{\sl linked pair of bonds} if both bonds are on the same side of the
main diagonal across $x^*$. A {\sl geometric contour} (for
short a contour in the sequel) is a
sequence $e_0,\ldots,e_n$ of bonds such that:
\begin{enumerate}
\item $e_i\ne e_j$ for $i\ne j$, except for $i=0$ and $j=n$ where $e_0=e_n$.
\item for every $i$, $e_i$ and $e_{i+1}$ have a common vertex in ${\Z^2}^*$
\item if $e_i,e_{i+1},e_j,e_{j+1}$ intersect at some $x^*\in {\Z^2}^*$,
then $e_i,e_{i+1}$ and $e_j,e_{j+1}$ are linked pairs of bonds.
\end{enumerate}
We denote the length of a contour $\gamma$ by $|\gamma|$, its interior
(the sites in $\bbZ^2$ it surrounds) by $V_\gamma$ and its
interior area (the number of such sites) by
$A(\gamma)$. Moreover we let $\partial_{\gamma}$ be the set of sites in $\Z^2$ such that either their distance
(in $\bbR^2$) from $\gamma$ is $\tfrac12$, or their distance from the set
of vertices in ${\Z^2}^*$ where two non-linked bonds of $\gamma$ meet
equals $1/\sqrt2$. Finally we let $\partial^+_\gamma=\partial_\gamma\cap
V_\gamma$ and $\partial^-_\gamma = \partial_\gamma\setminus \partial^+_\gamma$.
\end{definition}

\begin{definition}[$h$-contour; $\sC_{\gamma,h}$]
\label{def:levels}
Given a contour $\gamma$ we say that $\gamma$ is an \emph{$h$-contour} (or an $h$-level line)
for the configuration $\eta$, denoting this event by $\sC_{\gamma,h}$, if
\[
\eta\restriction_{\partial^+_\gamma}\leq h-1\,, \quad \eta\restriction_{\partial^-_\gamma}\geq h\,.
\]
We call $\gamma$ a \emph{contour} if it is an $h$-contour for some $h$ in $\eta$.
For the DG model on $\Lambda_L$, the box of side-length $L$, a contour
  will be called \emph{macroscopic} iff it is longer than $(\log L)^2$,
  and we let $\mac_h $ denote the event that there exists a macroscopic $h$-contour.
\end{definition}
We will further let $\mac_*=\cup_h\mac_h$ denote the event there is any macroscopic contour.

\begin{definition}[Negative $h$-contour; $\sC_{\gamma,h}^-$]
We say that a closed contour $\gamma$ is a negative $h$-contour, denoting this event by $\sC_{\gamma,h}^-$, if
\[
\eta\restriction_{\partial^-_\gamma}\leq h-1\,,\quad \eta\restriction_{\partial^+_\gamma}\geq h\,,
\]
i.e., the external boundary $\gamma$ is at least $h$ whereas its internal boundary is at most $h-1$.
As before, for the DG model on $\Lambda_L$ we call $\gamma$ \emph{macroscopic} iff it is longer than $(\log L)^2$,
  and $\mac_h^- =$ denotes the event that there exists a macroscopic negative $h$-contour.
\end{definition}

The following proposition adapts~\cite{CLMST2}*{Proposition~2.7} to the DG model.
\begin{proposition}\label{prop:h-contour-upper}
Fix $j\ge 0$ and consider the DG model in  a finite connected subset
$\Lambda$ of $\bbZ^2$ with floor at height $0$ and boundary
conditions at height $j\ge 0$. Then
\begin{align}
\label{e:contourFloorBound}
\pif^j_\Lambda\left( \sC_{\gamma,h} \right) &\leq
\exp\left[-\beta|\gamma|+\pi(\eta_0\geq h) A(\gamma) + e^{-(\frac{\pi\beta}2+o(1))h^2/\log h}|\gamma|\log|\gamma|\right]\,,\\
\label{e:negativeContour}
\pif^j_\Lambda\left( \sC_{\gamma,h}^- \right) &\leq\exp\left[-\beta|\gamma|\right]\,.
\end{align}
\end{proposition}
\begin{proof}
The estimate for $\sC_{\gamma,h}$ will be an immediate consequence of a Peierls-argument combined with FKG.
Consider the map $T_\gamma$ which decreases the value of $\eta$ by 1 in the interior of $\gamma$, that is, $(T_\gamma\eta)(x) = \eta_x-1$ if $x\in V_\gamma$ and elsewhere $(T_\gamma\eta)(x)=\eta_x$.
This map is well defined --- and moreover, bijective --- for any $\eta$ such that $\eta\restriction_\Lambda > 0$.
By definition, for any $\eta \in \sC_{\gamma,h}$ such that $\eta\restriction_\Lambda>0$
we have $\pif_\Lambda^j(T_\gamma\eta) \geq e^{\beta|\gamma|}\pif^j_\Lambda(\eta)$.
Hence,
\[ \sum_{\substack{\eta\in\sC_{\gamma,h} \\ \eta\restriction_\Lambda>0}}\pif_\Lambda^j(\eta) \leq e^{-\beta|\gamma|} \sum_{\substack{\eta: T_\gamma^{-1}\eta \in\sC_{\gamma,h} \\ \eta\restriction_\Lambda\geq0 }}\pif_\Lambda^0(T_\gamma^{-1}\eta) \leq e^{-\beta|\gamma|}\,.\]
By monotonicity, on the event $\sC_{\gamma,h}$ we may lower $\partial^-_\gamma$ exactly to $h$ and then drop the floor to obtain that
\[ \pif(\eta\in\sC_{\gamma,h}\,, \eta\restriction_\Lambda>0) \geq \pif_\Lambda^j(\sC_{\gamma,h})\pi_{V_\gamma}^h(\eta>0) \geq \pif_\Lambda^j(\sC_{\gamma,h}) \prod_{x\in V_\gamma} \left(1-\pi_{V_\gamma}(\eta_x \geq h)\right)\,,\]
with the last inequality following from FKG.

It remains to treat the last expression in the right-hand side above. In \S\ref{sec:ldev} we have seen that $\max_{x\in V_\gamma}
\pi^{0}_{V_\gamma}\left(\eta_x\geq h\right)\leq \exp[-(\pi\beta/2+o(1) h^2/\log h]$, where the $o(1)$-term goes to 0 as $h\to\infty$.
The exponential decay of correlations in the low-temperature DG model (cf.~\cite{BW}) then yields that, for instance,
\begin{gather*}
\pi^{0}_{V_\gamma}\left(\eta_x\geq h\right)
\leq \begin{cases}
e^{-(\pi\beta/2+o(1))h^2/\log h}& \text{ if ${\rm dist}(x,\gamma) \leq \log |\gamma|$}\\
\pi\left(\eta_0\geq
  h\right) + A(\gamma)^{-2}& \text{otherwise}
\end{cases}
  \end{gather*}
provided that $\beta$ is large enough. Therefore,
\begin{align*}
\prod_{x\in V_\gamma} \left(1-\pi_{V_\gamma}(\eta_x \geq h)\right) \geq &\exp\Big[-(1-o(1))e^{-(\frac{\pi\beta}2+o(1))h^2/\log h} |\gamma|\log|\gamma|\Big] \\
\cdot &\exp\Big[-(1-o(1))\pi\left(\eta_0\geq h\right)A(\gamma)\Big]\,,
 \end{align*}
implying the required estimate.

The estimate for $\sC_{\gamma,h}^-$ is simpler: here the map $T_\gamma$ which increases the heights in the interior of $\gamma$ by 1
reduces the Hamiltonian by at least $\beta|\gamma|$, yet no longer jeopardizes the floor constraint (hence the absent area term in~\eqref{e:negativeContour} compared to~\eqref{e:contourFloorBound}).
\end{proof}

The following straightforward lemma, adapting a part of \cite{CLMST2}*{Lemma 4.2} to the DG model, will reduce the height histogram of the surface (modulo the obvious local thermal fluctuations in an $\epsilon_\beta$-fraction of the sites) to the collection of macroscopic contours.
\begin{lemma}\label{lem:negative,H+2}
Consider the DG model on $\Lambda_L$ and let $h \geq \log\log L$. Then
\begin{align}
\pif_{\Lambda_L}\bigg(\sum_{\gamma:\eta\in\sC_{\gamma,h}}\!\!\!A(\gamma) \geq \epsilon_\beta L^2 ~,~ \mac_{h}^c\bigg) &= O(e^{-\log^2 L})\label{eq-h-without-mac}
\end{align}
for some $\epsilon_\beta>0$ with $\lim_{\beta\to\infty}\epsilon_\beta = 0$.
\end{lemma}
\begin{proof}
Recall from Proposition~\ref{prop:h-contour-upper} that for any given $\gamma$ of length $k \leq \log^2 L$,
\[ \pif_{\Lambda_L}(\sC_{\gamma,h}) \leq \exp\left[-\beta k + \pi(\eta_0 \geq h) k^2 + e^{-(\frac{\pi\beta}2+o(1))\frac{h^2}{\log h}} k \log k\right] = \exp\left[-(\beta - o(1))k\right]\,,\]
since $\log \pi(\eta_0 \geq h) = -(\log\log L)^{2-o(1)}$ compared to $\log k = O(\log\log L)$ (and similarly we have $\exp(-c h^2/\log h)k\log k = o(k)$ for the third term in the exponent);
for large enough $L$ we can therefore use the upper $\exp(-\beta k/2)$ for this event.

Let $N_k$ be the number of $h$-contours whose length is precisely $k \leq \log^2 L$.
There are at most $L^2 4^k$ possible such contours, and so for any integer $a$,
\begin{align*}
\pif_{\Lambda_L}(N_k \geq a) \leq \sum_{r \geq a} \binom{L^2 4^k}r e^{-\frac\beta2 k r} \leq \frac{\P(Y_k\geq a)}{(1-e^{-\frac\beta2 k})^{L^2 4^k}}
\leq e^{2e^{-\frac\beta2k}  L^2 4^k} \P(Y_k\geq a)\,,
\end{align*}
where $Y_k \sim \bin(L^2 4^k,e^{-\frac\beta2 k})$ and we used $1-x \geq e^{-2x}$ for $0\leq x \leq \frac12$, certainly the situation here for $x=\exp(-\beta k/2)$ with $\beta$ large.
Selecting
\[ a_k = 2 L^2 \left(4 e^{-\frac\beta2}\right)^k + 2\log^2 L\,,\]
the bound $\P(Y_k \geq \mu + t) \leq \exp[-\frac12 \frac{t^2}{\mu+t/3}]$, valid for any $t>0$ and binomial variable $Y_k$ with mean $\mu$, shows here (where $t \geq 2\mu$ and so $\P(Y_k \geq \mu + t) \leq \exp(-t)$ holds) that
\[ \pif_{\Lambda_L}(N_k \geq a_k) \leq e^{-2\log^2 L}\,.\]
A union bound now implies that $N_k \leq a_k$ for all $k\leq \log^2 L$ except with probability $\exp(-(2-o(1))\log^2 L)$. On this event, and barring macroscopic $h$-contours, we have
\[ \sum_{\gamma: \eta\in\sC_{\gamma,h}}A(\gamma) \leq \sum_{k=1}^{\log^2 L} a_k k^2 \leq \epsilon_\beta L^2\,,\]
where $\epsilon_\beta$ decreases as $ O(e^{-\beta/2})$ for large $\beta$. This completes the proof.
\end{proof}

We conclude this subsection by introducing --- and thereafter studying --- an event which will be instrumental in estimating the probability that the entire surface rises above a certain height in the presence of a floor:
\begin{equation}
  \label{eq-def-P-event}
  \cP_r^{\neq h} = \Big\{ \exists P=(x_0,\ldots,x_k)\;:\; |x_k-x_0|\geq r~,~ |x_{i+1}-x_i|=1~,~\eta_{x_i}\neq h~\forall i\Big\}\,.
\end{equation}
That is, $\cP_{r}^{\neq h}$ is the event that there exists some path of vertices $P$ so that its endpoints have distance at least $r$ in $\Z^2$ and all along it the configuration differs from $h$.

\begin{lemma}\label{lem:neq-j-path}
Let $\cP_{r}^{\neq j}$ be the event defined in~\eqref{eq-def-P-event}. If $r = \log^2 L$ and $j\geq a\log\log L$ for some fixed $a>0$ then
\[ \pif_{\Lambda_L}^{j}\left(\cP_{r}^{\neq j}~,~\mac_*^c\right) = O(e^{-\log^2 L})\,.\]
\end{lemma}
\begin{proof}
Let $\Gamma = \{\gamma_i\}$ be a collection of contours with pairwise disjoint interiors $\{V_{\gamma_i}\}$ and lengths at most $\log^2 L$ each.
By Proposition~\ref{prop:h-contour-upper}, for each $i$ we have
\begin{equation}
   \label{eq-external-single-plus-minus}
 \pif_{\Lambda_L}^{j}\left(\sC_{\gamma_i,j+1}~,~\mac_*^c\right) \leq e^{-(\beta-o(1)) |\gamma_i| } \quad \mbox{ and }
\quad \pif_{\Lambda_L}^{j}\left(\sC^-_{\gamma_i,j-1}~,~\mac_*^c\right) \leq e^{-\beta |\gamma_i|} \,,
\end{equation}
where the first inequality used $A(\gamma_i) \leq |\gamma_i|^2/16 \leq |\gamma_i|\log^2 L$ combined with the bounds that
the two terms $\pi(\eta_0\geq j+1)\log^2 L $ and $\exp\big[-(\frac{\pi\beta}2+o(1))\frac{(j+1)^2}{\log (j+1)}\big]\log\log L$ are both $\exp\big(-(\log\log L)^{2-o(1)}\big)$
thanks to our assumption on $j$ and Theorem~\ref{th:p(h)} (Eq.~\eqref{eq-p(h)}).

As these are the only two types of contours we will need throughout this proof, we will simply call a $(j+1)$-contour a \emph{plus-contour} and a negative $(j-1)$-contour a \emph{minus-contour},
 and denote the corresponding events by $\sC_\gamma^+$ and $\sC_\gamma^-$, for brevity.

Strengthening~\eqref{eq-external-single-plus-minus}, we claim that for any partition of $\Gamma$ into $\Gamma = \Gamma^+ \cup \Gamma^-$,
\begin{equation}
   \label{eq-external-plus-minus}
   \pif_{\Lambda_L}^{j}\Bigg(\bigcap_{\gamma\in\Gamma^+}\sC_{\gamma_i}^{+} ~,~ \bigcap_{\gamma\in\Gamma^-}\sC^-_{\gamma_i}~,~\mac_*^c\Bigg) \leq e^{-\left(\beta - o(1)\right) \sum_i|\gamma_i|}\,.
 \end{equation}
Indeed, the maps $T_\gamma$ from the proof of Proposition~\ref{prop:h-contour-upper} can be applied simultaneously for all $\{\gamma_i\}$, as their interiors are pairwise disjoint. It is important to note that a dual edge $e$ cannot belong to two distinct plus-contours $\gamma'\neq\gamma'' \in \Gamma^+$ nor to two distinct minus-contours $\gamma'\neq\gamma'' \in \Gamma^-$, since that would make them either share a common interior vertex or violate the definitions of positive/negative $h$-contours.
If $e$ belongs to a unique $\gamma\in\Gamma$ then its contribution to the Hamiltonian will decrease by at least $\beta$ following the map $T$, whereas if it belongs to $\gamma'\in\Gamma^+$ as well as to $\gamma''\in\Gamma^-$ (in this case necessarily $e=(x,y)$ such that $\eta_x = j+1$ and $\eta_y=j-1$) then the change is $4\beta$, and either way we see that the Hamiltonian decreases by $\beta\sum_{i}|\gamma_i|$ (here it would have sufficed to have a contribution of $2\beta$, rather than $4\beta$, from the latter case).
As before, the map must be valid for every $\gamma\in\Gamma^+$ --- where we should have $\eta\restriction_{V_\gamma} \geq 0$ --- again resulting in the terms involving $A(\gamma)$ and $|\gamma|\log|\gamma|$, which as stated above translate to a $1+O(L^{-c})$ factor, thus substantiating~\eqref{eq-external-plus-minus}.

We will apply the above inequality for $\Gamma$ that is a subset of external-most contours: Thanks to the boundary conditions, every $x\in\Lambda_L$ for which $\eta_x\neq j$ must be surrounded either by an external-most plus-contour or by an external-most minus-contour. By definition, any two such contours (out of the set of external-most plus/minus-contours) have disjoint interiors.

Consider now some path of vertices $P=(x_1,\ldots,x_m)$ as a candidate for fulfilling the event $\cP_r^{\neq j}$. By the discussion above, every $x_i\in P$ must belong to $V_{\gamma_i}$ for some external-most contour $\gamma_i$ such that $\sC_{\gamma_i,j+1} \cup \sC_{\gamma_i,j-1}^-$ holds. Beginning with $x_1$, examine the contour $\gamma_1$ and consider the last $i$ such that $x_i\in V_{\gamma_1}$, i.e., the last time that an edge $x_i x_{i+1}$ of $P$ intersects an edge of $\gamma_1$, call that dual edge $e_1$. The key observation is that $e_1$ must belong to some external-most contour $\gamma_2$ --- with an opposite sign compared to $\gamma_1$ --- as otherwise there will be a vertex of $P$ (namely, $x_{i+1}$) that is not encircled by any external-most plus/minus contour.

Overall, the event $\cP_r^{\neq j}$ implies that there exists a chain of contours $\{\gamma_1,\ldots,\gamma_k\}$ with pairwise disjoint interiors and alternating signs, such that $\gamma_i,\gamma_{i+1}$ share a common edge for every $i$ and there are two points $a\in\ V_{\gamma_1}$ and $b\in V_{\gamma_k}$ whose distance is at least $r$. Noting that this implies  $\sum |\gamma_i| \geq r$, we can now appeal to~\eqref{eq-external-plus-minus} and obtain that
\[
  \pif_{\Lambda_L}^{j}\left(\cP_r^{\neq j}~,~\mac_*^c\right) \leq 2 L^2 \sum_{k\geq 1} {\sum}'_{\gamma_1,\ldots,\gamma_k} e^{-(\beta-o(1))\sum |\gamma_i|}\,,
\]
 where the $L^2$-term is for the starting point of $\gamma_1$, the factor $2$ is for whether $\gamma_1$ is a plus/minus-contour, and $\sum'$ runs over contours $\gamma_1,\ldots,\gamma_k$ with alternating signs and pairwise disjoint interiors, where each $\gamma_i,\gamma_{i+1}$ share a common edge and $\sum |\gamma_i| \geq r$.
 For a given choice of lengths $l_1,\ldots,l_k$ for these, there are at most $3^{l_1}$ choices for $\gamma_1$ (as we rooted it and chose its sign), and thereafter there are at most $l_{i-1}3^{l_i}$ for $\gamma_i$ (it is rooted at an edge of its predecessor and its sign is dictated to be the opposite of $\gamma_{i-1}$). Altogether, the above probability is at most
 \begin{align*}
\pif_{\Lambda_L}^{j}\left(\cP_r^{\neq j}~,~\mac_*^c\right) &\leq 2 L^2 \sum_{k\geq 1} \sum_{\substack{l_1,\ldots,l_k \\ \sum l_i \geq r}} 3^{\sum l_i} \big(\prod l_i\big) e^{-(\beta-o(1))\sum l_i} \\
&\leq 2 L^2 \sum_{k\geq 1}\bigg(\sum_{l} l \left(3 e^{-\frac12(\beta-o(1))}\right)^l \bigg)^k e^{-\frac12 (\beta-o(1)) r} \leq 2 L^2 e^{-2r}
 \end{align*}
 for large enough $\beta$, and recalling that $r \geq \log^2L$ now completes the proof.
 \end{proof}

\subsection{An upper bound on the probability that the DG surface is non-negative}
\begin{proposition}\label{prop:grill}
Consider the DG model on some region $V\supset \Lambda_L$ and define the event $\cP=\cP_{\log^2 L}^{\neq h}$
following the notation in Eq.~\eqref{eq-def-P-event} for $\log\log L\leq h \leq \log L$. Then
\begin{align}
\pi^h_{V}\left( \eta \geq 0 ~,~ \cP^c \right) &\leq
\exp\Big[-(1-o(1))\pi(\eta_0\geq h+1) L^2 \Big]\,.
\end{align}
\end{proposition}
\begin{proof}
Set
\[ \ell = \lfloor \log^3 L\rfloor \,,\qquad \ell^+ = \ell + 4\lfloor \log^2 L\rfloor\,,\]
partition the box $\Lambda_L$ into a grid of boxes $Q_i^+$, each of side-length $\ell^+$, and let $Q_i\subset Q_i^+$ be the
box of side-length $\ell$ centered in $Q_i^+$ (i.e., at distance $2\lfloor \log^2 L\rfloor$ from $\partial Q_i^+$).

Let $C_i$ denote the external-most circuit of sites such that
\begin{equation}\label{eq-Ci} \eta\restriction_{C_i} = h\,,\quad {\rm dist}(C_i,\partial Q_i^+) \leq \log^2 L\,.\end{equation}
We claim that, under the assumption $\cP^c$, necessarily such a circuit $\cC_i \subset Q_i^+$ exists. Indeed, if this were not the case then there would be a chain $C'$ crossing the frame of width $\log^2 L$ from $\partial Q_i^+$ where the heights all differ from $h$, contradicting $\cP^c$.

Condition on $C_i$ for each $i$, thereby de-correlating the marginals of $\eta$ on their interiors $V_i := V_{C_i}$, while noting that, crucially, this conditioning does not reveal any information on $\eta_{V_{i}}$ beyond the fact that $\eta\restriction_{C_i}=h$. It now easily follows that
\[ \pi^h_{V}\left( \eta \geq 0 ~,~ \cP^c \right) \leq \prod_{i} \sup_{V_i} \pi^h_{V_i}\left( \eta \geq 0 \right)
\leq \prod_{i} \sup_{V_i} \pi^h_{V_i}\left( \eta\restriction_{Q_i} \geq 0 \right)
\,,\]
where the supremum runs over all possible chains $C_i$ in the aforementioned frame as given in~\eqref{eq-Ci}.
To estimate the probabilities in the right-hand side we appeal to Bonferonni's inequalities, whence
\begin{multline*}
\pi^h_{V_i}\left( \eta\restriction_{Q_i} \geq 0 \right) \leq 1- \sum_{x\in Q_i}\pi_{V_i}^h(\eta_x<0)
+ \frac12 \sum_{x, y\in Q_i,\,x\neq y}\pi_{V_i}^h(\eta_x<0\,,\,\eta_y<0) \\
\leq 1- \sum_{x\in Q_i}\pi^h(\eta_x<0)
+ \frac12 \sum_{x,y\in Q_i,\, x\neq y}\pi^h(\eta_x<0\,,\,\eta_y<0) + O\left(|Q_i|^2 e^{-\log^2 L}\right)\,,
\end{multline*}
where the last inequality used the decay of correlation in the DG model (see, e.g.,~\cite{BW}) to replace the measure $\pi_{V_i}$ by $\pi$ thanks to the distance of $\log^2L$ between $Q_i$ and $\partial V_i$.
The summation over unordered pairs $x,y\in Q_i$ can be bounded from above by
\[ \sum_{\substack{x,y\in Q_i,\, x\neq y\\ \!\!\!{\rm dist}(x,y)\leq\log^2L}}\pi^h(\eta_x<0\,,\,\eta_y<0) +
\bigg(\sum_{x\in Q_i}\pi^h(\eta_x<0)\bigg)^2 + O\left(|Q_i|^2e^{-\log^2 L}\right)\,,\]
again by the decay of correlation. Moreover,
\begin{align*}
\sum_{\substack{x,y\in Q_i,\,x\neq y\\ {\rm dist}(x,y)\leq\log^2L}}\!\!\!\!\!\!\!\pi^h(\eta_x<0\,,\,\eta_y<0)&\leq
\sum_{x\in Q_i}\pi^h(\eta_0 <0)\!\!\!\!\!\!\! \sum_{\substack{y\in Q_i \\ {\rm dist}(x,y)\leq \log^2 L}} \!\!\!\!\!\!\!\pi^h\left(\eta_y <0\mid \eta_x <0\right) \\
&= o\bigg(\sum_{x\in Q_i} \pi^h(\eta_0 <0)\bigg)\,,
\end{align*}
using~\eqref{eq-p(h,h)} and that $\exp[-h^{2-o(1)}] = o(\log^{-4} L)$ since $h\geq \log\log L$. In conclusion,
as $|Q_i|e^{-\log^2 L} \ll \pi_h(\eta_0 < 0)$ for $h\leq\log L$, we obtain that
\begin{align*}
\pi^h_{V_i}\left( \eta\restriction_{Q_i} \geq 0 \right) &\leq e^{ - (1-o(1))\,|Q_i|\, \pi^h(\eta_0 < 0)} = e^{ - (1-o(1))\,|Q_i|\, \pi(\eta_0 \geq h+1)}\,.
 \end{align*}
The product over $(L/\ell^+)^2 = (1+o(1))L^2 / \ell^2$ squares $Q_i$ (recalling that $|Q_i|=\ell^2$) now shows that $
\pi^h_{V_i}( \eta\geq 0 \,,\,\cP^c) $ is at most $\exp[- (1-o(1))\pi(\eta_0\geq h+1)L^2 ]$, as required.
\end{proof}

\subsection{Two-point concentration for the surface height}

\begin{proposition}\label{prop:blackbox}
  Fix $\epsilon>0$. If $\beta$ is large enough and $\ell,h$ are two integers satisfying
  \begin{equation}\label{eq-l,h-relation} \frac{4\beta + 2}{\pi(\eta_0 \geq h)} \leq \ell \leq \frac{4\beta+4}{\pi(\eta_0 \geq h)} \end{equation}
  then the following holds. For any circuit of sites $C$ such that $|C|\leq (4+e^{-\beta})\ell$ and $V=V_C$ satisfies $\Lambda_\ell \subset V \subset \Lambda_{\bar\ell}$ for $\bar\ell=\lceil \ell+\log^2\ell \rceil$, with probability $1-O(e^{-\log^2 \ell})$ the configuration $\eta\sim \pif_{V}^{h-1}$ admits an $h$-contour $\gamma$ that encapsulates a square $\Lambda_{(1-\epsilon)\ell}$.
\end{proposition}
The proof we will use a straightfowrard isoperimetric estimate which appeared, e.g., in~\cite{CLMST2}*{Lemma~2.2}); we include its short proof for completeness.
\begin{lemma}\label{lem:isop}
    For every $\epsilon>0$ there exists some $\delta>0$ so that the following holds.
 Let $\{\gamma_i\}$ be a collection of closed contours with areas $A(\gamma_1)\geq A(\gamma_2)\geq \ldots$, and suppose
   \[
   \sum_i|\gamma_i|\le (1+\delta)4L\quad \mbox{and}\quad\sum_i A(\gamma_i)\ge (1-\delta)L^2\,.
   \]
   Then the interior of $\gamma_1$ contains a square of area at least $(1-\epsilon)L^2$.
\end{lemma}
\begin{proof}
Observe that $\sum\sqrt{a_i} \geq (\sum a_i)/(\max_j\sqrt{a_j})$ holds for any $a_1,\ldots,a_n \in \R_+$, which
together with the $\Z^2$ isoperimetric bound $A(\gamma)\leq |\gamma|^2/16$ yields
\[
   (1+\delta)4L\geq \sum_i|\gamma_i|\geq
   4\sum_i \sqrt{A(\gamma_i)} \geq
   4\frac{\sum A(\gamma_i)}{\sqrt{A(\gamma_1)}} \geq \frac{4(1-\delta)L^2}{\sqrt{A(\gamma_1)}}\,.
\]
Rearranged, $A(\gamma_1) \geq (\frac{1-\delta}{1+\delta})^2 L^2$, and the result now follows from continuity since the square is the unique shape
in $\Z^2$ with area at least $1$ and perimeter at most $4$.
\end{proof}
\begin{proof}[\emph{\textbf{Proof of Proposition~\ref{prop:blackbox}}}]
Set $\delta = 2/\beta$ and let $\cB$ be the event under consideration, i.e., that there exists an $h$-contour $\gamma$ such that $V_\gamma \supset \Lambda_{(1-\epsilon)\ell}$.
  Further let $\cI$ be the set of all contours $\gamma$ that satisfy
  \[\mbox{ either }\quad|\gamma|>(4+3\delta)\ell\quad \mbox{ or } \quad
  \left\{\begin{array}{l} \quad|\gamma|>\log^2\ell \\
   A(\gamma)<(1-3\delta)\ell^2
  \end{array}\right.\,.
  \]
  By Lemma~\ref{lem:isop}, the combination of $|\gamma|\leq (4+3\delta)\ell $ and $A(\gamma)\geq (1-3\delta)\ell^2 $ implies $\cB$ provided that $\beta$ is large enough (and hence $\delta$ is small enough). Thus, if $\cB^c$ occurs then either there is no macroscopic $h$-contour $\gamma$, or some $\gamma\in\cI$ is such an $h$-contour, so
  \begin{align}\label{eq-B}
  \pif_{V}^{h-1}(\cB^c) \leq  \pif_{V}^{h-1}\left(\mbox{$\bigcup_{\gamma\in \cI}$} \sC_{\gamma,h} \right)  +  \pif_{V}^{h-1}\left(\mac_h^c \right)\,.
  \end{align}
  For the first term in~\eqref{eq-B}, we use Proposition~\ref{prop:h-contour-upper}.
If $(4+3\delta)\ell < |\gamma| < 10\ell$ then
    \begin{align*}
 \pif_{V}^{h-1}(\sC_{\gamma,h}) &\leq
\exp\left[-\beta(4+3\delta)\ell +\pi(\eta_0\geq h){\bar\ell}^2 + e^{-(\frac{\pi\beta}2+o(1))h^2/\log h}\ell\log\ell\right] \\
      &\leq \exp\left[-\left(2- O\Big(\frac{\log^2\ell}\ell\Big)- e^{-(\frac{\pi\beta}2+o(1))\frac{h^2}{\log h}}\log\ell\right) \ell  \right]
       = e^{-(2-o(1))\ell} \,,
    \end{align*}
where the second inequality used the upper bound on $\pi(\eta_0\geq h)$ from our hypothesis, while the last inequality
used the fact that $\exp(-h^{-2+o(1)}) < (\log\ell)^{-10}$ for large $\ell$ (with room to spare).
This clearly outweighs the total of $O(\ell^2 3^{10\ell})$ possible such $\gamma$, yielding an overall estimate of $\exp(-(2-o(1)) \ell)$ for $\bigcup\{ \sC_{\gamma,h}:(4+3\delta)\ell<|\gamma|<10\ell\}$.

Whenever $|\gamma| > 10\ell$ we can break the factor $\exp(-\beta|\gamma|)$ into two equal parts, utilizing one as above and the other to help with the enumeration over the contours $\gamma$; namely,
    \begin{align*}
 \pif_{V}^{h-1}(\sC_{\gamma,h}) &\leq
\exp\left[-\left(\beta/2 - e^{-(\frac{\pi\beta}2+o(1))h^2/\log h}\log\ell\right)|\gamma|\right] e^{-\left(\beta-4\right) \ell  }  \leq e^{-\left(\beta/2 - o(1)\right)|\gamma|}\,,
    \end{align*}
and so $\sum_{k\geq 10\ell} \sum_{|\gamma|\geq k} \pif_{V}^{h-1}(\sC_{\gamma,h}) = O(\exp(-\ell))$, say, when $\beta$ is large.

Finally, if $|\gamma|>\log^2\ell$ and $A(\gamma) < (1-3\delta)\ell^2$ we write $A(\gamma) < \sqrt{1-3\delta}\, \bar\ell |\gamma|/4$, yielding
    \begin{align*}
 \pif_{V}^{h-1}(\sC_{\gamma,h}) &\leq
      \exp\left[-\beta\left(1-\frac{\pi(\eta_0\geq h)\sqrt{1-3\delta}\,\bar\ell}{4\beta} + e^{-(\frac{\pi\beta}2+o(1))h^2/\log h}\log\ell\right) |\gamma|  \right] \\
      &\leq \exp\left[-\beta\left(1-(1+\delta/2)\sqrt{1-3\delta}+ o(1)\right) |\gamma| \right]
      = e^{-\beta\left(\delta-O(\delta^2)\right)|\gamma|}\,,
    \end{align*}
with the second inequality again stemming from our upper bound on $\pi(\eta_0\geq h)$. This is equal to $\exp[-(2-\epsilon_\beta)|\gamma|]$ where $\epsilon_\beta=O(1/\beta)$ and so,
for large $\beta$, this easily outweighs the enumeration over the contour $\gamma$ including its starting position (since $|\gamma|>\log^2\ell$).

Altogether we have shown that $\pif_{V}^{h-1}(\bigcup_{\gamma\in\cI} \sC_{\gamma,h}) < O(e^{-\log^2\ell})$
and can now turn our attention to the second term in~\eqref{eq-B}.
Given our boundary conditions at height $h-1$, if there are no macroscopic $h$-contours and yet there are macroscopic contours for some $h'\neq h$ then there necessarily must exist some macroscopic negative contour. This, in turn, has probability $O(\exp(-\log^2\ell))$ for large enough $\beta$ thanks to~\eqref{e:negativeContour}; thus,
\begin{equation}
   \label{eq-no-macro1}
\pif_{V}^{h-1}\left(\mac_h^c\right) = \pif_{\Lambda_\ell}^{h-1}\left(\mac_*^c\right) + \pif_{V}^{h-1}(\mac_* \setminus \mac_h) \leq \pif_{V}^{h-1}\left(\mac_*^c\right)+
O\big(e^{-\log^2 \ell}\big)\,.
 \end{equation}
To estimate $\pif_{V}^{h-1}(\mac_*^c)$ we consider whether the event $\cP=\cP^{\neq h-1}_{\log^2L}$ from~\eqref{eq-def-P-event} (a path along which $\eta \neq h-1$  connecting points at distance at least $\log^2 L$ in $\Lambda_L$) occurs or not, abbreviating it here by $\cP$. By Lemma~\ref{lem:neq-j-path}, $\pif_{V}^{h-1}\big(\mac_*^c~,~ \cP\big) = O(e^{-\log^2\ell})$, so
\begin{align}\label{eq-no-macro2}
\pif_{V}^{h-1}\left(\mac_*^c\right) \leq \pif_{V}^{h-1}\left( \cP^c\right) + O\big(e^{-\log^2\ell}\big)\,.
\end{align}
It remains to assess the probability of $\cP^c$, to which end we will leverage Proposition~\ref{prop:grill}.
Put $\cZf_{V}^{j}$ for the partition function restricted to configurations on $V$ with boundary conditions $j$ (we omit this superscript when $j=0$ and there is no ambiguity) and a floor at 0, and similarly for $\cZ_V^j$ (in the absence of a floor), whence
\begin{align*}
  \pif_{V}^{h-1}\left(\cP^c\right) &= 
  \frac{\cZ^{h-1}_{V} \, \pi^{h-1}_{V}\left(\eta \geq 0 ~,~ \cP^c\right) }{\cZf^{h-1}_V}
  =  \frac{\cZ_{V}}{\cZf^{h-1}_V} \pi^{h-1}_{V}\left(\eta \geq 0 ~,~ \cP^c\right)\,,
\end{align*}
where $\cZ_V^{h-1}=\cZ_V^{0}$ due to translation invariance.
Observe that if $V_\star = V \setminus \partial V$ (the subset of $V$ excluding the sites adjacent to its boundary) then
 $\cZf_V^{h-1} \geq e^{-\beta|\partial V|} \cZf_{V_\star}^h$ by restricting our summation to configurations with value $h$ along $\partial V$. Thus,
\begin{align*} \cZf_V^{h-1} &\geq e^{-\beta|\partial V|} \,\cZ_{V_\star}^h \, \pi^h_{V_\star}(\eta \geq 0) \geq
e^{-\beta|\partial V|} \, \cZ_{V_\star} \, \prod_{x\in V_\star}\pi_{V_\star}(\eta_x \geq h+1)  \\
&\geq
\exp\left[-\beta|\partial V| - \pi(\eta_0 \geq h+1)|V_\star|^2 - e^{-(\frac{\pi\beta}2+o(1))\frac{h^2}{\log h}}\ell\log\ell \right] \cZ_{V_\star} \,,
\end{align*}
where the second inequality is by FKG and the $\ell\log\ell$ error term arises due to points close to $\partial V_\star$ where the approximation of $\pi_{V_\star}$ via the infinite-volume measure $\pi$ fails (exactly as in the proof of Proposition~\ref{prop:h-contour-upper}). In our situation $h \geq\log\log\ell$ for large $\ell$ (our hypothesis~\eqref{eq-l,h-relation}, in view of Theorem~\ref{th:p(h)}, in fact shows that $h= (\log \ell)^{1/2+o(1)}$), making the pre-factor of $\ell\log\ell$ in the above exponent be less than, say, $(\log\ell)^{-10}$. Moreover, $|V_\star| = \ell^2 + O(\ell\log^2\ell)$ (being sandwiched between $\Lambda_\ell\setminus \partial\Lambda_\ell$ and $\Lambda_{\bar\ell}\setminus \partial\Lambda_{\bar\ell}$) whereas $\pi(\eta_0 \geq h+1) = \ell^{-1+o(1)}$, and from the last two estimates we now get that
\begin{align*}
  \pif_{V}^{h-1}\left(\cP^c\right) &\leq \frac{\cZ_{V}}{  \cZ_{V_\star}} \exp\Big[\beta|\partial V| + \pi(\eta_0 \geq h+1)\ell^2 + o(\ell) \Big]   \pi^{h-1}_{V}\left(\eta \geq 0 ~,~ \cP^c\right)\,.
\end{align*}
The last term is handled by Proposition~\ref{prop:grill}, according to which this probability is at most
$\exp\left[-(1-o(1))\pi(\eta_0\geq h) \ell^2 \right] $. Finally, it is well-known (see, e.g.,~\cite{BW}) that $\cZ_V \leq \cZ_{V_\star} \exp(\epsilon_\beta |\partial V|)$ since the cluster-expansion of these partition functions agrees everywhere except on clusters incident to $\partial V$, whose contribution to the partition function is $\exp(\epsilon_\beta)$ provided $\beta$ is large (this can alternatively be seen by forcing the configuration of $\eta\sim\pi_V$ to be 0 along $\partial V$ at a cost of $\exp(-\epsilon_\beta |\partial V|)$).
Altogether,
\begin{align*}
  \pif_{V}^{h-1}\left(\cP^c\right) &\leq \exp\Big[- \big((1-o(1))\pi(\eta_0 \geq h)- \pi(\eta_0 \geq h+1)\big)\ell^2 + \left (\beta + \epsilon_\beta\right)|\partial V| +o(\ell)\Big]\\
  & \leq \exp\Big[-(1-o(1))\pi(\eta_0 \geq h)\ell^2 + (4\beta + \epsilon'_\beta)\ell\Big]\,,
\end{align*}
where for the inequality in the second line we used $\pi(\eta_0\geq h+1) \ll \pi(\eta_0 \geq h)$
and $|\partial V| \leq (4+e^{-\beta})\ell$. The lower bound on $\pi(\eta_0\geq h)$ now implies that
\begin{align*}
  \pif_{V}^{h-1}\left(\mac_*^c\right) &\leq \exp\Big[- \left(2 - \epsilon'_\beta - o(1)\right) \ell \Big]\,,
\end{align*}
and revisiting~\eqref{eq-B}--\eqref{eq-no-macro2} we conclude that $\pif^{h-1}_{V}(\cB^c)=O(e^{-\log^2\ell})$, as required.
\end{proof}

\begin{lemma}\label{lem:encapsulate}
Let $V$ be a region containing the square $\Lambda_\ell$, fix $\beta$ large enough and set $\bar\ell = \lceil \ell + \log^2\ell\rceil$. Let
$\cQ_\ell$ denote
  the event that $\eta\sim\pi_V$ admits a circuit of sites $C$ with
  \[ \eta\restriction_C = 0\,,\qquad \Lambda_\ell \subset V_C \subset \Lambda_{\bar\ell}\,,\qquad |C| \leq \big(1 + e^{-\beta}\big)4\ell\,.\]
  Then $\pi_{V}(\cQ_\ell) =1-O(e^{-\log^2\ell})$.
\end{lemma}
\begin{proof}
As already used above, the probability that a given $\gamma$ is an external-most contour (positive or negative) in $\eta\sim\pi_{V_{\varphi}}$ is at most $\exp(-\beta|\gamma|)$.
Hence, the probability
that $\Lambda_\ell$ is surrounded by a positive or negative external-most contour $\gamma$ (which must then satisfy $|\gamma| \geq  4\ell$ as well as intersect the $x$-axis of the bottom face of $\Lambda_\ell$ at distance at most $|\gamma|/2$ to its right, for instance) is at most
\[ 2\sum_{|\gamma|\geq 4\ell} \frac{|\gamma|}2 3^{|\gamma|} e^{-\beta|\gamma|} = O\big(e^{-\ell}\big)\]
for large enough $\beta$ (here the first factor of 2 accounted for the sign of $\gamma$).

Similarly, setting $\delta=e^{-\beta}$, the probability that $\partial \Lambda_\ell$ is incident to any collection of external-most contours (positive or negative) of total length at least $\delta\ell$ is
at most
\begin{align*}
\sum_{k\geq1} 2^k &\sum_{\substack{\gamma_1,\ldots,\gamma_k\\ \sum|\gamma_i|\geq \delta\ell}} e^{-\beta\sum \gamma_i} \leq
e^{-\frac\beta2 \delta\ell} \sum_{k\geq 1} \binom{4\ell}k \bigg(2 \sum_{r\geq 4} \left(3e^{-\frac\beta2}\right)^r\bigg)^k\,,
\end{align*}
where the restriction $r\geq 4$ comes from the minimal length of a closed contour $\gamma_i$. For large $\beta$ the inner summation over $r$ is at most $ c e^{-2\beta }$
and the entire summation over $k$ is at most $\exp\big[c'  e^{-2\beta} \ell\big]$, translating the above estimate into $\exp\big[ -\big(\frac\beta2 \delta - c'e^{-2\beta}\big)\ell\big]$.
By our choice of $\delta=e^{-\beta}$ we see that the pre-factor of $\ell$ is positive for large enough $\beta$.

The fact that $\eta\equiv 0$ outside of its external-most contours
implies that one can form $C$ by following $\partial \Lambda_\ell$ while detouring around the external-most contours it intersects, so that $|C| \leq 4\ell +e^{-\beta}\ell$ with probability $1-O(\exp(c \ell))$ for some $c(\beta)>0$.
Moreover, for $C$ defined in this way to step beyond the box $\Lambda_{\bar\ell}$ we must find an external-most contour incident to $\partial \Lambda_\ell$ whose length is at least $\log^2\ell$, an event whose probability is $O(\ell e^{-\beta\log^2\ell})$ under $\pi_{V_\varphi}$.
This concludes the proof.
\end{proof}

\subsection{Proof of Theorem~\ref{mainthm:floor-shape}}
Set $H=H(L)$ as in~\eqref{eq-H-def} to be the maximum integer such that $\pi(\eta_0\geq H)\geq 5\beta/L$. Observe that by~\eqref{eq-ratio-p(h)/p(h-1)}--\eqref{eq-p(h)} we have
\begin{align*}
  H-1 &\geq \frac{\exp[(\log L)^{1/2-o(1)}]}L\,,\qquad H+2 \leq \frac{\exp[-(\log L)^{1/2-o(1)}]}L\,.
\end{align*}
Next, define
\[ \ell = \left\lfloor\frac{4\beta + 3}{\pi(\eta_0 \geq H-1)}\right\rfloor\qquad\left( = L^{-1+o(1)}\right)\,,\]
and note that $\ell$ and $h=H-1$ satisfy the relation~\eqref{eq-l,h-relation} for large enough $L$ (the lower bound holds provided $\pi(\eta_0\geq h) $ is small enough, our case here as $H\to\infty$ with $L$).
We will sequentially show a high probability for the event $\cR_j$ ($j=0,\ldots,H-1$) given by 
\[ \cR_j = \left\{ \exists \mbox{ a circuit of sites $C$ }:\; \eta\restriction_C\geq j \,,\; V_{C} \supset \Lambda_{L-j \ell}\right\}\,.\]
Of course, $\pif_{\Lambda_L}(\cR_0)=1$, and therefore it will hence suffice to show that
\begin{equation}
  \label{eq-(j-1)-to-j}
  \pif_{\Lambda_L}(\cR_j^c\,,\,\cR_{j-1}) = O\big(e^{-\log^2\ell}\big)\quad\mbox{ for any $j=1,\ldots,H-1$}
\end{equation}
in order to deduce $\cap_{j<H} \cR_j$ via a union-bound over the $(\log L)^{1/2+o(1)}$ possible $j$'s.

To prove~\eqref{eq-(j-1)-to-j}, expose all the external-most circuits $C_0$ in $\Lambda_L$ where $\eta\restriction_{C_0} \geq j-1$. The event $\cR_{j-1}$ says that the area of (precisely) one of these circuits of sites will be at least $[L-\ell (j-1)]^2=(1-o(1))L^2$. Crucially, on this event, our only information on the configuration in the interior of this circuit $C_0$ is that $\eta\restriction_{\partial V_{C_0}}\geq j-1$.

Next, consider some square $\Lambda_\ell \subset V_{C_0}$. We wish to find a circuit of sites $S$ tightly encapsulating $\Lambda_\ell$ such that $\eta\restriction_{S}\geq j-1$. To this end, by monotonicity we can drop the floor, and further set the boundary conditions on $V_{C_0}$ to be exactly $j-1$. An application of Lemma~\ref{lem:encapsulate} now finds that with probability $1-O(e^{-\log^2\ell})$ the event $Q_\ell$ holds, i.e., there exists such an $S$ (in fact, one satisfying $\eta\restriction_S=j-1$) for which
\begin{equation}\label{eq-S-chain} |S|\leq (1+e^{-\beta})4\ell\,,\quad \Lambda_\ell \subset V_S \subset \Lambda_{\ell+\log^2\ell}\,.\end{equation}
Back in the setting of $\pif_{\Lambda_L}$ and a given $\Lambda_\ell$, condition on the external-most such circuit $S$ within the bigger box $\Lambda_{\ell+\log^2\ell}$ satisfying~\eqref{eq-S-chain}, guaranteed to exist with probability $1-O(e^{-\log^2\ell})$. (As before, this reveals no information on the interior of $V_S$.)

Our next goal is to find a large circuit of sites $C_1$ in $V_S$ such that $\eta\restriction_{C_1} \geq j$ and $V_{C_1} \supset \Lambda_{(1-\epsilon)\ell}$ for some small $\epsilon>0$.
For this purpose, again by monotonicity, we may drop the floor to height $j-(H-1)$ (thus translating the distribution on $V_S$ to $\pif_{V_S}^{h-1}$ for $h=H-1$). The aforementioned properties of $S$ now justify an application of Proposition~\ref{prop:blackbox}, which shows that the sought $C_1$ exists with probability $1-O(e^{-\log^2 \ell})$.

Recalling that $\ell=L^{-1+o(1)}$, the aforementioned probabilities of $O(e^{-\log^2 \ell})$ support a union bound over all possible locations for the box $\Lambda_\ell\subset \Lambda_L$. Clearly, for each pair of such boxes with a side-length overlap of $\ell/2$, the two respective circuits must intersect, and altogether we obtain the following: If $V_{C_0} \supset \Lambda_r$ for some $r$, then there is a single circuit along which $\eta\geq j$ whose interior contains $\Lambda_{r-\ell}$ (the outer frame of width $\ell/2$ in $\Lambda_r$ was waved in this argument). By the definition of the event $\cR_j$ we can take $r$ to be $L-(j-1)\ell$, and~\eqref{eq-(j-1)-to-j} now follows.

So far we have shown that with probability $1-O(\log^2\ell)$ the event $\cR_{H-1}$ occurs, i.e., there is a single circuit $C$ encapsulating an area of $(1-o(1))L^2$  such that $\eta\restriction_C \geq H-1$.
To get from level $H-1$ to level $H$ we apply a similar strategy, except now the designated $\ell$ we choose will satisfy~\eqref{eq-l,h-relation} w.r.t.\ $h=H$.
Recalling that $L \pi(\eta_0 \geq H) \geq 5\beta$, starting at $\ell=L$ and repeatedly decreasing $\ell$ by $1$ modifies the right-hand side that was initially $5\beta$ by $L^{-1+o(1)}$ in each step, and so certainly it is feasible to find such an $\ell$, which will range from about $\frac45 L$ (when $\pi(\eta_0\geq H)$ is close to $5\beta/L$) to about $L/e^{-c H/\log H}=L^{1-o(1)}$.
The conclusion is now that there exists a single circuit $C$ such that $\eta\restriction_C\geq H$ and $V_C\geq(1-\epsilon)L^2$, where $\epsilon$ can be made arbitrarily small provided that $\beta$ is large enough.
We have thus proved that w.h.p.\ the configuration $\eta\sim\pif_{\Lambda_L}$ contains an $(H-1)$-contour of area $(1-o(1))L^2$ and an $H$-contour of area at least $(1-\epsilon)L^2$.

  As for level $H+2$, by definition  $\pi(\eta_0 \geq H+1) < 5\beta/L $, and it follows from Eq.~\eqref{eq-ratio-p(h)/p(h-1)} in Theorem~\ref{th:p(h)}
  that
  \[ \pi(\eta_0 \geq H+2) = o(1/L)\,.\]
  Further note that $H \geq \sqrt{\log L}$ for large enough $L$, whereas $\log|\gamma| = O(\log L)$, and so the last term in~\eqref{e:contourFloorBound} is $o(|\gamma|)$.
  The fact that $A(\gamma) \leq |\gamma|L/4$ then implies that
  \[ \pif_{\Lambda_L}(\gamma) \leq e^{-(\beta-o(1))|\gamma|}\,,\]
  and summing over all macroscopic contours $\gamma$ rules out the event $\mac_{H+2}$ except with the usual probability of $O(e^{-\log^2\ell})$. Similarly, within the aforementioned $H$-contour there are no macroscopic negative contours, as again each such potential $\gamma$ has a probability of $e^{-\beta|\gamma|}$.
  The proof is therefore completed by Lemma~\ref{lem:negative,H+2}.
  \qed

\section{Extensions to other random surface models}\label{sec:other-p}

In this section we extend our results on the DG to all values of $1< p \leq \infty$ including the restricted solid-on-solid model $(p=\infty$).  In order to use the proof from Section~\ref{sec:floor} we need to establish the analogues of equations~\eqref{eq-ratio-p(h)/p(h-1)}, \eqref{eq-p(h)} and~\eqref{eq-p(h,h)} for the asymptoics of  $\pi\left(\eta_0 \geq h\right)$, $\frac{\pi\left(\eta_0 \geq h\right)}{\pi\left(\eta_0 \geq h-1\right)}$ and $\pi\left(\eta_z\geq h \mid \eta_0 \geq h\right)$.

\subsection{Between SOS and the Discrete Gaussian ($1<p<2$)}\label{sec:1<p<2}

We begin with the case of $1 < p < 2$ in which large deviations of the surface are formed by thin spikes which are of a constant width for most of their height but unlike the case of $p=1$ have a growing width at their base.
\begin{theorem}\label{thm:max1p2}
Let $1<p< 2$ and fix $\beta= \beta(p) > 0$ large enough.  Then there
exists $c_p>0$ such that, for $\eta$ given by the infinite volume $p$-SOS model in $\Z^2$ at inverse-temperature $\beta$,
 \begin{equation}\label{e:tailRate1p2}
\pi\left(\eta_0 \geq h\right) = \exp\left( -(c_p \beta +o(1))h^p \right)\,.
 \end{equation}
With $X_L$ denoting the maximum  on an $L\times L$ box with 0 boundary condition there is a sequence $M=M(L)$ with
 \begin{equation}
   \label{eq:E[XL]1p2}
   M(L) \sim \bigg(\frac{2+o(1)}{c_p\beta}\log L\bigg)^{1/p}
 \end{equation}
such that $X_L \in \{ M,M+1\}$ with probability going to 1 as $L\to\infty$.
Furthermore, there exists some $H=H(L)$ with $H \sim \big(\frac{1+o(1)}{c_p\beta}\log L\big)^{1/p}$ such that w.h.p.\
\begin{equation}\label{eq-height-conc1p2}
 \#\big\{v : \eta_v \in \{H,H+1\}\big\} \geq (1-\epsilon_\beta) L^2\,,
\end{equation}
where $\epsilon_\beta$ can be made arbitrarily small as $\beta$ increases.
 \end{theorem}
 \begin{proof}
Using a standard Peierls-type argument --- a straightforward adaptation of the proof of~\cite{BW} --- we have
\begin{equation}\label{e:CrudeMax}
\pi(\eta_0 \geq h) \leq C \exp(-4\beta h)\,.
\end{equation}
We first control the size of  $|\Gamma_1|$, the outermost $1$-contour encircling the origin.
  Suppose that $\eta_0=h$ and $\eta_x \leq -h$ for some $x\sim 0$.  Then there must exist nested negative $(-i)$-contours $\gamma_i(\eta)$ for $1\leq i \leq h$ which each contain $x$ but not 0.  Consider the map
\[
\left(T_{\{\gamma_i(\eta)\}} \sigma\right)_y = \sigma_y + \sum_{i=1}^h \ind_{\{y\in V_{\gamma_i}\}}\,.
\]
Applying $T_{\{\gamma_i(\eta)\}}$ to $\eta$, the Hamiltonian decreases by at least one at every point along each $\gamma_i$, and also by at least $h^p$ along the bond from $0$ to $x$, so
\[
\pi(T_{\{\gamma_i(\eta)\}} \eta) \geq \pi(\eta) e^{\beta (h^{p}-h+\sum_{i=1}^h |\gamma_i| )}\,.
\]
Therefore,
\[
\P\left(\eta_x \leq - h \mid \eta_0 = h \right) \leq \sum_{\{g_i\}} \P\left(\eta_0 = h\right) e^{-\beta (h^p - h + \sum_{i=1}^h |g_i| )} \leq \P\left(\eta_0 = h\right) e^{-\beta h^{p}}\,.
\]
Next, let $T'$ denote the map
\[
\left(T_{\gamma} \eta\right)_y = \eta_y  - \ind_{\{y\in V_{\gamma}\}} + \ind_{\{y=0\}}\,,
\]
applied when $\eta_0=h$ and $\gamma=\Gamma_1$.  This map forces down the outermost 1-contour and then raises the origin by 1 (overall leaving the origin at $h$, unchanged). Then for $\eta$ with $\eta_0=h$ and $\min_{x\sim 0} \eta_x \geq - h$,
\[
\pi(T_{\Gamma_1}(\eta)) \geq \pi(\eta) e^{-\beta |\Gamma_1| +4\beta p (2h)^{p-1}}\,,
\]
and hence for large enough $C(\beta,p)$ we have that
\begin{align}\label{e:radius}
\pi(|\Gamma_1| > Ch^{p-1} \mid \eta_0 \geq h ) & \leq \pi(\min_{x\sim 0} \eta_x  \leq - h \mid \eta_0 = h ) + \sum_{\gamma:|\gamma|\geq Ch^{p-1}}  e^{-\beta |\gamma| +4\beta p (2h)^{p-1}}\nonumber\\
& \leq  e^{-\beta h^{p-1}}\,.
\end{align}
Now we define $\phi^*:\Z^2\to\R$ to be the unique minimizer of
\[
E(\phi)=\sum_{x\sim y} |\phi_x - \phi_y|^p
\]
subject to $\phi_0=1$ and $\lim_{|x|\to\infty} \phi_x= 0$ (see, e.g.,~\cite{Soardi}*{pp.176--178}).
\begin{lemma}\label{l:Al12PH}
For any $\epsilon>0$ and large enough $h$,
\[
e^{-\beta (E(\phi^*)+\epsilon) h^p} \leq \pi(\eta_0=h) \leq e^{-\beta (E(\phi^*)-\epsilon) h^p}.
\]
\end{lemma}
\begin{proof}
Fix $R=C h^{p-1}$ so that Eq.~\eqref{e:radius} holds.
As in the proof of Corollary~\ref{cor:2} we have that $\pi(\eta_{B_R}=0) \geq e^{-R^2}$.  For large enough $R$ we can find a finitely supported $\phi_f$ such that the support of $\phi_f$ is contained in $B_{R-1}$ and $E(\phi_f) - E(\phi^*) \leq \epsilon/2$.  Then
\[
\pi(\eta_0=h) \geq\pi\left(\eta_{B_R}=\lfloor \phi_f h \rfloor\right) =  \pi(\eta_{B_R}=0) e^{-\beta E(\lfloor \eta_f h \rfloor)} \geq e^{-\beta (E(\phi^*)+\epsilon) h^p}\,.
\]
For the upper bound, by Eq.~\eqref{e:radius} we use the fact that we can lower bound the energy by $h^{p} E(\phi^*)$.  We also know that given $\eta_0 \geq h$ w.h.p.\ there exists a circuit of radius at most $R$ around the origin on which $\eta$ is non-positive.  Hence, by monotonicity,
\begin{align*}
\pi(\eta_0 \geq h) &\leq \frac{\max_{\Lambda\subset B_R} \pi_{\Lambda}\left( \eta_0 \geq h\right)}{\pi\left(|\Gamma_1| \leq R \mid \eta_0 \geq h\right)}  \leq 2\max_{\Lambda\subset B_R} \pi_{\Lambda}\big(\max_{x\in\Lambda} \eta_x \geq h^2\big) \!+ 2e^{-\beta h^{p} E(\phi^*)}(2h^2+1)^{|\Lambda|}\\
 &\leq  e^{-\beta (E(\phi^*)-\epsilon) h^p}\,,
\end{align*}
where we used Eq.~\eqref{e:CrudeMax} to bound the probability that it exceeds $h^2$, that $h^{p} E(\phi^*)$ is a lower bound on the energy and the fact that $|\Lambda| = O(h^{2(p-1)})$.
\end{proof}

\begin{lemma}\label{e:consecutiveHeight1p2}
There exists $c(\beta,p)>0$ such that,
\[
\frac{\pi(\eta_0=h)}{\pi(\eta_0=h-1)} \leq e^{-c\beta h^{p-1}}\,.
\]
\end{lemma}
\begin{proof}
It is easy to see that for all $x\neq 0$ the value of $\phi_x^*$ must be strictly less than the maximum of its neighbours.  Let $\kappa= 1 - \max_{x\sim 0} \phi^*_x >0$.  By the uniqueness of $\phi^*$, for some $\delta >0$ we have
\[
\sup_{\substack{\phi: \phi_0=1\\ \max_{x\sim 0} \phi_x > 1-\kappa/2}} E(\phi) \geq E(\phi^*)+\delta\,,
\]
where the supremum is over all finitely supported $\phi$.  Similarly to Lemma~\ref{l:Al12PH}
\begin{align*}
\pi\left(\eta_0 \geq h ~,~\max_{x\sim 0} \eta_x \geq (1-\kappa/2)h\right) &\leq \pi\left(|\Gamma_1| \geq R\mid \eta_0\geq h\right)\pi(\eta_0\geq h)  \\
 + 2\max_{\Lambda\subset B_R} &\pi_{\Lambda}\left(\max \eta_x\in\Lambda \geq h^2\right) + 2e^{-\beta h^{p} (E(\phi^*)+\delta)}(2h^2+1)^{|\Lambda|}\\
 &\leq  e^{-\beta  h^{p-1}} \pi(\eta_0\geq h)\,.
\end{align*}
Hence, by considering the map $T(\eta)(x) = \eta_x - \ind_{\{x=0\}}$ we have that, whenever $\eta_0=h$ and $\max_{x\sim 0} \eta_x < (1-\kappa/2)h]$,
\[
\pi(T\eta) \geq e^{\beta p (\kappa/2 h^{p})^{p-1}} \pi(\eta)\,,
\]
and so
\begin{align*}
\pi(\eta_0 \geq h) &\leq \pi\left(\eta_0 \geq h~,~ \max_{x\sim 0} \eta_x \geq (1-\kappa/2)h\right) + e^{-\beta p (\kappa/2 h^{p})^{p-1}}\pi(\eta_0 \geq h-1)\\
&\leq e^{-c\beta h^{p-1}}\pi(\eta_0 \geq h-1)\,.\qedhere
\end{align*}
\end{proof}

The following lemma is the analogue of equation~\eqref{eq-p(h,h)} in Theorem~\ref{th:p(h)}.
\begin{lemma}\label{l:phh1p2}
There exists $c(\beta,p)>0$ such that for any $z\in \Z^2$,
\[
\pi(\eta_z= h\mid \eta_0= h) \leq e^{-c\beta h^{p-1}}\,.
\]
\end{lemma}

\begin{proof}
The proof is similar to the proof of equation~\eqref{eq-p(h,h)} in Theorem~\ref{th:p(h)} where we give more detailed explinations.
Fix $z\in \bbZ^2$ and let
\[
X:=\max_{x\sim
  z}\eta_x\,,\, \quad Y(\eta):=\min_{x\sim
  z}\eta_x\,.
\]
Given $0<\delta\le 1$, define the events $F= \{X\le
h\}$ and $E=\{Y\ge h-\delta h^{\frac{p-1}{p}}\}$. Similarly to before using Lemma~\ref{e:consecutiveHeight1p2} that $\pi(F^c\mid \eta_0=h)\le  O\left(e^{-c_1 \beta h^{p-1}}\right)$.
Therefore, it will suffice to establish a similar upper bound on $\pi(\eta_z=h\mid \eta_0=h\,,\, F)$.  Conditioning over the values of the
neighbors of $z$ and then using monotonicity yields
\[
\pi(\eta_z=h\mid \eta_0=h\,,\, E^c\,,\, F)\le
e^{-c' \delta^p h^{p-1}}\,.
\]
Finally, we will bound $\pi(E\mid \eta_0=h\,,\, F)$ from above as follows.
On one hand we have
\[
\pi\left(\eta_z\ge h+1\mid \eta_0=h\,,\, E\,,\, F\right)\ge e^{-4c_2 \beta \delta^p h^{p-1}}\,,
\]
while
\begin{align*}
\pi\left(\eta_z\ge h+1\mid \eta_0=h\,,\, E\,,\,F\right)
&\le
\frac{\pi(\eta_z\ge h+1\mid \eta_0=h)}{\pi\left(E\mid \eta_0=h\,,\, F\right)}
\le \frac{(1+o(1))e^{-c_1\beta h^{p-1}}}{\pi\left(E\mid \eta_0=h\,,\, F\right)}\,,
\end{align*}
Combining the last two displays gives
\[
\pi\left(E\mid \eta_0=h\,,\, F\right)\le (1+o(1))e^{- \beta (c_1- 4c_2\delta^p)h^{p-1}}\,,
\]
and the proof is completed by choosing $c_2 \delta^p< c_1/4$.
\end{proof}

The proof of Theorem~\ref{thm:max1p2} now follows.  Equation~\ref{e:tailRate1p2} follows by Lemma~\ref{l:Al12PH}.  Having bounded the tails of the  height distribution together with the estimates in Lemmas~\ref{e:consecutiveHeight1p2} and~\ref{l:phh1p2}, the size of the maximum height follows from essentially the same proof as Theorem~\ref{mainthm:max}.  Finally the height of the surface of the SOS model with a floor is given by essentially the same proof as Theorem~\ref{mainthm:floor-shape}.

 \end{proof}
%

\subsection{Between the Discrete Gaussian and Restricted SOS ($2<p<\infty$)}\label{sec:2<p<inf}
We establish similar results now for $2<p<\infty$.
\begin{theorem}\label{thm:max2pInf}
For $2<p< \infty$ fix $\beta= \beta(p) > 0$ large enough.  There exist $c_1,c_2>0$ so that  for $\eta$ given by the infinite volume $p$-SOS model in $\Z^2$ at inverse-temperature $\beta$,
 \begin{equation}\label{e:tailRate2pInf}
e^{-c_1 \beta h^2} \leq \pi(\eta_0=h) \leq e^{-c_2 \beta h^2}\,.
 \end{equation}
Letting $X_L$ denote the maximum  on an $L\times L$ box with 0 boundary condition, there is a sequence $M=M(L)$ with
 \begin{equation}
   \label{eq:E[XL]2pInf}
   M(L) \asymp\sqrt{\frac{1}{\beta}\log L}
 \end{equation}
such that $X_L \in \{ M,M+1\}$ with probability going to 1 as $L\to\infty$.
Furthermore, there exists some $H=H(L)$ with $H \sim \frac{1+o(1)}{\sqrt{2}} M(L)$ such that w.h.p.\
\begin{equation}\label{eq-height-conc2pInf}
 \#\big\{v : \eta_v \in \{H,H+1\}\big\} \geq (1-\epsilon_\beta) L^2\,,
\end{equation}
where $\epsilon_\beta$ can be made arbitrarily small as $\beta$ increases.
 \end{theorem}
 \begin{proof}
Let $\gamma_1,\ldots,\gamma_h$ be a collection of nested contours containing the origin and let $\Delta_e$ denote the number of $\gamma_i$ that the dual edge $e$ is contained in. Let $E(\{\gamma_i\}) = \sum_e \Delta_e^p$.
\begin{claim}\label{cl:nestedEnergy}
For all $p>2$, there exists $c(p)>0$ such that for all collections of nested clusters $\gamma_1,\ldots,\gamma_h$ containing the origin,
\begin{equation}\label{pGeq2EnergyLowerBound}
E(\{\gamma_i\}) \geq c h^2\,.
\end{equation}
Moreover, for all $c'$ there exists $\epsilon(c',p)>0$ such that if $\gamma_{h/2} \subset B_{\epsilon h}$ then $E(\{\gamma_i\}) \geq c' h^2$.
\end{claim}
\begin{proof}[Proof of the claim]
Let $r_k$ be the maximal distance of $\gamma_{h(1-2^{-k})}$ from the origin.  As $E(\{\gamma_i\}) \geq \sum |\gamma_i|$ it follows that
\begin{equation}\label{e:sumOfOuterContours}
E(\{\gamma_i\}) \geq \frac12 h r_1\,,
\end{equation}
so we may assume that $r_1 = O(h)$.  Let $k_*$ be the $k$ which maximizes $r_k^{2-p} 2^{-kp}$.  Then $\frac{r_{k+1}}{r_k} \geq 2^{\frac{-p}{p-2}}$.  Since $r_1 = O(h)$ it follows that
\[
r_1^{2-p} 2^{-p} > 2^{-\lfloor \log_2 h \rfloor p} \geq r^{2-p}_{\lfloor \log_2 h \rfloor} 2^{-\lfloor \log_2 h \rfloor p}
\]
and hence $k_* <\lfloor \log_2 h \rfloor$.

Note that for all $(1-2^{-k})h\leq i \leq h$ the edges in $\gamma_{i}$ lie inside $B_{r_k}$ and for all $(1-2^{-k})h\leq i \leq (1-2^{-k-1})h$ the contour lengths satisfy $|\gamma_{i}| \geq r_{k+1}$.  Hence we have that
\begin{align*}
E(\{\gamma_i\}) &\geq \sum_{e\in B_{r_k}} \Delta_e^p \geq \max_k \left(\frac{\sum_{i=(1-2^{-k})h}^{(1-2^{-k-1})h} |\gamma_{i}|}{|B_{r_k}|} \right)^p |B_{r_k}|\\
&\geq \max_k \left(\frac{ r_{k+1} 2^{-k-1}h }{4 r_k^2} \right)^p 4 r_{k}^2 = 8^{-p} h^{p} r_k^{2-p} 2^{2-kp} \left(\frac{r_{k+1}}{r_k}\right)^{p} \\
&\geq 8^{-p} 2^{2-\frac{p(p+1)}{p-2}} h^{p} r_{k_*}^{2-p} 2^{-k_*p}
\geq 16^{-p} 2^{2-\frac{p(p+1)}{p-2}} h^{p} r_{1}^{2-p}\,,
\end{align*}
where the second inequality is by Jensen's Inequality.  Combined with Eq.~\eqref{e:sumOfOuterContours} we have
\[
E(\{\gamma_i\}) \geq \max\left\{ \frac12 h r_1, 16^{-p} 2^{2-\frac{p(p+1)}{p-2}} h^{p} r_{1}^{2-p}  \right\}\,.
\]
Taking the infimum of the left hand side over $r_1$ completes the result.
\end{proof}

\begin{lemma}\label{l:pGeqPH}
For each $p,\beta$, there exist constants $c_1,c_2,c_3>0$ such that
\[
e^{-c_1 \beta h^2} \leq \pi(\eta_0=h) \leq e^{-c_2 \beta h^2}
\]
and
\[
\frac{\pi(\eta_0=h)}{\pi(\eta_0=h-1)} \leq e^{-c_3\beta h}\,.
\]
\end{lemma}
\begin{proof}
Similarly to the proof of Corollary~\ref{cor:2} we have that $\pi(\eta_{B_h}=0) \geq e^{-h^2}$.  Then writing $f(x) = (h-|x|_1)\vee 0$,
\[
\pi(\eta_{B_h}=f(x)) =  \pi(\eta_{B_h}=0) e^{-\beta \sum_{j=1}^h (8j+4)} \geq e^{-\beta c_1 h^2}\,.
\]
For contours $\gamma_1,\ldots,\gamma_h$ define
\[
T_{\{\gamma_i\}} (\eta)(y) = \eta_y - \sum_{i=1}^h I(y\in V_{\gamma_i})\,.
\]
Then if $\eta_0=h$ then, $\pi(T_{\{\Gamma_i\}} \eta) \geq e^{\beta E(\{\Gamma_i\})}\pi(\eta)$.  Hence by Claim~\ref{cl:nestedEnergy}
\begin{equation}\label{e:PeierlsBound}
\pi(\eta_0=h)\leq \sum_{\gamma_1,\ldots,\gamma_h}  e^{-\beta E(\{\gamma_i\})} \leq e^{-\tfrac{\beta}{2} c h^2} \sum_{\gamma_1,\ldots,\gamma_h}
e^{-\tfrac{\beta}{2} \sum_i |\gamma_i|}\leq e^{-\beta c_2 h^2}\,.
\end{equation}
Similarly by the second part of the claim, for some $\epsilon>0$
\[
\pi(\eta_0=h, |\Gamma_1|\leq \epsilon h)\leq \sum_{\gamma_1,\ldots,\gamma_h}  e^{-\beta E(\{\gamma_i\})} \leq e^{-2\beta c_1 h^2} \leq e^{-\beta c_1 h^2} \pi(\eta_0=h)\,.
\]
Letting $S_\gamma (\eta)(y) = \eta_y - I(y\in V_{\gamma})$ we have that when $\eta_0\geq 1$, that $\pi(S_{\Gamma_1} \eta) \geq e^{\beta|\Gamma_1|}\pi(\eta)$.  It follows that
\[
\pi(\eta_0=h) \leq 2\pi(\eta_0=h, |\Gamma_1|> \epsilon h) \leq 2 \pi(\eta_0 = h-1)\sum_{\gamma:|\gamma|>\epsilon h} e^{-\beta|\gamma|} \leq e^{-\tfrac{\beta}{2}\epsilon h}\pi(\eta_0 = h-1)\,,
\]
which completes the proof.
\end{proof}

The following lemma is the analogue of equation~\eqref{eq-p(h,h)} in Theorem~\ref{th:p(h)}.
\begin{lemma}\label{l:phh2pInf}
There exists $c(\beta,p)>0$ such that for any $z\in \Z^2$,
\[
\pi(\eta_z= h\mid \eta_0= h) \leq e^{-c\beta h}\,.
\]
\end{lemma}

\begin{proof}
The proof is similar to the proof of equation~\eqref{eq-p(h,h)} in Theorem~\ref{th:p(h)} where we give more detailed explanations.
Fix $z\in \bbZ^2$ and let
\[
X:=\max_{x\sim
  z}\eta_x\,,\, \quad Y(\eta):=\min_{x\sim
  z}\eta_x\,.
\]
Given $0<\delta\le 1$, define the events $F= \{X\le
h\}$ and $E=\{Y\ge h-\delta h^{\frac{1}{p}}\}$. Similarly to before using Lemma~\ref{l:pGeqPH} that $\pi(F^c\mid \eta_0=h)\le  O\left(e^{-c_1 \beta h}\right)$.
Therefore, it will suffice to establish a similar upper bound on $\pi(\eta_z=h\mid \eta_0=h\,,\, F)$.  Conditioning over the values of the
neighbors of $z$ and then using monotonicity yields
\[
\pi(\eta_z=h\mid \eta_0=h\,,\, E^c\,,\, F)\le
e^{-c' \delta^p h}\,.
\]
Finally, we will bound $\pi(E\mid \eta_0=h\,,\, F)$ from above as follows.
On the one hand we have
\[
\pi\left(\eta_z\ge h+1\mid \eta_0=h\,,\, E\,,\, F\right)\ge e^{-4c_2 \beta \delta^p h}\,,
\]
while
\begin{align*}
\pi\left(\eta_z\ge h+1\mid \eta_0=h\,,\, E\,,\,F\right)
&\le
\frac{\pi(\eta_z\ge h+1\mid \eta_0=h)}{\pi\left(E\mid \eta_0=h\,,\, F\right)}
\le \frac{(1+o(1))e^{-c_1\beta h}}{\pi\left(E\mid \eta_0=h\,,\, F\right)}\,,
\end{align*}
Combining the last two displays gives
\[
\pi\left(E\mid \eta_0=h\,,\, F\right)\le (1+o(1))e^{- \beta (c_1- 4c_2\delta^p)h}\,,
\]
and the proof is completed by choosing $c_2 \delta^p< c_1/4$.
\end{proof}

The proof of Theorem~\ref{thm:max2pInf} now follows.  Equation~\eqref{e:tailRate2pInf} follows by Lemma~\ref{l:pGeqPH}.  Together with the bounds from Lemma~\ref{l:phh2pInf} the bounds on the size and concentration of the maximum height follow from essentially the same proof as Theorem~\ref{mainthm:max}.  Finally the height of the surface of the SOS model with a floor is given by essentially the same proof as Theorem~\ref{mainthm:floor-shape}.
 \end{proof}

\subsection{Restricted SOS ($p=\infty$)}\label{sec:RSOS}
Our final result in this section is for the RSOS model, where we recall that any admissible $\eta$ satisfies
$|\eta_x-\eta_y|\in\{0,\pm1\}$ for all $x\sim y$.

\begin{theorem}\label{thm:maxRes}
Fix $\beta>0$ large enough.  There exists $C>0$ such that  for $\eta$ given by the infinite volume restricted-SOS model in $\Z^2$ at inverse-temperature $\beta$,
 \begin{equation}\label{e:tailRateRes}
e^{-4\left(\beta + 2\log\frac{27}{16} + C e^{-\beta}\right) h^2} \leq \pi(\eta_0 =h) \leq e^{-4\left(\beta + 2\log\frac{27}{16} - C e^{-\beta}\right) h^2}\,.
 \end{equation}
With $X_L$ denoting the maximum  on an $L\times L$ box with 0 boundary condition there is a sequence $M=M(L)$ with
 \begin{equation}
   \label{eq:E[XL]Res}
   M(L) \sim (1\pm\epsilon_\beta)  \sqrt{\frac2{4(\beta + 2\log\frac{27}{16})}\log L}
 \end{equation}
such that $X_L \in \{ M,M+1\}$ with probability going to 1 as $L\to\infty$.
Furthermore, there exists some $H=H(L)$ with $H \sim \frac{1+o(1)}{\sqrt{2}} M(L)$ such that w.h.p.\
\begin{equation}\label{eq-height-concRes}
 \#\big\{v : \eta_v \in \{H,H+1\}\big\} \geq (1-\epsilon_\beta) L^2\,,
\end{equation}
where $\epsilon_\beta$ can be made arbitrarily small as $\beta$ increases.
 \end{theorem}
 \begin{proof} As in the previous cases the proof boils down to
   control the large deviations of the surface at one or two
   vertices. For the one vertex  large deviation
   \eqref{e:tailRateRes} we first need to control the contribution to the
   partition function of nested contours around the origin.
  \subsubsection{ The partition function of nested circuits and the six-vertex model}
Let $\cN_0$ be the set of collections of $h$ nested self-avoiding circuits $\{\cC_1,\ldots,\cC_h\}$
on the dual lattice ${\Z^2}^*$, ordered from the outermost one to the
innermost one, which do not overlap and encircle the
origin. We then define the
associated partition function by
\[
\Upsilon=\sum_{\{\cC_1,\ldots,\cC_h\}\in \cN_0}e^{-\beta\sum_{i=1}^h |\cC_i|}\,.
\]
Each contour must cross each of the positive and negative axes at least
once.  Let $a_i+1/2$ and $b_i+1/2$ denote the minimal crossing points
of $\cC_i$ of the positive $x$ and $y$ axes respectively and let $\underline{a}=(a_1,\ldots,a_h), \ \underline{b}=(b_1,\ldots,b_h)$. Note that the $a_1 > a_2> \ldots > a_h$ and similarly $b_1 >
\ldots > b_h$. By definition, for each $i=1,\dots,h$ $\cC_i$ connects $(a_i+1/2,1/2)$ to $(1/2,b_i+1/2)$
without crossing the positive $x$-axis to the left of $a_i$ or the positive
$y$-axis below $b_i$.
Therefore
\begin{equation}
  \label{eq:temporale}
\Upsilon \leq e^{-4\beta h} \bigg(\sum_{\underline{a},\underline{b}} \hat{\Upsilon}_{\underline{a},\underline{b}}  \bigg)^4\,,
\end{equation}

where
\[
\hat{\Upsilon}_{\underline{a},\underline{b}} := \sum_{
  \gamma_1,\ldots,\gamma_h} e^{-\beta\sum_{i=1}^h |\gamma_i|}
\]
and the sum is over collections of $h$ dual paths which do not cross or
share common edges and such that $\gamma_i$ connects $(1/2,b_i+1/2)$ to
$(a_i+1/2,1/2)$ without crossing the positive $x$-axis to the left of $a_i$ or the positive
$y$-axis below $b_i$ (cf. Figure~\ref{fig:rsos_red}). The factor
$e^{-4\beta h}$ comes from the edges of $\cC_1,\dots,\cC_h$ crossing the
axes at the points $a_i+1/2$ or $b_i+1/2$, $i=1,\dots,h$.

In order to estimate $\hat{\Upsilon}_{\underline{a},\underline{b}}$ we
associate to each path $\gamma_i$  a \emph{down-right} path
$\psi(\gamma_i)$, {\it i.e.}
a path satisfying  the same
constraints as $\gamma_i$ and which in addition only makes steps down or
right from $(1/2,b_i+1/2) $ to $(a_i +1/2,1/2)$ (cf. Figure~\ref{fig:rsos_red}).
For this purpose, for each $0\leq x < a_i$ we define
\begin{align*}
m_{x}(\gamma_i)&:= \min\{k \geq 0: ((x+1/2,k+1/2),(x+3/2,k+1/2))\in \gamma_i\},\\
m^*_{x}(\gamma_i)&:= \min\{m_{x'}(\gamma_i): 1\leq x'\leq x\}\,.
\end{align*}
Then
$\psi(\gamma_i)$ is defined as the path from $(1/2,b_i+1/2) $ to $(a_i+1/2,1/2)$ consisting of
\begin{compactitem}[$\bullet$]
\item the horizontal edges $((x+1/2,m^*_{x}(\gamma_i)+1/2),(x+3/2,m^*_{x}(\gamma_i)+1/2))$ for $0\leq x \leq a_i-1$, and
\item the vertical edges in the direct paths from $(x+1/2,m^*_{x-1}(\gamma_i)+1/2))$ to $(x+1/2,m^*_{x}(\gamma_i)+1/2))$ where $m^*_{-1}(\gamma_i)=b_i$.
  \end{compactitem}

{
\begin{figure}[htp]
\setlength{\abovecaptionskip}{-5pt}
  \centering
  \begin{tikzpicture}[font=\footnotesize]

    \node (plot1) at (0,0) {
      \includegraphics[scale=.7]{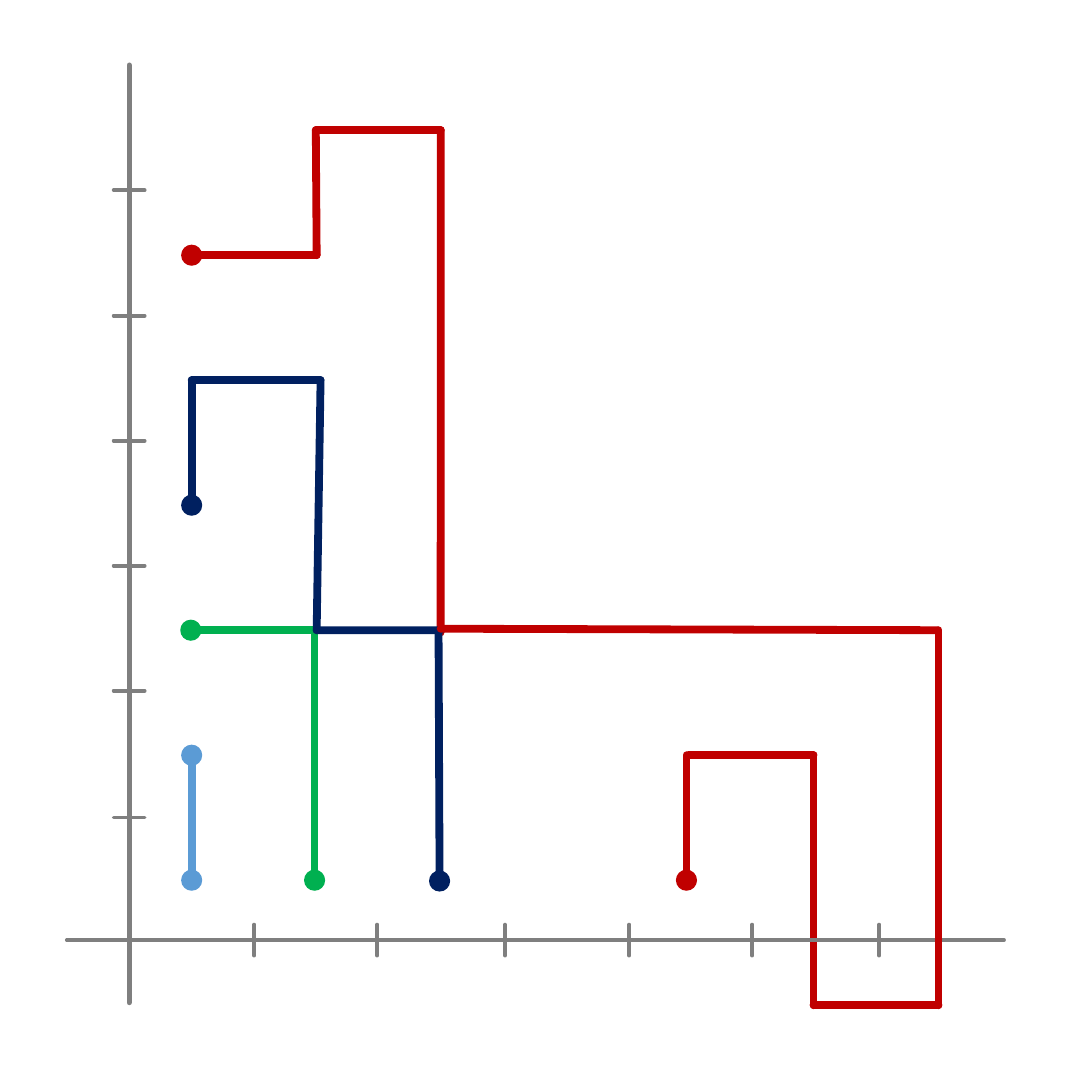}};
    \node (plot2) at (7cm,0) {
      \includegraphics[scale=.7]{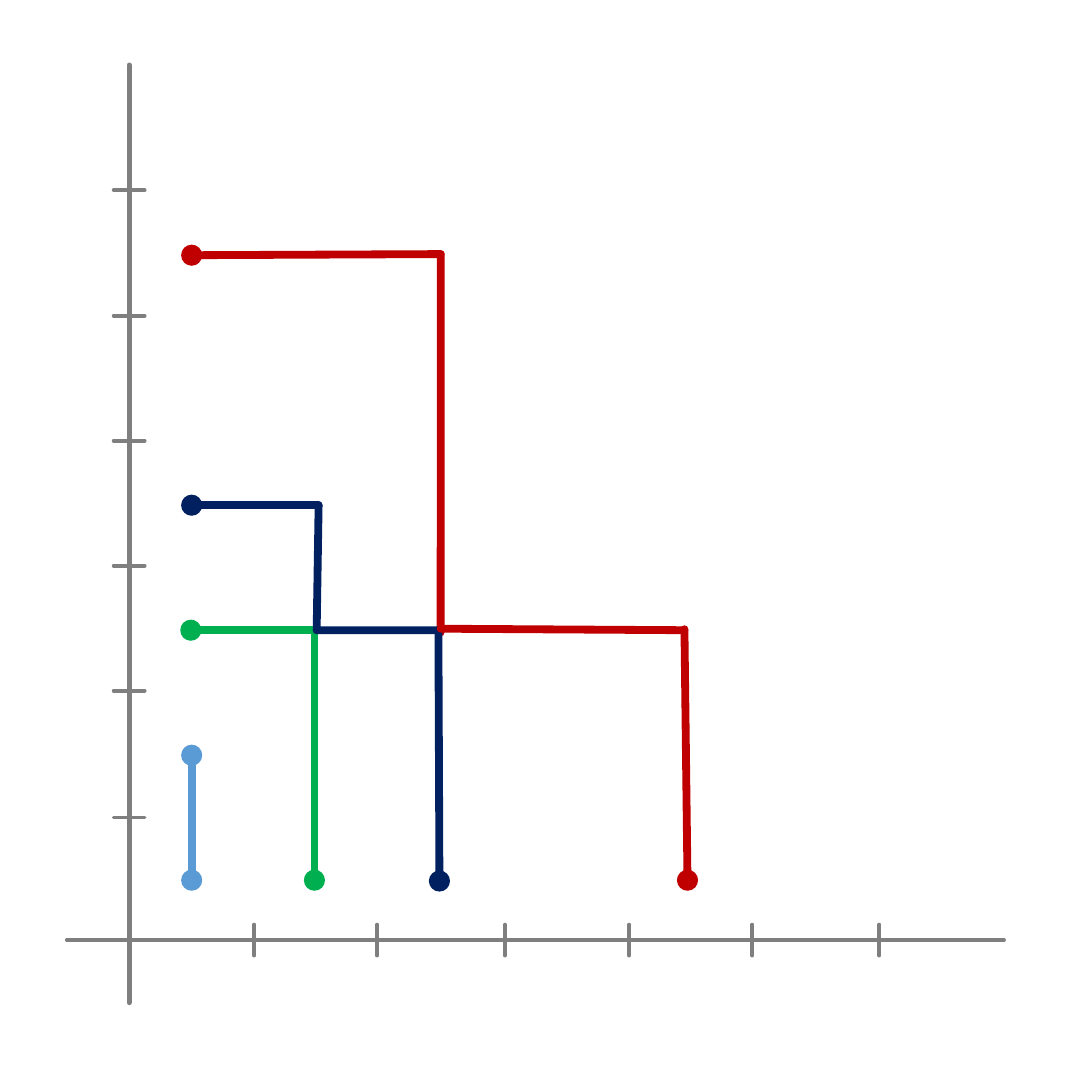}};
    \node (plot3) at (5.25cm,-5cm) {
      \includegraphics[scale=.8]{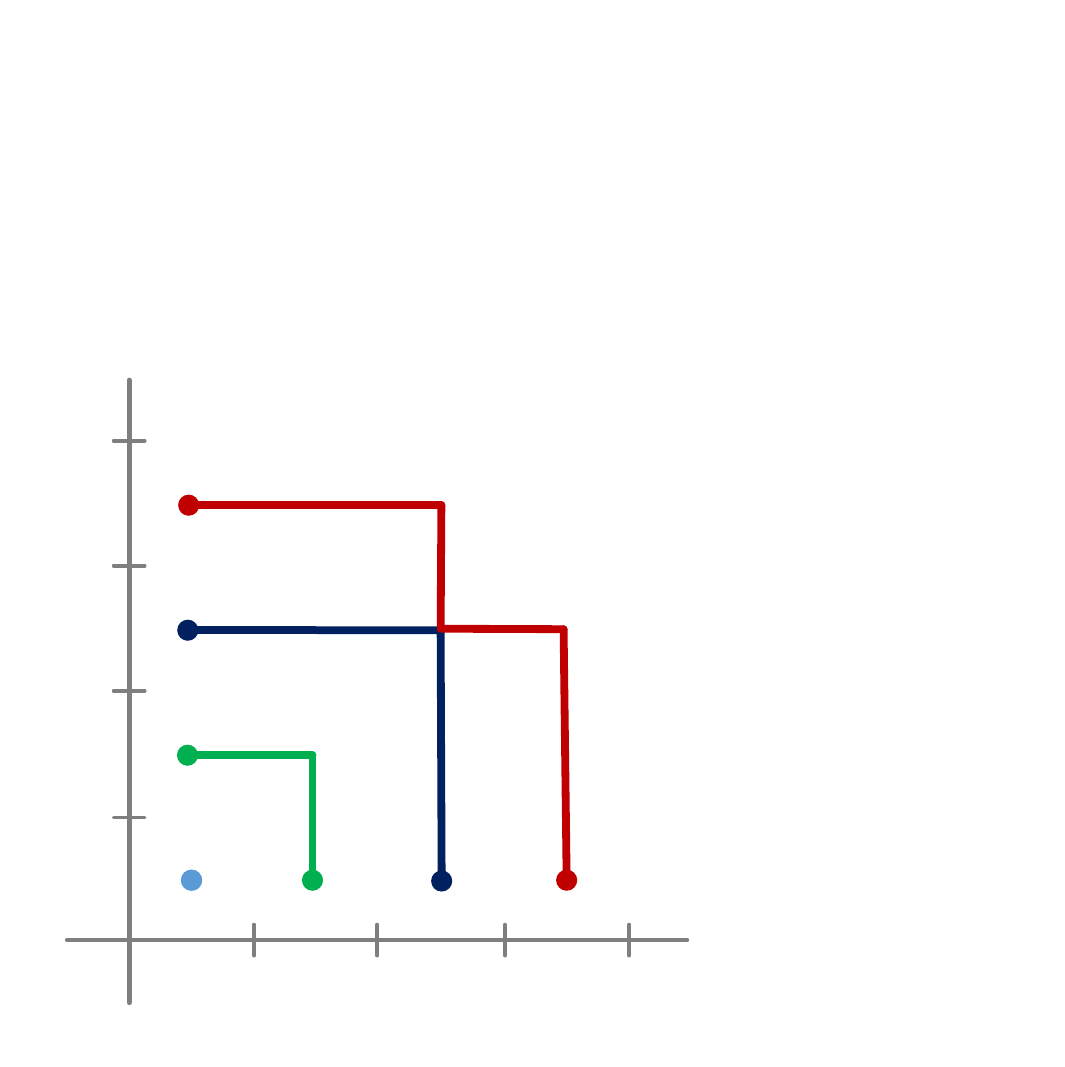}};

    \begin{scope}[shift={(plot1.south west)},x={(plot1.south east)},y={(plot1.north west)}]
      \node[below=-.5em] at (.64,0.16) {$a_1$};
      \node[below=-.5em] at (.42,0.16) {$a_2$};
      \node[below=-.5em] at (.31,0.16) {$a_3$};
      \node[below=-.5em] at (.2,0.16) {$a_4$};
      \node[left=-.7em] at (0.162,.745) {$b_1$};
      \node[left=-.7em] at (0.162,.525) {$b_2$};
      \node[left=-.7em] at (0.162,.415) {$b_3$};
      \node[left=-.7em] at (0.162,.305) {$b_4$};
      \node[above] at (.46,.7) {\color[RGB]{212,84,84}$\gamma_1$};
      \node[above] at (.34,.53) {\color[RGB]{0,31,95}$\gamma_2$};
      \node[above] at (.34,.28) {\color[RGB]{0,175,80}$\gamma_3$};
      \node[above] at (.23,.21) {\color[RGB]{91,155,212}$\gamma_4$};
    \end{scope}

   \begin{scope}[shift={(plot2.south west)},x={(plot2.south east)},y={(plot2.north west)}]
      \node[below=-.5em] at (.64,0.16) {$a_1$};
      \node[below=-.5em] at (.42,0.16) {$a_2$};
      \node[below=-.5em] at (.31,0.16) {$a_3$};
      \node[below=-.5em] at (.2,0.16) {$a_4$};
      \node[left=-.7em] at (0.162,.745) {$b_1$};
      \node[left=-.7em] at (0.162,.525) {$b_2$};
      \node[left=-.7em] at (0.162,.415) {$b_3$};
      \node[left=-.7em] at (0.162,.305) {$b_4$};
      \node[above] at (.48,.7) {\color[RGB]{212,84,84}$\psi(\gamma_1)$};
      \node[above] at (.34,.52) {\color[RGB]{0,31,95}$\psi(\gamma_2)$};
      \node[above] at (.36,.28) {\color[RGB]{0,175,80}$\psi(\gamma_3)$};
      \node[above] at (.25,.21) {\color[RGB]{91,155,212}$\psi(\gamma_4)$};
    \end{scope}

    \begin{scope}[shift={(plot3.south west)},x={(plot3.south east)},y={(plot3.north west)}]
    \node[below=-.5em] at (.53,0.16) {$u_1$};
      \node[below=-.5em] at (.42,0.16) {$u_2$};
      \node[below=-.5em] at (.31,0.16) {$u_3$};
      \node[below=-.5em] at (.2,0.16) {$u_4$};
      \node[left=-.7em] at (0.162,.525) {$u_1$};
      \node[left=-.7em] at (0.162,.415) {$u_2$};
      \node[left=-.7em] at (0.162,.3) {$u_3$};
      \node[left=-.7em] at (0.162,.19) {$u_4$};
      \node[above] at (.63,.3) {\color[RGB]{212,84,84}$\chi_1(\psi(\gamma_1))$};
      \node[above] at (.3,.41) {\color[RGB]{0,31,95}$\chi_2(\psi(\gamma_2))$};
      \node[above] at (.3,.295) {\color[RGB]{0,175,80}$\chi_3(\psi(\gamma_3))$};
      \node[above] at (.22,.195) {\color[RGB]{91,155,212}$\chi_4(\psi(\gamma_4))$};
    \end{scope}

  \end{tikzpicture}
  \caption{The upper left frame displays the paths $\gamma_i$
  contributing to the partition function
    $\hat{\Upsilon}_{\underline{a},\underline{b}}$.  These are transformed into the down-right paths $\psi(\gamma_i)$ with
    partition function $\tilde{\Upsilon}_{\underline{a},\underline{b}}$ in the upper right frame.  The bottom frame denotes $\chi_i(\gamma_i)$ where the endpoints are shifted to $u_i$.}
  \label{fig:rsos_red}
\end{figure}
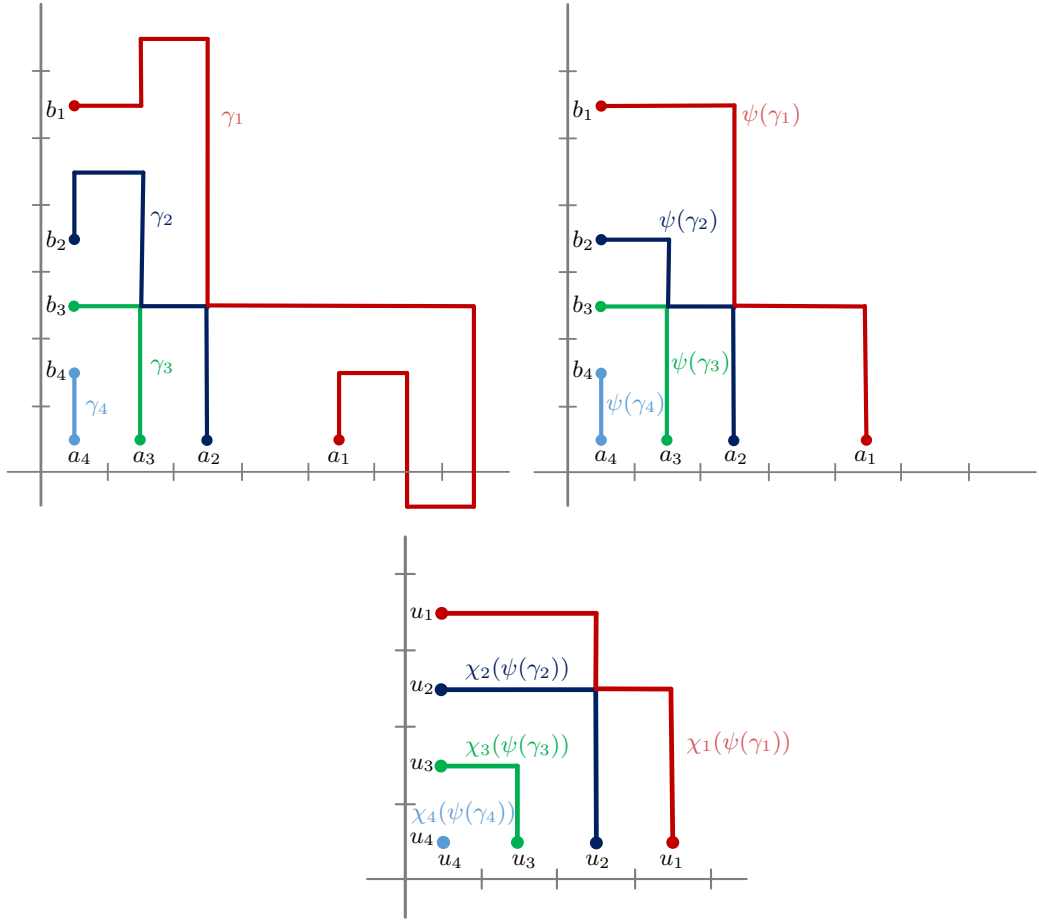
}

\begin{claim}\label{cl:pathStraightening}
There exists an absolute constant $C>0$ such that, for all large enough $\beta$, all integers
$a,b$ and all down-right paths $\gamma^*$ from $u=(1/2,b+1/2)$ to $v=(a+1/2,1/2)$,
\[
\sum_{\gamma:\, \psi(\gamma)=\gamma^*} e^{-\beta |\gamma|} \leq e^{-(\beta-C e^{-\beta})(a+b)}\,.
\]
where the sum is over all paths connecting $u$ to
$v$ which do not cross the positive $x$-axis to the left of
$a$ or the positive $y$-axis below $b$.
\end{claim}
\begin{proof}
Let $
W_{z,z'}:=\sum_{\gamma} e^{-\beta|\gamma|}
$
where the sum is over all paths (not necessarily down-right) from $z=(z_1,z_2)$ to $z'=(z_1',z_2')$.  By standard estimates (see, e.g.,~\cite{DKS}) this can be bounded by
\begin{equation}\label{e:contourWeightedSum}
W_{z,z'} \leq e^{-(\beta-5e^{-\beta})\min\{|z_1-z_1'|,|z_2-z_2'|\}}\,.
\end{equation}
Given a down-right path $\gamma^*$ as in the claim, let
$0=x_0<x_1<x_2<\ldots<x_s<a$ denote the points where $m^*_{x}(\gamma^*) <
m^*_{x-1}(\gamma^*)$ (i.e. where the height of the path decreases). 

Let now $\gamma$ be any path connecting $u$ to $v$ which do not cross the positive $x$-axis to the left of
$a$ or the positive $y$-axis below $b$ such that $\psi(\gamma)=\gamma^*$. We claim that each path $\gamma$ must pass through each vertex $z_j:=(x_j+1/2,m^*_{x_j}(\gamma^*)+1/2))$ in order of $j$.  By construction the edges $e_j=(z_j,z_j+(1,0))$ must all be present in the path $\gamma$ since these represent new record low horizontal edges for the path moving from left to right.  To see that they appear in order take $0\leq j<j' < s$.  Suppose that in the direction from $u$ to $v$ the path first reaches the edge $e_{j'}$ before $e_j$.  The path from $u$ to $e_{j'}$ must then by definition pass above $e_j$.  It must then continue onto $e_j$.  However, it is then geometrically impossible to reach $v$ without passing below $e_{j'}$, crossing itself or crossing the positive $x$-axis to the left of
$a$ or the positive $y$-axis below $b$.  This gives a contradiction and thus it must cross the $e_j$ in order.

We, therefore, may split the path into segments $\gamma_j$ from $z_j$
to $z_{j+1}$. Defining $z_j':= (x_{j+1}+1/2,m^*_{x_{j}}(\gamma^*)+1/2)$ we have that $\gamma_j$ must pass through or above $z_j'$, that is that for some $\ell_j \geq 0$, $z_j'+\ell_j(0,1)\in\gamma_j$.  If this were not the case there would have to be a horizontal edge $((x_{j+1}-1/2,r),(x_{j+1}+1/2,r))$ for some $r < m^*_{x_{j}}(\gamma^*)+1/2$ and so $m^*_{x_{j+1}-1}(\gamma^*) < m^*_{x_{j}}(\gamma^*)$ which contradicts the definition of $x_{j+1}$.  For concreteness we take $\ell_j$ to correspond to the first vertex on or above $z_j'$ on $\gamma_j$.

Summing over the possible segments $\gamma_j$ which satisfy the aforementioned conditions we have that
\begin{align*}
\sum_{\gamma_j}  e^{-\beta|\gamma|} &\leq  \sum_{\ell_j} W_{z_j,z_j'+\ell_j(0,1)} W_{z_j'+\ell_j(0,1),z_{j+1}} \\
&\leq \sum_{\ell_j} e^{-(\beta-5e^{-\beta})(x_{j+1}-x_j+m^*_{x_{j}}(\gamma^*)-m^*_{x_{j+1}}(\gamma^*)+\ell_j)} \\
&\leq  \left(\frac1{1 - e^{-(\beta-5e^{-\beta})}}\right) e^{-(\beta-5e^{-\beta})(x_{j+1}-x_j+m^*_{x_{j}}(\gamma^*)-m^*_{x_{j+1}}(\gamma^*))}\\
&\leq  e^{-(\beta-7e^{-\beta})(x_{j+1}-x_j+m^*_{x_{j}}(\gamma^*)-m^*_{x_{j+1}}(\gamma^*))} \,.
\end{align*}
Combining the segments $\gamma_j$ we have that
\begin{align*}
\sum_{\gamma:\psi(\gamma)=\gamma^*} e^{-\beta |\gamma|} &\leq \prod_j e^{-(\beta-7e^{-\beta})(x_{j+1}-x_j+m^*_{x_{j}}(\gamma^*)-m^*_{x_{j+1}}(\gamma^*)}= e^{-(\beta-7e^{-\beta})(a+b)}\,,
\end{align*}
which completes the proof.
\end{proof}
Let
\[
\tilde{\Upsilon}_{\underline{a},\underline{b}} = \sum_{\gamma_1,\ldots,\gamma_h} e^{-\beta\sum_{i=1}^h |\gamma_i|} = \sum_{\gamma_1,\ldots,\gamma_h} e^{-\beta\sum_{i=1}^h (a_i + b_i)}\,,
\]
where now the sum is over collections of down-right dual paths which do not cross or share common edges such that $\gamma_i$ connects $(1/2,b_i+1/2)$ to $(a_i+1/2,1/2)$.  By  Claim~\ref{cl:pathStraightening} we have that
\begin{equation}\label{eq:ZHatTildeComparison}
\hat{\Upsilon}_{\underline{a},\underline{b}} \leq
e^{Ce^{-\beta}\sum_{i=1}^h (a_i + b_i)}\  \tilde{\Upsilon}_{\underline{a},\underline{b}}\,.
\end{equation}
Let $\underline{u}=(h-1,h-2,\ldots,0)$, the minimal possible value of $\underline{a}$ or $\underline{b}$.
\begin{claim}\label{cl:tightPathPacking}
For all $\underline{a}$ or $\underline{b}$ we have
\[
\tilde{\Upsilon}_{\underline{a},\underline{b}}  \leq
\tilde{\Upsilon}_{\underline{u},\underline{u}}\  e^{C e^{-\beta} h^2 - \sum_i (a_i+b_i-2(h-i))}\,.
\]
\end{claim}
\begin{proof}
For down-right paths $\gamma_1,\dots,\gamma_h$ it is convenient to think of them as the graph of
a function; we will write $\gamma_i(s)$ to denote the maximum height
of the path along the line $x=s+1/2$.  Our conditions on the
$\{\gamma_i\}_{i=1}^h$ in the definition of $\tilde{\Upsilon}$ implies that
$\gamma_i(x)$ is strictly decreasing in $i$.  Define the new
down-right path $\chi_i(\gamma_i)$ by
\[
\chi_i(\gamma_i)(x) = \begin{cases}
\min\{\gamma_i(x),h-i+1/2\} & x \leq a_i\\
0 & x>a_i\,.
\end{cases}
\]
The paths $\chi_1(\gamma_1),\ldots,\chi_h(\gamma_h)$ still do not cross or share a common edge.
We will count the number of down-right paths $\gamma_i$ from $(1/2,b_i+1/2)$ to $(a_i+1/2,1/2)$ which are mapped to a given $\tilde{\gamma}$.  This is the number of paths from $(1/2,b_i+1/2)$ to $(x,h-i+1/2)$ where $x=\min\{x':\gamma^*(x') \leq h-i+1/2\}\wedge (h-i+1/2)$ times the number of paths from $(h-i+1/2,\gamma^*(h-i)\wedge (h-i+1/2))$ to  $(a+1/2,1/2)$.  In particular
\[
\# \{\gamma_i: \chi_i(\gamma_i) = \tilde{\gamma}\} \leq \binom{ b_i }{ h-i}\binom{ a_i }{ h-i}\,.
\]
Now if $s\leq t$ then by maximizing over $s$,
\[
e^{-(\beta -1) s} \binom{ s+t }{ s} \leq \frac{(2e^{-(\beta - 1)}t)^s}{s!} \leq \frac{(2e^{-(\beta-1)}t)^{2e^{-(\beta-1)}t}}{(2e^{-(\beta-1)}t)!} \leq e^{C e^{-\beta}t}
\]
for some absolute constant $C$.  If $s>t$ then
\[
e^{-(\beta - 1) s} \binom{ s+t }{ s} \leq e^{-(\beta - 1) s} 2^{2s} \leq 1\,.
\]
Together this gives us that
\[
\# \{\gamma_i: \chi_i(\gamma_i) = \tilde{\gamma}\}e^{-\beta (a_i+b_i - 2(h-i))} \leq  e^{2C e^{-\beta}(h-i) - (a_i+b_i - 2(h-i))}\,.
\]
Finally by considering the mapping $(\gamma_1,\ldots,\gamma_h) \mapsto (\chi_1(\gamma_1),\ldots,\chi_h(\gamma_h))$ we have that
\[
\tilde{\Upsilon}_{\underline{a},\underline{b}} \leq \tilde{\Upsilon}_{\underline{u},\underline{u}} \prod_{i=1}^h e^{2C e^{-\beta}(h-i) - (a_i+b_i - 2(h-i))}\leq \tilde{\Upsilon}_{\underline{u},\underline{u}} e^{C e^{-\beta} h^2 - \sum_i (a_i+b_i-2(i-1))}\,,
\]
as required.
\end{proof}
We now combine the above claims to establish the following result
\begin{lemma}\label{l:fullContourPartitionBound}
There exists an absolute constant $C>0$ such that the partition function for $h$ nested non-overlapping contours around the origin satisfies
\[
\Upsilon \leq e^{C e^{-\beta} h^2} \tilde{\Upsilon}_{\underline{u},\underline{u}}^4\,.
\]
\end{lemma}
\begin{proof}
If we combine \eqref{eq:temporale}, \eqref{eq:ZHatTildeComparison} and Claim~\ref{cl:tightPathPacking} we have that
\begin{align*}
\Upsilon &\leq e^{-4\beta h} \bigg(\sum_{\underline{a},\underline{b}} e^{Ce^{-\beta}\sum_{i=1}^h (a_i + b_i)} \tilde{\Upsilon}_{\underline{a},\underline{b}} \bigg)^4\\
&\leq e^{-4\beta h} \bigg(\sum_{\underline{a},\underline{b}} e^{C'e^{-\beta}h^2 } e^{- \frac12\sum_i (a_i-(h-i))-\frac12\sum_{i=1}^h ( b_i-(h-i)))} \tilde{\Upsilon}_{\underline{u},\underline{u}}  \bigg)^4\\
&\leq e^{4C'e^{-\beta}h^2 +O(\beta h)} \left(\frac1{1+e^{-1/2}}\right)^{8h} \tilde{\Upsilon}_{\underline{u},\underline{u}}^4 \leq e^{4C'e^{-\beta}h^2 +O(\beta h)} \tilde{\Upsilon}_{\underline{u},\underline{u}}^4\,.\qedhere
\end{align*}
\end{proof}

The asymptotics of $\tilde{\Upsilon}_{\underline{u},\underline{u}}$ as
$h\uparrow \infty$ will follow from a
bijection between configurations of non-overlapping down-right paths
and the \emph{six-vertex model} (together with the bijection between
the latter and ASMs), which was pointed out to us by David B.\ Wilson
and which represents a special case of the isomorphism between the \emph{terrace-ledge-kink} model and the six-vertex model (see, e.g.,~\cite{Abraham}*{pp.43--45 and in particular Figs.~13--14}).
\begin{proposition}\label{p:ASMBound}
We have that asymptotically
\[
\tilde{\Upsilon}_{\underline{u},\underline{u}}e^{-\beta\sum_{i=1}^h
  (h-i)} = \bigg(\frac{3\sqrt{3}}{4}\bigg)^{(1+o(1))h^2}\qquad
\text{as $h\to \infty$}\,.
\]
\end{proposition}
\begin{proof}
Consider a set of edge-disjoint non-crossing SE/NE paths counted by
$\tilde{\Upsilon}_{\underline{u},\underline{u}}$ between $\{ (-i,-i) \leftrightsquigarrow (i,-i) : i=1,\ldots,h\}$
(as was illustrated in Figure~\ref{fig:rsos1} in the introduction), and observe that there are only six possible constellations of existing/missing edges incident to an internal vertex:
Indeed, as shown in Figure~\ref{fig:rsos2}, since paths cannot overlap, upon directing the edges towards SE/NE the in-degree of every internal vertex must equal its out-degree; thus,
such a vertex can have either in-degree 0 (no incident edges appear), or in-degree 1 (whence there are  4 possibilities: 2 choices for an incoming edge
and 2 for an outgoing one), or in-degree 2 (all incident edges appear).

\begin{figure}
\begin{center}
\hspace{-0.5in}
\includegraphics[width=0.4\textwidth]{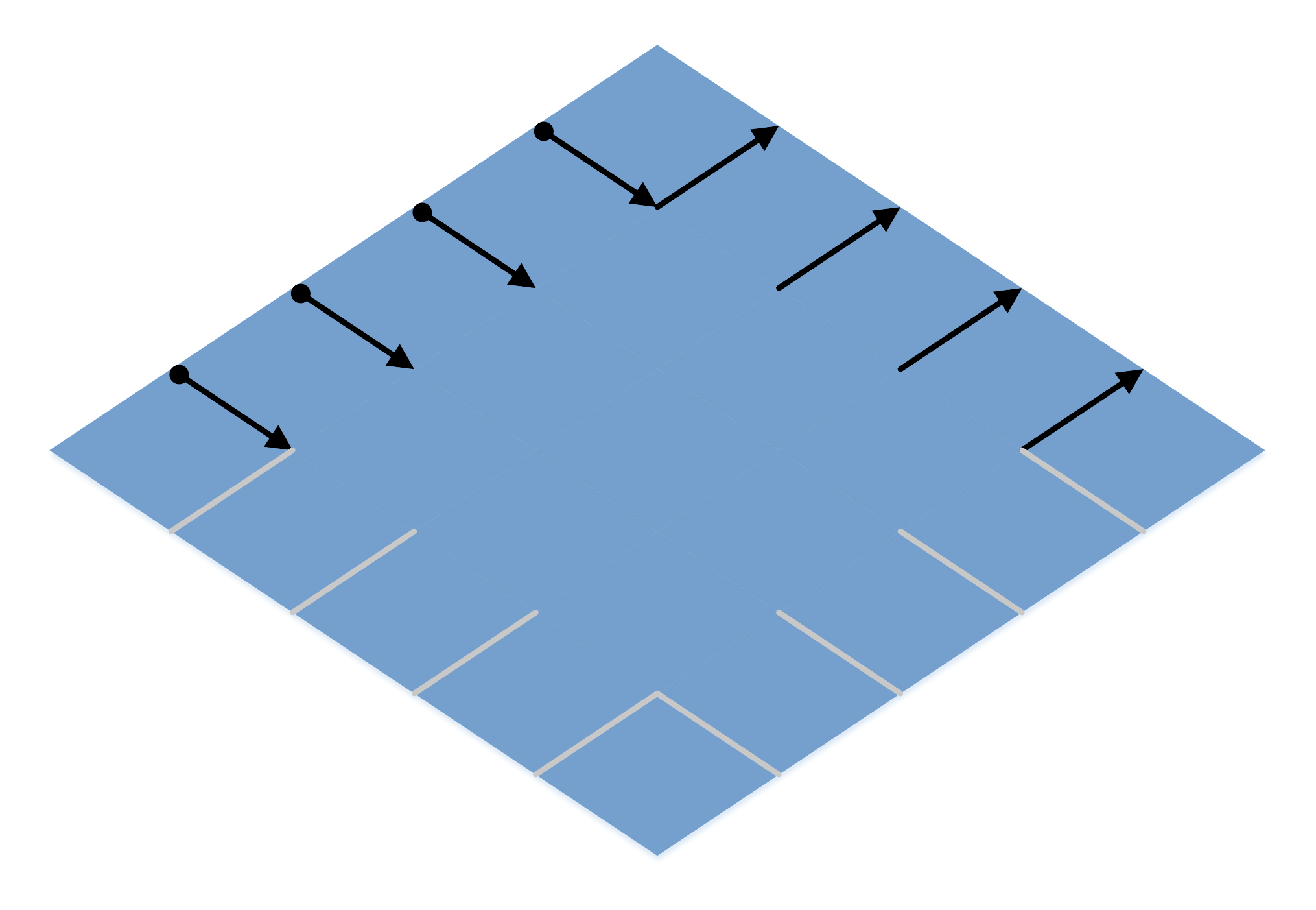}
\hspace{-0.25in}
\raisebox{0.2\height}{\includegraphics[width=0.4\textwidth]{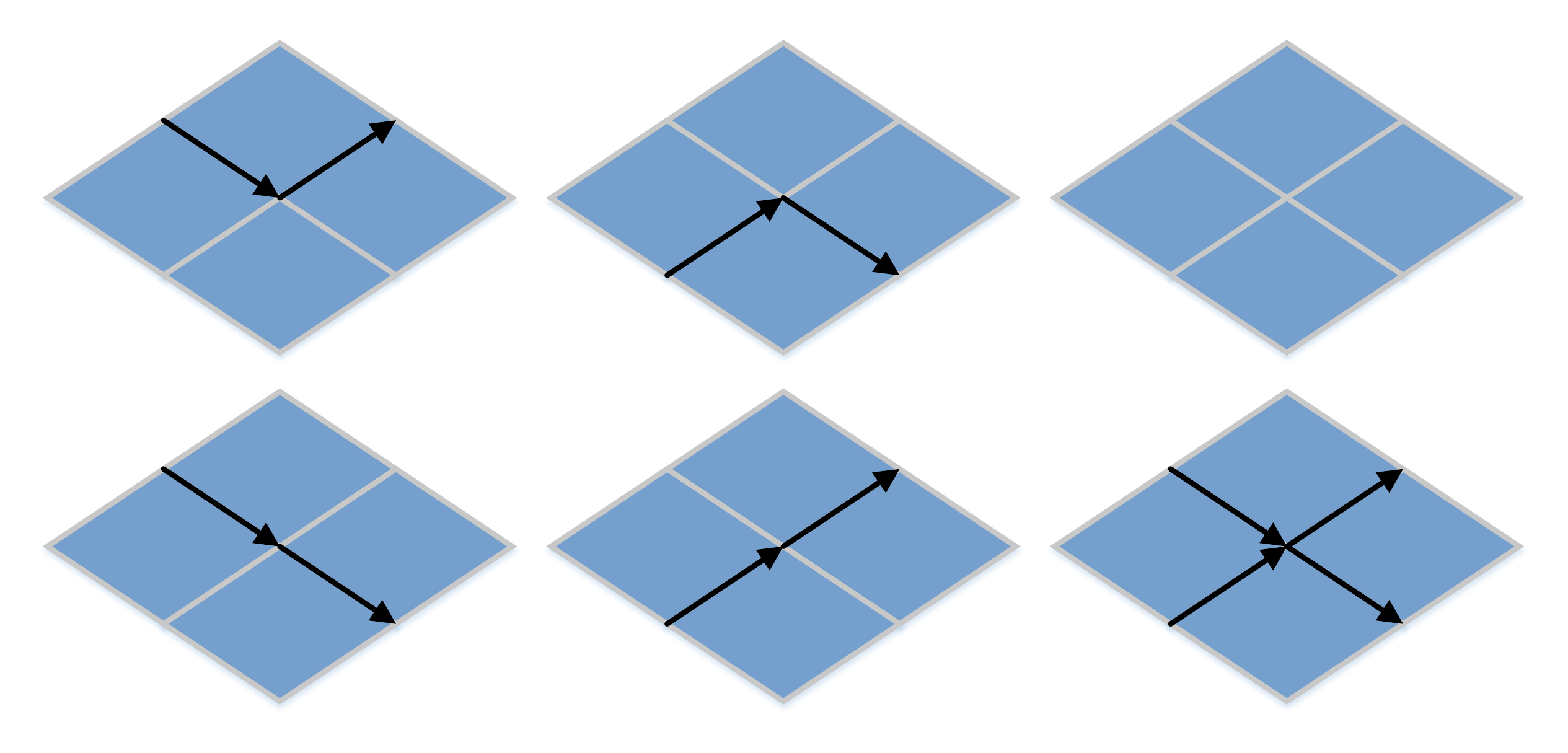}}
\end{center}
\vspace{-0.1in}
\caption{The six-vertex model with domain-wall boundary conditions.}
\label{fig:rsos2}
\end{figure}

The requirement that the paths are to connect $(-i,-i)\leftrightsquigarrow (i,-i)$ for all $i$ can then be embedded in boundary conditions along the $h\times h$ diamond, in the form of always having the $2h$ edges incident to the boundary points along the upper two faces (i.e., $\{ (-i,-i) : i\in[h]\}\cup \{ (i,-i) : i\in[h]\}$) and forbidding the $2h$ edges incident to the remaining boundary points (the lower two faces of the diamond), as in Figure~\ref{fig:rsos2}. This is precisely the six-vertex model with domain-wall boundary conditions, in precise correspondence with the required set of paths.

In the special case of the domain-wall boundary condition, there is a well known correspondence between the six-vertex model and ASMs: one can follow the SE lines of the diamond starting from the second edge (the first is always present as part of the boundary conditions), and construct a $\{0,\pm1\}$-matrix as follows: associating the rows with the SE lines, one reads the row from left to right by processing the line towards SE, registering $1$ if we move from a present edge to a missing one, a $-1$ if we move from a missing edge to a present one, and $0$ otherwise. The boundary conditions guarantee that each row would sum to $1$ (as it begins with a present edge and ends with a missing one). The same conclusion applies to the $\{0,\pm1\}$-matrix that one reads from the configuration by following its SW lines (reading the columns from top to bottom). Finally, by definition of the six-vertex model, these two methods produce the same matrix, which is thereby an ASM.
\end{proof}
\subsubsection{Single vertex large deviation: proof of \eqref{e:tailRateRes}}
If $\eta_0=h$ then there exist nested contours $\Gamma_1,\ldots, \Gamma_h$ surrounding the origin.  By the same Peierls argument as in Eq.~\eqref{e:PeierlsBound},
\[
\pi(\eta_0 =h) \leq \sum_{\{\gamma_1,\ldots,\gamma_h\}\in \cN_0} e^{-\beta|\gamma_i|} = \Upsilon\,.
\]
By Lemma~\ref{l:fullContourPartitionBound} and Proposition~\ref{p:ASMBound} we therefore have that
\[
\pi(\eta_0 =h) \leq e^{-4\left(\beta + 2\log\frac{27}{16} - C e^{-\beta}\right) h^2}\,.
\]
Now let $Q=\{(x,y):|x|\vee |y| \leq h+1\}$.  Then
\[
\pi(\eta_Q=0) \geq e^{-C e^{-\beta}h^2}\,.
\]
Let $\gamma_1,\ldots,\gamma_h$ be a nested collection of non-intersecting contours with the minimum possible lengths, i.e., $|\gamma_i| = 8(h-i)+4$.  The number of such collections is exactly $(\tilde{\Upsilon}_{\underline{u},\underline{u}}e^{-\beta\sum_{i=1}^h (h-i)})^4$.
Let $\xi_\gamma(z) = \#\{i:z\in V_{\gamma_i}\}$ which is constructed to have its contours as $\gamma_i$.  Then
\[
\pi(\eta_Q=\xi_Q) = \pi(\eta_Q=0) e^{-\sum_{i=1}^h |\gamma_i|} \geq e^{-4\beta h^2 - C e^{-\beta}h^2}\,.
\]
By our bound in Proposition~\ref{p:ASMBound} on the number of such contours, we deduce that
\[
\pi(\eta_0 =h) \geq e^{-4\left(\beta + 2\log\frac{27}{16} + C e^{-\beta}\right) h^2}\,,
\]
as required.

We also have the following result proved essentially the same argument as Lemma~\ref{l:pGeqPH}.
\begin{lemma}
\label{lem:4.8}
For each $\beta$, there exists a constant $c>0$ such that,
\[
\frac{\pi(\eta_0=h)}{\pi(\eta_0=h-1)} \leq e^{-c\beta h}\,.
\]
\end{lemma}

\subsubsection{Large deviations at two vertices.}

\begin{lemma}
For all $z \neq 0$ we have that
\[
\pi(\eta_z \geq h \mid \eta_0\geq h) \leq e^{-c'\beta h}\,.
\]
\end{lemma}
\begin{proof}[Proof of Lemma]
Write $\theta=\pi(\eta_z \geq h \mid \eta_0\geq h)$.  If $z=(x,y)$ we
split the proof into two cases.  First suppose that $\max(|x|,|y|)\leq \epsilon h$ where $\epsilon =\frac{c}{100}$ and $c$ is the constant in Lemma~\ref{lem:4.8}.  By the FKG inequality and by symmetry
\begin{equation}\label{e:twoXStep}
\pi(\eta_{(2x,0)}\geq h\mid\eta_0\geq h) \geq \pi(\eta_{(2x,0)}\geq h\mid\eta_0\geq h,\eta_z\geq h)\pi(\eta_{z}\geq h\mid\eta_0\geq h) \geq \theta^2\,.
\end{equation}
We define $W = \{-2x,0,2x\}^2$ and
$U=\{(x',y')\in\Z^2:\max\{|x'|,|y'|\} = 2x\}$.  Then, since each
element of $W$ is $2x$ offset from another element of $W$, by applying
Eq.~\eqref{e:twoXStep}, the symmetry of the model and the FKG inequality we have that
\[
\pi(\min_{w\in W} \eta_w \geq h \mid \eta_0 \geq h) \geq \theta^{16}\,.
\]
Since the step size of the restricted SOS surface is at most one, on the event $\min_{w\in W} \eta_w \geq h$ we have $\min_{w\in U} \eta_w \geq h-x$ and so
\[
\pi(\min_{w\in U} \eta_w \geq h-x) \geq \theta^{16}\pi(\eta_0 \geq h)\,.
\]
Next we observe that, since every gradient along an edge can be $-1,0$
or $1$, the total contribution of a single edge to the
partition function is at most $1+2e^{-\beta}$.  As the total number of
interior edges in $U$ is less than $32x^2$,
the partition function of the model on the interior of $U$ is at most $((1+2e^{-\beta})^{32 x^2}$  under any
boundary conditions (we neglect the fact that not all gradients
correspond to configurations, only those that are curl free).
Moreover, the energy of a pyramid with base $U$ and
height $2x$ is bounded from above by $32x^2$. Thus, using again the FKG
inequality we get
\[
\pi(\eta_0 \geq h+x \mid \min_{w\in U} \eta_w \geq h-x) \geq\pi(\eta_0 \geq h+x \mid \eta_U \equiv h-x) \geq \frac{e^{-32\beta x^2}}{(1+2e^{-\beta})^{32 x^2}}\,.
\]
Combining the above estimates we get that
\[
\pi(\eta_0 \geq h+x) \geq \theta^{16} \frac{e^{-32\beta x^2}}{(1+2e^{2-\beta})^{32 x^2}} \pi(\eta_0 \geq h)\,.
\]
However, by Lemma~\ref{l:pGeqPH} we have that
\[
\pi(\eta_0 \geq h+x) \leq e^{-c\beta x h} \pi(\eta_0 \geq h)\,,
\]
and by combining these while using that $|x| \leq \epsilon h$ it follows that for some $c'>0$ depending only $\beta$,
\[
\theta \leq e^{-c' h}\,.
\]
Now suppose that $\max(|x|,|y|) > \epsilon h$.  Let $\mathcal{A}_{z_1,z_2}$ denote the event that there is a chain of vertices at height at least $h/2$ surrounding both $z_1$ and $z_2$.  Then
\begin{align*}
\pi(\eta_z\geq h \mid \eta_0 \geq h)  &\leq \pi(\mathcal{A}_{0,z}\,,\,\eta_z\geq h\mid \eta_0 \geq h) + \pi(\eta_z\geq h\mid \eta_0 \geq h\,,\, \mathcal{A}_{0,z}^c)\,.
\end{align*}
Now, $\mathcal{A}_{0,z}^c$ implies that the outermost chain of vertices at least $h/2$ surrounding $0$ does not include $z$.  Hence,
\[
\pi(\eta_z\geq h\mid \eta_0 \geq h\,,\, \mathcal{A}_{0,z}^c) \leq \sup_\gamma \pi(\eta_z \geq h \mid \eta_\gamma= h/2) \leq c e^{-\beta h}\,,
\]
where the supremum is over all chains of vertices $\gamma$ surrounding
$z$ and the second inequality follows by a basic Peierls estimate. So
either $\theta\leq 2 e^{-\tfrac{\epsilon}{100} h}$ (in which case we
are done) or $\pi(\mathcal{A}_{0,z},\eta_z\geq h\mid \eta_0 \geq h)
\geq e^{-\tfrac{\epsilon}{100} h}$ which we assume.  Using FKG, \eqref{e:tailRateRes} and
translation invariance,
\begin{align*}
&\pi(\eta_z\geq h\,,\, \mathcal{A}_{0,z},\mathcal{A}_{z,2z},\ldots,
\mathcal{A}_{(\frac{10}{\epsilon}-1)z,\frac{10}{\epsilon}z})\\
&\geq \pi(\eta_z\geq h\,,\, \mathcal{A}_{0,z},\, \eta_{2z}\ge h,\, \mathcal{A}_{z,2z},\ldots,
\eta_{\frac{10}{\epsilon}z}\ge h,\,\mathcal{A}_{(\frac{10}{\epsilon}-1)z,\frac{10}{\epsilon}z}\mid
\eta_0\ge h) \pi(\eta_0 \geq h)\\
&\geq
\prod_{j=1}^{10/\epsilon}\pi(\mathcal{A}_{(j-1)z,jz},\ \eta_{j z}\geq
h\mid \eta_{(j-1)z} \geq h)\; \pi(\eta_0 \geq h)\\
&\geq e^{-h/10}\pi(\eta_0 \geq h) \geq e^{-\beta
  (4+\epsilon_\beta)h^2}\,.
\end{align*}
However, another Peierls argument shows that the event $\{\mathcal{A}_{0,z},\mathcal{A}_{z,2z},\ldots,
\mathcal{A}_{(\frac{10}{\epsilon}-1)z,\frac{10}{\epsilon}z}\}$ 
has probability less
that $e^{-(5 \beta - \log 3) h^2}$ 
which yields a contradiction.
\end{proof}
With the above bounds the size and concentration of the maximum height $X_L$
(cf. \eqref{eq:E[XL]Res}) follows from essentially the same proof as
Theorem~\ref{mainthm:max}.  Similarly the height of the surface of the
SOS model with a floor (cf. \eqref{eq-height-concRes}) is given by essentially the same proof as Theorem~\ref{mainthm:floor-shape}.
 \end{proof}

\subsection*{Acknowledgments}
We thank David Wilson for pointing out the correspondence between configurations of edge-disjoint walks on a square and ASMs via the six-vertex model, which allowed us to sharpen our large deviation estimate for the RSOS model. We also thank Marek Biskup, Ron Peled, Yuval Peres and Ofer Zeitouni for useful discussions.



\end{document}